\documentclass[11pt]{amsart}
\usepackage{graphicx}
\usepackage[shortlabels]{enumitem}
\usepackage{amssymb,amsmath,commath}
\usepackage{tikz,tikz-cd,soul,eucal}
\usepackage{hyperref}
\usepackage{todonotes}
\usepackage[nameinlink,capitalise,noabbrev]{cleveref}
\usepackage[margin=1in,marginparwidth=0.8in, marginparsep=0.1in]{geometry}
\usetikzlibrary{calc,arrows}

\numberwithin{equation}{section}

\newtheorem{thm}{Theorem}[section]

\newtheorem{cor}[thm]{Corollary}
\newtheorem{lma}[thm]{Lemma}
\newtheorem{cnj}[thm]{Conjecture}

\theoremstyle{definition}
\newtheorem{dfn}[thm]{Definition}
\newtheorem{ex}[thm]{Example}
\newtheorem{asm}[thm]{Assumption}

\theoremstyle{remark}
\newtheorem{rmk}[thm]{Remark}

\numberwithin{equation}{section}

\newcommand{\R}{{\mathbb{R}}}
\newcommand{\Z}{{\mathbb{Z}}}
\newcommand{\bigast}{\DOTSB\bigop[0.85]{\ast}}

\DeclareMathOperator*{\colim}{colim}
\DeclareMathOperator{\id}{id}
\DeclareMathOperator{\im}{im}
\DeclareMathOperator{\Cord}{Cord}
\DeclareMathOperator{\Skel}{Skel}

\newcommand{\suchthat}{\;\ifnum\currentgrouptype=16 \middle\fi|\;}

\makeatletter
\DeclareRobustCommand\bigop[2][1]{
	\mathop{\vphantom{\sum}\mathpalette\bigop@{{#1}{#2}}}\slimits@
}
\newcommand{\bigop@}[2]{\bigop@@#1#2}
\newcommand{\bigop@@}[3]{
	\vcenter{
		\sbox\z@{$#1\sum$}
		\hbox{\resizebox{\ifx#1\displaystyle#2\fi\dimexpr\ht\z@+\dp\z@}{!}{$\m@th#3$}}
	}
}
\makeatother

\begin{document}

\title{Tangle contact homology}
\author{Johan Asplund}
\address{Department of Mathematics, Stony Brook University, 100 Nicolls Road, Stony Brook, NY 11794}
\email{johan.asplund@stonybrook.edu}

\begin{abstract}
	Knot contact homology is an ambient isotopy invariant of knots and links in $\mathbb R^3$. The purpose of this paper is to extend this definition to an ambient isotopy invariant of tangles and prove that gluing of tangles gives a gluing formula for knot contact homology. As a consequence of the gluing formula we obtain that the tangle contact homology detects triviality of tangles.
\end{abstract}

\maketitle
\tableofcontents
\section{Introduction}
	Knot contact homology is an ambient isotopy invariant of knots and links in $\R^3$ first defined combinatorially by Ng \cite{ng2005knoti,ng2005knotii}. It is known to be isomorphic to the Legendrian contact homology of the unit conormal bundle of the link, with homology coefficients \cite{ekholm2013knot}. It is known that Knot contact homology detects the unlink, cabled knots, composite knots and torus knots, see \cite{gordon2017knot}, \cite[Proposition 5.10]{ng2008framed} and \cite[Corollary 1.5]{cieliebak2017knot}. Moreover, an enhanced version of the knot contact homology is in fact a complete knot invariant \cite{ekholm2018a}. We refer the reader to \cite{ng2014a} and references therein for a complete survey of knot contact homology.

	The unit conormal bundle of a link $K \subset \R^3$ is the set
	\[
		\varLambda_K := \left\{(x,v) \in T^\ast \R^3 \suchthat x \in K, \; |v| = 1, \; \left\langle v, T_x K\right\rangle = 0\right\},
	\]
	for some metric on $\R^3$. The unit cotangent bundle of $\R^3$, denoted by $ST^\ast \R^3$, is a contact manifold when equipped with the one-form $\alpha := \eval[0]{(p_1dx + p_2dy + p_3dz)}_{ST^\ast \R^3}$ where $(x,y,z)$ are local coordinates in $\R^3$ and $(p_1,p_2,p_3)$ are local coordinates in the fiber directions. The unit conormal $\varLambda_K \subset ST^\ast \R^3$ is a Legendrian submanifold, meaning that $\eval[0]{\alpha}_{\varLambda_K} = 0$.

	In this paper we are concerned with the \emph{fully non-commutative} version of knot contact homology of links in $\R^3$ \cite{ekholm2013knot,cieliebak2017knot,ekholm2018a} which for our purposes is defined as the Chekanov--Eliashberg dg-algebra of $\varLambda_K$ in the unit cotangent bundle of $\R^3$ with loop space coefficients
	\[
		KCC^\ast_{\text{EENS}}(K) := CE^\ast(\varLambda_K, C_{-\ast}(\varOmega \varLambda_K)),
	\]
	see \cite{ekholm2017duality} for the definition of Chekanov--Eliashberg dg-algebras with loop space coefficients. The homology of the Chekanov--Eliashberg dg-algebra is a Legendrian isotopy invariant of Legendrian submanifolds in contact manifolds \cite{ekholm2005the,ekholm2007legendrian,ekholm2015legendrian,dimitroglou2016lifting,karlsson2020legendrian}. The differential of the Chekanov--Eliashberg dg-algebra is defined by counting punctured $J$-holomorphic disks in $\R \times ST^\ast \R^3$ with boundary in $\R \times \varLambda_K$. It was first studied independently by Chekanov and Eliashberg \cite{chekanov2002differential,eliashberg1998invariants} and it is part of the more general symplectic field theory package defined by Eliashberg--Givental--Hofer \cite{eliashberg2000introduction}.

	Using methods developed in \cite{asplund2020chekanov} we may understand the Chekanov--Eliashberg dg-algebra with loop space coefficients as the Chekanov--Eliashberg dg-algebra of a cotangent neighborhood of $\varLambda_K$, denoted by $N(\varLambda_K)$, together with a choice of \emph{handle decomposition} $h$ which encodes all Weinstein handles and their attaching maps, see \cref{dfn:handle_decomp_mfd} for details. Using the latter point of view, we define $KCC^\ast(K,h) := CE^\ast((N(\varLambda_K),h);T^\ast \R^3)$ using notation as in \cite{asplund2020chekanov}. For each choice of $h$, $KCC^\ast(K,h)$ is an ambient isotopy invariant of the link $K \subset \R^3$, and for a certain choice of $h$ it recovers $KCC^\ast_{\text{EENS}}(K)$, see \cref{thm:kch_and_kch_eens_single_top_handle}.

\subsection{Statement of results}
	The main construction in this paper is that of \emph{tangle contact homology}. It is the homology of a dg-algebra associated to an oriented tangle $T$ in $\R^3_{x\geq 0}$, and a choice of handle decomposition $h$ of a Weinstein neighborhood of its unit conormal bundle which is denoted by $N(\varLambda_T)$, see \cref{dfn:handle_decomp_pair}. The unit conormal bundle of $T$ is a Legendrian manifold-with-boundary $\varLambda_T \subset T^\ast \R^3_{x\geq 0}$. We show that tangle contact homology is independent of choices and an ambient isotopy invariant of $T$ (with fixed boundary). 

	Suppose that $K \subset \R^3$ is an oriented link and $H \subset \R^3$ is a smooth submanifold that is diffeomorphic to $\R^2$ such that $K$ intersects $H$ transversely. The hypersurface $H$ splits $K$ into two tangles $T_1,T_2 \subset \R^3_{x\geq 0}$, and conversely we say that $K$ is the gluing of $T_1$ and $T_2$.

	In the following, let $X \to Y$ denote an oriented binormal geodesic for $X,Y \in \left\{T_1,T_2,H\right\}$ in the appropriate copy of $\R^3_{x\geq 0}$. Let $\mathcal G_{T_i \to H \to T_j}$ denote the set of words of oriented binormal geodesic chords of the form
	\[
		T_i \to H \to \cdots \to H \to T_j.
	\]

	Let $\left\langle \mathcal S\right\rangle$ denote the free algebra on the set $\mathcal S$, $\ast$ denote free product of algebras, and $\ast_C$ denote amalgamated free product of algebras, see \cref{rmk:category_2} for details. Our main result is the following gluing formula for $KCC^\ast(K,h)$.
	\begin{thm}[\cref{thm:gluing_free_product}]\label{thm:intro_gluing_kch}
			Let $T_1$ and $T_2$ be two tangles in $\R^3_{x\geq 0}$ whose gluing is the link $K \subset \R^3$, where $H := \partial \R^3_{x\geq 0}$. Let $h_1$ and $h_2$ be choices of handle decomposition of $N(\varLambda_{T_1})$ and $N(\varLambda_{T_2})$, respectively. The free algebra
			\[
				\mathcal K := KCC^\ast(T_1,h_1) \ast_{KCC^\ast(\partial T, h_{\partial})} KCC^\ast(T_2,h_2) \ast \left\langle \mathcal G_{T_1 \to H \to T_2}\right\rangle \ast \left\langle \mathcal G_{T_2 \to H \to T_1}\right\rangle \subset KCC^\ast(K,h_1\#h_2),
			\]
			is a canonical subalgebra. Furthermore, fixing a generic choice of metric and almost complex structure we have the following:
			\begin{enumerate}
				\item $\mathcal K$ is dg-algebra when equipped with the same differential as the one in $KCC^\ast(K, h_1 \# h_2)$ such that $KCC^\ast(T_1,h_1), KCC^\ast(T_2,h_2) \subset \mathcal K$ are dg-subalgebras.
				\item There is a quasi-isomorphism of dg-algebras
				\[
					\mathcal K \cong KCC^\ast(K, h_1 \# h_2).
				\]
			\end{enumerate}
	\end{thm}
	\begin{rmk}\label{rmk:intro_main_thm}
		\begin{enumerate}
			\item Only knowing the dg-algebras $KCC^\ast(T_1,h_1)$ and $KCC^\ast(T_2,h_2)$ is not sufficient to recover $KCC^\ast(K,h_1\#h_2)$. To recover $KCC^\ast(K,h_1 \# h_2)$ we also need to know the set of oriented binormal geodesic chords between $T_i$ and $H$ and between $H$ and $T_j$ in the two copies of $\R^3_{x\geq 0}$.
			\item The two algebras $\left\langle \mathcal G_{T_1 \to H \to T_2}\right\rangle$ and $\left\langle \mathcal G_{T_2 \to H \to T_1}\right\rangle$ and their free product by themselves are not dg-algebras when equipped with the differential in $KCC^\ast(K,h_1\# h_2)$, and we do not know any way of expressing $KCC^\ast(K, h_1 \# h_2)$ as the colimit of \emph{dg-algebras} in general.
			\item See \cref{sec:explicit_handle_decomp} and in particular \cref{fig:handle_decomp1,fig:handle_decomp2} for a discussion which handle decompositions to consider in order to recover $KCC^\ast_{\text{EENS}}(K)$.
			\item \Cref{thm:intro_gluing_kch} and every construction in this paper still works if the link $K \subset \R^3$ is split into tangles by any smooth codimension $1$ submanifold of $\R^3$. A particularly important case is that of $S^2$, which is more amenable to concrete calculations (see \cref{sec:ex_unknot}). We have chosen to state results for $\R^3_{\geq 0}$ as it is the most ``symmetric''.
			\item \Cref{thm:intro_gluing_kch} is the simplest gluing formula to state, however, the general framework is significantly more flexible. For instance we can further decompose $\R^3_{x\geq 0}$ by splitting it further by hypersurfaces $\R^3_{x = a}$ for $a > 0$ which allows us to find gluing formulas for tangle contact homology itself, see \cref{thm:gluing_tch}.
		\end{enumerate}
	\end{rmk}
	We have the following geometric characterization of tangle contact homology.
	\begin{thm}[\cref{thm:top_desc_tch}]\label{thm:into_tch_char}
		Let $T$ be a tangle in $\R^3_{x\geq 0}$ where $H := \partial \R^3_{x \geq 0}$. Let $h$ be a choice of handle decomposition of $N(\varLambda_T)$. Let $\mathcal T$ denote the algebra generated by
		\begin{enumerate}
			\item Composable words of oriented binormal geodesic chords in $\R^3_{x\geq 0}$ of the form $T\to T$ or
			\[
				T \to H \to \cdots \to H \to T.
			\]
			\item Oriented binormal geodesic chords $\partial T \to \partial T$ in $H$.
		\end{enumerate}
		Then $\mathcal T \subset KCC^\ast(T,h)$ is a canonical subalgebra. For generic choices of metrics and almost complex structures $\mathcal T$ becomes a dg-algebra when equipped with the same differential as in $KCC^\ast(T,h)$ and there is a quasi-isomorphism
		\[
			\mathcal T \cong KCC^\ast(T,h).
		\]
	\end{thm}
	\begin{rmk}
		The construction of $KCC^\ast(T,h)$ makes sense for tangles of any dimension, and a version of \cref{thm:intro_gluing_kch} still holds in this case.
	\end{rmk}
	\begin{thm}[\cref{thm:untangle_detection}]\label{thm:kcc_detects_untangle_intro}
		Let $h$ be one of the two handle decompositions described in \cref{sec:explicit_handle_decomp}. Then $KCC^\ast(T,h)$ detects the $r$-component trivial tangle.
	\end{thm}
	\Cref{thm:kcc_detects_untangle_intro} is a consequence of \cref{thm:intro_gluing_kch}, and by the property that the fully non-commutative knot contact homology detect the unlink. In fact, we are also utilizing a gluing formula for tangle contact homology itself described in \cref{rmk:intro_main_thm}(5).

\subsection{Construction and method of proof}\label{sec:method}
	We first give a rough description of the definition of tangle contact homology of a tangle $T \subset \R^3_{x\geq 0}$. Taking the unit conormal bundle of $T$ (denoted by $\varLambda_T$) naturally yields a Legendrian submanifold in the contact boundary of the (open) Weinstein sector $T^\ast \R^3_{x\geq 0}$. This sector corresponds to the Weinstein pair $(B^6, T^\ast H)$ where $H := \partial \R^3_{x\geq 0}$ and $\varLambda_T$ is viewed as a Legendrian submanifold in $\partial B^6$ with Legendrian boundary in $\partial T^\ast H$. After picking a handle decomposition $h_T$ of $N(\varLambda_T)$ which is a Weinstein neighborhood of $\varLambda_T$, the dg-algebra $KCC^\ast(T,h_T)$ is defined as the Chekanov--Eliashberg dg-algebra of the pair $(N(\varLambda_T),h_T)$ in $(B^6, T^\ast H)$, see \cref{sec:kch_tangles} for details. This is the natural definition of what the Chekanov--Eliashberg dg-algebra with loop space coefficients would be for a Legendrian with boundary.

	In order to prove \cref{thm:intro_gluing_kch} we generalize the gluing formulas for Chekanov--Eliashberg dg-algebras proven in \cite{asplund2020chekanov,asplund2021simplicial}, to hold for Chekanov--Eliashberg dg-algebras with loop space coefficients. We now give a rough sketch of the gluing formula.

	First consider two Weinstein pairs $(X_1,V)$ and $(X_2,V)$, we glue $X_1$ and $X_2$ together along their common Weinstein hypersurface $V^{2n-2} \hookrightarrow \partial X^{2n}$ to obtain a new Weinstein manifold denoted by $X_1 \#_V X_2$. This operation is called Weinstein connected sum \cite{avdek2021liouville,eliashberg2018weinstein,alvarez2020positive}, and gives gluing formulas for the Chekanov--Eliashberg dg-algebra with field coefficients of the Legendrian attaching link \cite{asplund2020chekanov,asplund2021simplicial}.

	For the generalization of the above, assume we have two Weinstein pairs $(X_1,W_1)$ and $(X_2,W_2)$ where $W_i = V_i^{2n-2} \#_{Q^{2n-4}} (V')^{2n-2}$ for $i\in \left\{1,2\right\}$. We glue $X_1$ and $X_2$ along the common Weinstein subhypersurface $(V')^{2n-2} \subset W_i^{2n-2}$. The result is a new Weinstein pair $(X, V)$, where $X = X_1 \#_{V'} X_2$, and $V = V_1 \#_Q V_2$, see \cref{fig:gluing_pairs_intro}. We show that this type of gluing yields a gluing formula for the Chekanov--Eliashberg dg-algebra of the Weinstein hypersurface $(V,h)$ where $h$ is a handle decomposition of $V$, which for certain choices of $h$ is the Chekanov--Eliashberg dg-algebra with loop space coefficients, see \cref{sec:descent_hypersurfaces} for details.

	\begin{figure}[!htb]
		\centering
		\includegraphics[scale=0.8]{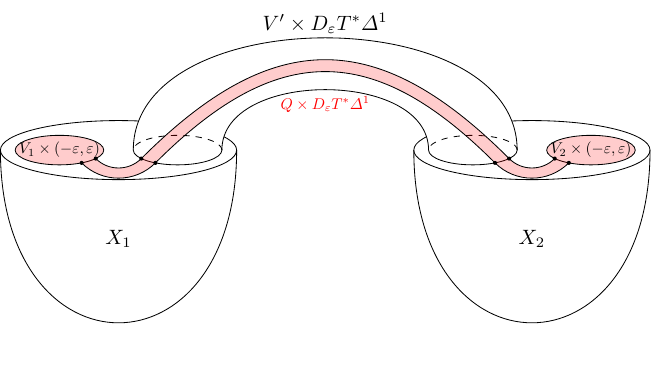}
		\caption{The Weinstein pair $(X_1 \#_{V'} X_2, V_1 \#_Q V_2)$ obtained by gluing together the Weinstein pairs $(X_1,W_1)$ and $(X_2, W_2)$ along their common Weinstein subhypersurface $V' \subset W_1, W_2$.}\label{fig:gluing_pairs_intro}
	\end{figure}

\subsection{Future directions}
	Knot contact homology of (framed oriented) knots in $\R^3$ is closely related to several smooth knot invariants. Some examples include string homology, the cord algebra, the Alexander polynomial and the augmentation polynomial which is closely related to the $A$-polynomial (and also conjectured to be related to a specialization of the HOMFLY-PT polynomial) \cite{ng2005knotii,ng2008framed,cieliebak2017knot,ekholm2018knot}. The knot group and its peripheral subgroup can be extracted from the enhanced version of the knot contact homology \cite{ekholm2018a}.

	In view of \cref{thm:into_tch_char}, it is natural to propose the following conjecture, extending the scope of \cite[Theorem 1.1]{cieliebak2017knot}.
	\begin{cnj}\label{cnj:cord_alg}
		If $T$ is a framed oriented $r$-component tangle in $\R^3_{x\geq 0}$ there exists a choice of handle decomposition $h$ of $N(\varLambda_T)$ such that we have an isomorphism of $\Z[\mu_1^{\pm 1}, \ldots, \mu_r^{\pm 1}]$-algebras
		\[
			KCH^0(T,h) \cong \Cord(T,H).
		\]
	\end{cnj}
		Here $\Cord(T,H)$ denotes the framed cord algebra of $T$, relative to $H$. It is defined as the algebra of cords that are allowed to have none, either, or both endpoints on $H$, modulo homotopy and Skein relations (compare with the usual definition in \cite[Section 4.3]{ng2005knotii} and \cite[Section 2.1]{ng2008framed}).
	\begin{rmk}
		A refinement of \cref{cnj:cord_alg} is that the choice of handle decomposition mentioned should be the one depicted in \cref{fig:handle_decomp2}. There is a certain asymmetry in choices of handle decompositions $h_1$ and $h_2$ to recover $KCC^\ast_{\text{EENS}}(K)$ as described in \cref{sec:explicit_handle_decomp}. Thus we expect to find a gluing formula similar to that of \cref{thm:intro_gluing_kch}, recovering $\Cord(K)$ from $\Cord(T_1,H)$ and $\widetilde{\Cord}(T_2,H)$ where $\widetilde{\Cord}(T_2,H) := KCH^0(T_2,h_2)$ with $h_2$ chosen as in \cref{fig:handle_decomp1}. We also expect there to be a suitable topological interpretation of $\Cord(T_1,H)$ and $\widetilde{\Cord}(T_2,H)$.
	\end{rmk}
	\subsection{Related work}
		Dattin has defined a sutured Legendrian isotopy invariant of sutured Legendrian submanifolds in sutured contact manifolds called cylindrical sutured homology \cite{dattin2022wrapped}. In case the sutured contact manifold is balanced, we expect that the dg-algebra of $\varLambda_T \subset \partial T^\ast \R^3_{x\geq 0}$ defined by Dattin is quasi-isomorphic to $KCC^\ast(T,h)$.

		Let $\varSigma$ be a surface without boundary and let $B \subset \varSigma \times [-1,1]$ be a braid. It is proven that a certain quotient of cylindrical sutured homology of $\varLambda_B \subset \partial T^\ast(\varSigma \times [-1,1])$ together with a product structure gives a complete invariant of braids \cite{dattin2022wrapped,dattinsutured}.
	\subsection*{Outline}
		In \cref{sec:simplicial_decomps_for_pairs} we generalize results from \cite{asplund2021simplicial}. Namely we construct simplicial decompositions of Weinstein pairs, and prove that the Chekanov--Eliashberg dg-algebra of top attaching spheres satisfies a gluing formula. This specializes to gluing formulas for the Chekanov--Eliashberg dg-algebra with loop space coefficients.

		In \cref{sec:tch_and_gluing_formulas} we define tangle contact homology for tangles in $\R^3_{x\geq 0}$. We apply the machinery of \cref{sec:simplicial_decomps_for_pairs} to show that gluing of tangles induces gluing formulas for the tangle contact homologies.

		In \cref{sec:ex_unknot} we compute tangle contact homology in some examples and end with a calculation of the knot contact homology of the unknot via the gluing formula.
	\subsection*{Acknowledgments}
		The author thanks Tobias Ekholm and Côme Dattin for helpful discussions, and Lenhard Ng for his correspondence. This paper grew out as an offshoot from on-going collaboration with William E. Olsen, to whom the author extends a special thanks to for carefully reading earlier drafts of this paper. Finally, the author thanks the anonymous referee whose helpful comments improved the exposition of this paper. The author was supported by the Knut and Alice Wallenberg Foundation.
\section{Simplicial decompositions for Weinstein pairs}\label{sec:simplicial_decomps_for_pairs}
	In this section we generalize the notion of a simplicial decomposition of a Weinstein manifold introduced in \cite{asplund2021simplicial} to Weinstein pairs. We first review some basic definitions we use throughout the paper.

	A \emph{Liouville domain} is a pair $(X^{2n},\lambda)$ of a smooth $2n$-dimensional manifold-with-boundary $X$ and a one-form $\lambda$ such that $\omega := d \lambda$ is symplectic and the $\omega$-dual of $\lambda$ (called the \emph{Liouville vector field}) is outwards pointing along $\partial X$. A \emph{Liouville manifold} is an exact symplectic manifold $(X, \lambda)$ which admits an exhaustion by successively larger Liouville domains, such that the Liouville vector field is complete. The \emph{skeleton} of a Liouville manifold, denoted by $\Skel X$, is the subset which does not escape every successive exhausting Liouville domain, under the flow of the Liouville vector field. A \emph{Weinstein manifold} is a Liouville manifold for which the Liouville vector field is gradient-like for a Morse function on $X$. It is well-known that any Weinstein manifold can equivalently be obtained by successive attachments of standard Weinstein handles of index $k \in \left\{0,\ldots,n\right\}$, see \cite{weinstein1991contact}. A Weinstein manifold $X$ is called \emph{subcritical} if it can be built using no Weinstein handles of index $n$.

	\begin{dfn}[Weinstein hypersurface]
		Let $(X^{2n},\lambda)$ be a $2n$-dimensional Weinstein manifold. A \emph{Weinstein hypersurface} consists of a $(2n-2)$-dimensional Weinstein manifold $(V^{2n-2},\lambda_V)$ together with a Weinstein embedding $(V,\lambda_V) \hookrightarrow (X \setminus \Skel X, \lambda)$ such that the induced map $V \to \partial X$ is an embedding. We will use the shorthand $V \hookrightarrow \partial X$ to denote a Weinstein hypersurface.
	\end{dfn}
	\begin{dfn}[Weinstein pair]
		A \emph{Weinstein pair} is a tuple $(X, V \hookrightarrow \partial X)$ of a Weinstein manifold together with a Weinstein hypersurface.
	\end{dfn}
	\begin{dfn}[Liouville sector {\cite[Definition 1.1]{ganatra2020covariantly}}]
		A Liouville sector is a Liouville manifold-with-boundary $(X, \lambda, Z)$ for which there is a function $I \colon \partial X \to \R$ such that:
		\begin{itemize}
			\item $I$ is \emph{linear at infinity}, meaning $ZI = I$ outside a compact set, where $Z$ denotes the Liouville vector field.
			\item The Hamiltonian vector field $X_I$ of $I$ is outward pointing along $\partial X$.
		\end{itemize}
	\end{dfn}
	For every Liouville sector $X$, one can modify the Liouville form to obtain a Liouville pair $(\overline X, F)$ called the \emph{convexification} of $X$, see \cite[Section 2.7]{ganatra2020covariantly}. Moreover, up to a contractible choice, there is a one-to-one correspondence between Liouville sectors and Liouville pairs \cite[Lemma 2.32]{ganatra2020covariantly}.
	\begin{dfn}[Weinstein sector]
		A \emph{Weinstein sector} is a Liouville sector $X$ such that its convexification $(\overline X, F)$ is a Weinstein pair up to deformation.
	\end{dfn}
	We will use the notions of Weinstein sector and Weinstein pair interchangeably throughout this paper.
	\begin{dfn}[Handle decomposition of a Weinstein manifold]\label{dfn:handle_decomp_mfd}
		Let $X^{2n}$ be a $2n$-dimensional Weinstein manifold. By a \emph{handle decomposition} of $X$, we mean a collection of tuples $(X_k,\varSigma_k)$ for $k\in \left\{0,\ldots,n-1\right\}$ where 
		\begin{itemize}
			\item $X_k^{2n}$ is a $2n$-dimensional subcritical Weinstein manifold that only has handles of index $< k$.
			\item There are Weinstein embeddings $X_0^{2n} \subset X_1^{2n} \subset \cdots \subset X_{n-1}^{2n} \subset X^{2n}$.
			\item $\varSigma_k \subset \partial X_k$ is a disjoint union of $k$-dimensional isotropic spheres.
		\end{itemize}
	\end{dfn}
	\begin{rmk}\label{rmk:handle_decomp}
		\begin{enumerate}
			\item The isotropic submanifolds $\varSigma_k$ are the attaching spheres for the $k$-dimensional Weinstein handles used to construct $X$, and the embedding $\varSigma_k \hookrightarrow \partial X_k$ canonically determines the index $k+1$ Weinstein handles of $X$. By abuse of notation, we may refer to (the embedding of) $\varSigma_k$ as \emph{the} index $k+1$ Weinstein handles of $X$.
			\item For $k = n-1$, we have that $\varSigma_{n-1} \subset \partial X_{n-1}$ are Legendrian submanifolds, and determine the \emph{top} (or \emph{critical}) Weinstein handles of $X$.
			\item Every construction in this paper only depends on the tuple $X_0 := (X_{n-1},\varSigma_{n-1})$ which we call the \emph{subcritical part of $X$}. This notation and terminology will not cause any confusion since the Weinstein manifold $X_k^{2n}$ for $k < n-1$ will play no role in this paper.
		\end{enumerate}
	\end{rmk}
	\begin{dfn}[Handle decomposition of a Weinstein pair]\label{dfn:handle_decomp_pair}
		Let $(X^{2n},V^{2n-2})$ be a Weinstein pair. A handle decomposition of $(X,V)$ consists of a handle decomposition $\{(X_k,\varSigma_k)\}_{k=0}^{n-1}$ of $X$ and a handle decomposition $\{(V_k,\sigma_k)\}_{k=0}^{n-2}$ of $V$.
	\end{dfn}
	\begin{rmk}
		The Weinstein embedding $V \hookrightarrow \partial X$ induces Weinstein embeddings $V_k \hookrightarrow \partial X$ for each $k \in \left\{0,\ldots,n-2\right\}$, and we do not require any further compatibilities between the two handle decompositions of $X$ and $V$ in \cref{dfn:handle_decomp_pair}.
	\end{rmk}
	By following \cite[Section 2.3]{chantraine2017geometric} we can also define Weinstein sectors as the result of consecutive attachments of standard Weinstein handles and Weinstein half-handles. The standard model for Weinstein half-handles is revealed by studying Morse theory on manifolds with boundary, see \cite{borodzik2016morse} and \cite[Section 2.4]{kronheimer2007monopoles}. The boundary at infinity of a Weinstein sector is a contact manifold with boundary and can be made into a sutured contact manifold, see \cite[Section 2.4]{dattin2022sutured} and \cite[Section 4.2]{colin2011sutures}. From this point of view, the attaching locus of a Weinstein half-handle is a Legendrian disk with Legendrian boundary in the suture of the sutured contact manifold. The procedure of convexification will turn this Legendrian disk into the core disk of a Weinstein handle of the Weinstein hypersurface. The embedding of each attaching locus of the Weinstein half-handles together determine the Weinstein hypersurface.
	\begin{ex}\label{ex:handle_decomp_sector}
		If $Q$ is a compact smooth manifold-with-boundary, then $T^\ast Q$ is a Weinstein sector. If we pick a Morse function $f$ on $Q$ that restricts to a Morse function along $\partial Q$, then $f$ determines a handle decomposition of $Q$ as a manifold-with-boundary (see \cite[Sections 2.1 and 2.2]{borodzik2016morse}). Such a handle decomposition of $Q$ induces a Weinstein handle decomposition of $T^\ast Q$ as described above.
	\end{ex}

	\subsection{Construction of simplicial decompositions for Weinstein pairs}\label{sec:cns_simplicial_decomp_pairs}
		We assume in the following that the reader is familiar with the construction of a simplicial decomposition of a Weinstein manifold as defined in \cite[Section 2.2]{asplund2021simplicial}. 
		\begin{rmk}
			Recall that a Weinstein pair $(X^{2n},P^{2n-2})$ corresponds to a Weinstein sector $X'$ via convex completion \cite[Section 2.7]{ganatra2020covariantly}. We may thus think about constructions in this section as a generalization of the simplicial decomposition of a Weinstein manifold as constructed in \cite{asplund2021simplicial} to a simplicial decomposition of a Weinstein sector.
		\end{rmk}
		In the following the superscript $X^{(k)}$ is used to indicate that $X$ has codimension $2k$ for $k \in \Z_{\geq 0}$. We give a brief recap of the definition of a simplicial decomposition and refer to \cite[Section 2.2]{asplund2021simplicial} for details. A simplicial decomposition of a Weinstein manifold $X$ is a tuple $(C, \boldsymbol V)$ where $C$ is a simplicial complex and $\boldsymbol V$ is a set containing one Weinstein manifold $(V_{\sigma_k}^{(k)},\lambda_{\sigma_k})$ for each $k$-face $\sigma_k\in C$ (ranging over all $k$) and a certain Weinstein hypersurface associated to each $k$-face $\sigma_k\in C$ (ranging over all $k$). Associated to each $V_{\sigma_k}^{(k)}$ is the ``basic building block''
		\[
			\left(V_{\sigma_k}^{(k)} \times T^\ast \varDelta^k, \lambda_{\sigma_k} + \sum_{i=1}^k (2x_i dy_i + y_i dx_i)\right),
		\]
		where $(y_1,\ldots,y_k)$ are coordinates along $\varDelta^k$ and $(x_1,\ldots,x_k)$ are coordinates in the fiber directions of $T^\ast \varDelta^k$. The basic building block is a Weinstein cobordism with negative end $V \times T^\ast \varDelta^k|_{\partial \varDelta^k}$, see \cite[Section 2.1]{asplund2021simplicial} for more details. We define $\# \boldsymbol V$ to be the \emph{gluing} of all the building blocks $V_{\sigma_k}^{(k)} \times T^\ast \varDelta^k$ using certain gluing maps induced by the Weinstein hypersurfaces in $\boldsymbol V$, see \cite[Section 2.2.1]{asplund2021simplicial}. The condition required for $(C, \boldsymbol V)$ to be a simplicial decomposition of $X$ is that there is a Weinstein isomorphism $X \cong \# \boldsymbol V$, meaning a strict symplectomorphism that also preserves the Weinstein Morse functions.

		\begin{dfn}\label{dfn:V_supset}
			Let $(C, \boldsymbol V)$ be a simplicial decomposition. For a $k$-face $\sigma_k\in C$ we define $\boldsymbol V_{\supsetneq \sigma_k} \subset \boldsymbol V$ (and $\boldsymbol V_{\supset \sigma_k} \subset \boldsymbol V$) to be the subset consisting of only those Weinstein manifolds $V_{\sigma_i}\in \boldsymbol V$ for which $\sigma_i \supsetneq \sigma_k$ (and $\sigma_i \supset \sigma_k$), and the corresponding Weinstein hypersurfaces.
		\end{dfn}
		\begin{dfn}[Hypersurface inclusion of simplicial decompositions]\label{dfn:hypersurface_inclusion}
			Let $(X,P)$ be a Weinstein pair. Let $(C, \boldsymbol V)$ be a simplicial decomposition of $X^{2n}$ and let $(C', \boldsymbol V')$ be a simplicial decomposition of $P^{2n-2}$. A \emph{hypersurface inclusion of simplicial decompositions} $(C', \boldsymbol V')\hookrightarrow (C, \boldsymbol V)$ consists of
			\begin{enumerate}
				\item A simplicial subcomplex $i \colon C' \hookrightarrow C$. When no confusion can arise we use the notation $\sigma_k := i(\sigma'_k)$ for each $\sigma_k' \in C'$.
				\item A Weinstein hypersurface 
				\[
					\#(\boldsymbol V'_{\supset \sigma_k'} \sqcup_i \boldsymbol V_{\supsetneq \sigma_k}) \hookrightarrow \partial V_{\sigma_k},
				\]
				for each $\sigma_k' \in C'$, where
				\[
					\boldsymbol V'_{\supset \sigma_k'} \sqcup_i \boldsymbol V_{\supsetneq \sigma_k} := \boldsymbol V'_{\supset \sigma_k'} \cup \boldsymbol V_{\supsetneq \sigma_k} \cup \bigcup_{\sigma_j' \supset \sigma_k'}\{V'_{\sigma_j'} \hookrightarrow \partial V_{\sigma_j}\},
				\]
				see \cref{dfn:V_supset} and \cite[Section 2.2]{asplund2021simplicial} for the notation.
			\end{enumerate}
		\end{dfn}
		\begin{dfn}[Simplicial decomposition of a Weinstein pair]\label{dfn:simplicial_decomposition_pair}
			A simplicial decomposition of the Weinstein pair $(X,P)$ which is denoted by $((C,C'),(\boldsymbol V, \boldsymbol V'))$ consists of
			\begin{itemize}
				\item A simplicial decomposition $(C, \boldsymbol V)$ of $X$.
				\item A simplicial decomposition $(C', \boldsymbol V')$ of $P$.
				\item A hypersurface inclusion of simplicial decompositions $(C', \boldsymbol V') \hookrightarrow (C, \boldsymbol V)$.
			\end{itemize}
			such that 
			\[
				X \cong \# \boldsymbol V, \quad P \cong \# \boldsymbol V'.
			\]
		\end{dfn}
		We now describe simplicial decompositions of Weinstein pairs in simple examples. These also serve as examples of simplicial decompositions of Weinstein manifolds by setting $C' = \varnothing$ in each of the following examples.
		\begin{description}
			\item[$C = \varDelta^1$] A simplicial decomposition of $X$ over the $1$-simplex is a tuple $(\varDelta^1, \boldsymbol V)$ where $\boldsymbol V := \left\{V^{2n-2},X^{2n}_1,X^{2n}_2, \iota_1 \colon V \hookrightarrow \partial X_1, \iota_2 \colon V \hookrightarrow \partial X_2\right\}$. Here $V$ is a Weinstein $(2n-2)$-manifold corresponding to the single edge of $\varDelta^1$, $X_1^{2n}$ and $X_2^{2n}$ are two Weinstein $2n$-manifolds corresponding to each vertex of $\varDelta^1$ and $\iota_1, \iota_2$ are two Weinstein hypersurfaces corresponding to the face inclusion of each vertex in the single edge of the $1$-simplex. 
			
			Using the tuple $(\varDelta^1, \boldsymbol V)$ we construct a Weinstein pair through the following surgery presentation. The basic building block associated to $V$ is the Weinstein cobordism $(V \times D_\varepsilon T^\ast \varDelta^1, \lambda_V + 2xdy+ydx)$ where $D_\varepsilon T^\ast \varDelta^1$ is the $\varepsilon$-disk cotangent bundle of $\varDelta^1$, $y$ is a coordinate in the $\varDelta^1$-factor and $x$ is a coordinate in the fiber direction. The negative end of this cobordism is $V \times D_\varepsilon T^\ast(\eval[0]{\varDelta^1}_{\partial \varDelta^1}) = (V \times (-\varepsilon,\varepsilon)) \sqcup (V \times (-\varepsilon,\varepsilon))$. Using the Weinstein hypersurfaces $\iota_1,\iota_2 \in \boldsymbol V$ we attach the Weinstein cobordism $V \times D_\varepsilon T^\ast \varDelta^1$ to $X_1 \sqcup X_2$ along $(V \times (-\varepsilon,\varepsilon))\sqcup (V \times (-\varepsilon,\varepsilon))$ and denote the resulting Weinstein manifold by $\# \boldsymbol V$, or $X_1 \#_V X_2$ using more common notation, as this is nothing but the Weinstein connected sum of $X_1$ and $X_2$ over $V$ \cite{avdek2021liouville,alvarez2020positive,eliashberg2018weinstein}.

			We require $C'$ to be a simplicial subcomplex of $C$ and consider the two cases $C' = \varDelta^0, \varDelta^1$.
			\begin{description}
				\item[$C' = \varDelta^0$]
					In this case we set $\boldsymbol V' = \left\{P^{2n-2}\right\}$. Via \cref{dfn:hypersurface_inclusion} we now describe a hypersurface inclusion $(\varDelta^0, \boldsymbol V') \hookrightarrow (\varDelta^1, \boldsymbol V)$. It consists of an inclusion $\varDelta^0 \hookrightarrow \varDelta^1$ as a simplicial subcomplex. We include $\varDelta^0$ into the first vertex, which corresponds to $X_1$. Such a hypersurface inclusion additionally consists of a Weinstein hypersurface $P \sqcup V \hookrightarrow \partial X_1$, because we have $\boldsymbol V_{\supsetneq \sigma_0} = \{V\}$, see \cref{dfn:hypersurface_inclusion}(2).

					Such a hypersurface inclusion $(\varDelta^0, \boldsymbol V') \hookrightarrow (\varDelta^1, \boldsymbol V)$ now allows us to construct the Weinstein pair $(X_1\#_V X_2, P)$, by simply taking the Weinstein connected sum of $X_1$ and $X_2$ along $V$. The Weinstein hypersurface $P$ being disjoint from $V$ in $\partial X_1$ ensures that $P$ is a Weinstein hypersurface in $X_1 \#_V X_2$, see \cref{fig:surgery_pres_pair}.
					\begin{figure}[!htb]
						\centering
						\includegraphics[scale=0.75]{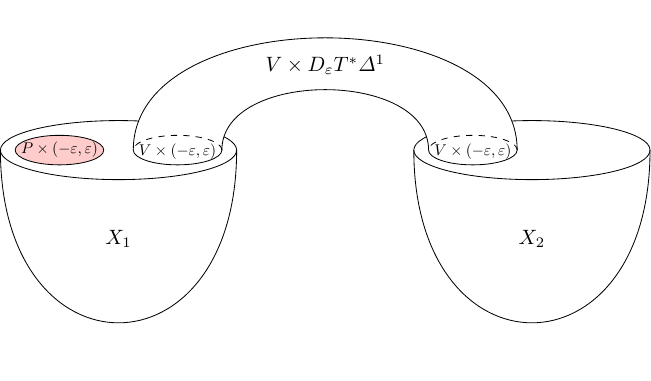}
						\hspace{2cm}
						\raisebox{10ex}{\includegraphics{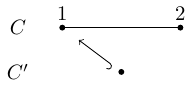}}
						\caption{Left: The surgery presentation of the Weinstein pair $(X_1\#_V X_2, P)$. Right: The simplicial subcomplex $C' \hookrightarrow C$.}\label{fig:surgery_pres_pair}
					\end{figure}
				\item[$C' = \varDelta^1$]
					In this case $\boldsymbol V' = \left\{Q^{2n-4},P_1^{2n-2}, P_2^{2n-2}, \iota_{Q,1} \colon Q \hookrightarrow \partial P_1^{2n-2}, \iota_{Q,2} \colon Q \hookrightarrow \partial P_2^{2n-2}\right\}$ where $Q$ is a Weinstein $(2n-4)$-manifold, $P_1$ and $P_2$ are Weinstein $(2n-2)$-manifolds and $\iota_{Q,1}$ and $\iota_{Q,2}$ are Weinstein hypersurfaces. This gives the surgery presentation $P = P_1 \#_Q P_2$ as described above.

					Pick the hypersurface inclusion $(\varDelta^1, \boldsymbol V') \hookrightarrow (\varDelta^1, \boldsymbol V)$ such that $\varDelta^1 \hookrightarrow \varDelta^1$ is the identity map. By the definition \cref{dfn:hypersurface_inclusion} the hypersurface inclusion additionally consists of the following.
					\begin{enumerate}
						\item A Weinstein hypersurface $Q^{2n-4} \hookrightarrow \partial V^{2n-2}$.
						\item Two Weinstein hypersurfaces $P_i \#_Q V \hookrightarrow \partial X_i$ for $i\in \left\{1,2\right\}$.
					\end{enumerate}
					The Weinstein hypersurface $Q \hookrightarrow \partial V$ extends to a Weinstein hypersurface $Q \times D_\varepsilon T^\ast \varDelta^1 \hookrightarrow \partial_\infty(V \times D_\varepsilon T^\ast \varDelta^1)$ which is glued together with each of the negative ends of $Q \times D_\varepsilon T^\ast \varDelta^1 \hookrightarrow \partial X_i$ along $\partial V$ for $i\in \left\{1,2\right\}$. The result is a Weinstein hypersurface $P_1 \#_Q P_2 \hookrightarrow \partial(X_1 \#_V X_2)$, see \cref{fig:surgery_pres_pair_2}.
					\begin{figure}[!htb]
						\centering
						\includegraphics[scale=0.75]{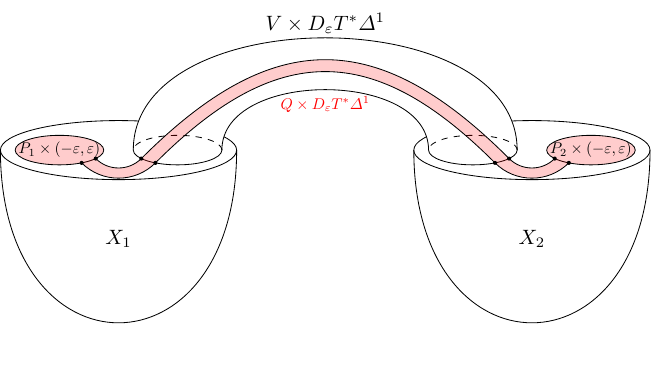}
						\hspace{2cm}
						\raisebox{10ex}{\includegraphics{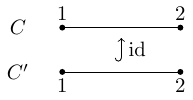}}
						\caption{Left: The surgery presentation of the Weinstein pair $(X_1\#_V X_2, \color{red}{P_1\#_Q P_2}\color{black}{)}$. Right: The simplicial subcomplex $C' \hookrightarrow C$.}\label{fig:surgery_pres_pair_2}
					\end{figure}
			\end{description}
			\item[$C = \varDelta^2$]
				We consider a simplicial decomposition $(\varDelta^2, \boldsymbol V)$ of $X$ as follows. Let
				\[
					\boldsymbol V := \{V^{2n-2}_1,V^{2n-2}_2,V^{2n-2}_3,W^{2n-4},X^{2n}_1,X^{2n}_2,X^{2n}_3, \iota_{12}, \iota_{23}, \iota_{13}, \iota_1, \iota_2, \iota_3\},
				\]
				where 
				\begin{itemize}
					\item $X_1$, $X_2$ and $X_3$ are Weinstein $2n$-manifolds corresponding to the vertices of $\varDelta^2$.
					\item $V_1$, $V_2$ and $V_3$ are Weinstein $(2n-2)$-manifolds corresponding to the edges of $\varDelta^2$.
					\item $W$ is a Weinstein $(2n-4)$-manifold corresponding to the $2$-face of $\varDelta^2$.
					\item $\iota_i \colon W \hookrightarrow \partial V_i$ is a Weinstein hypersurface for $i\in \left\{1,2,3\right\}$ corresponding to the inclusions of the edges into the $2$-simplex.
					\item $\iota_{12} \colon V_1 \#_W V_2 \hookrightarrow \partial X_3$, $\iota_{23} \colon V_2 \#_W V_3 \hookrightarrow \partial X_1$ and $\iota_{13} \colon V_3 \#_W V_1 \hookrightarrow \partial X_2$ are Weinstein hypersurfaces corresponding to the inclusion of the vertices into the $2$-simplex.
				\end{itemize}
				We describe a simplicial decomposition of the Weinstein pair $(X,P)$ in the two cases $C' = \varDelta^1$ and $C' = \varDelta^2$ below.
				\begin{description}
					\item[$C' = \varDelta^1$] As above we have a simplicial decomposition over the $1$-simplex $(\varDelta^1, \boldsymbol V')$ of $P$ so that $P \cong P_1\#_Q P_2$. We choose the hypersurface inclusion $(\varDelta^1, \boldsymbol V') \hookrightarrow (\varDelta^2, \boldsymbol V)$ that is given by the following.
					\begin{enumerate}
						\item An inclusion $\varDelta^1$ as a simplicial subcomplex of $\varDelta^2$ as the edge corresponding to $V_1$.
						\item A Weinstein hypersurface $Q^{2n-4} \hookrightarrow \partial V_1^{2n-2}$.
						\item Two Weinstein hypersurfaces $P_1 \#_Q V_1 \hookrightarrow \partial X_2$ and $P_2 \#_Q V_1 \hookrightarrow \partial X_3$.
					\end{enumerate}
					As before the Weinstein hypersurface $Q \hookrightarrow \partial V$ extends to a Weinstein hypersurface $Q \times D_\varepsilon T^\ast \varDelta^1 \hookrightarrow \partial_\infty(V \times D_\varepsilon T^\ast \varDelta^1)$ which results in a Weinstein hypersurface $P_1 \#_Q P_2 \hookrightarrow \partial X$, see \cref{fig:surgery_pres_pair_3}
					\begin{figure}[!htb]
						\centering
						\includegraphics[scale=0.75]{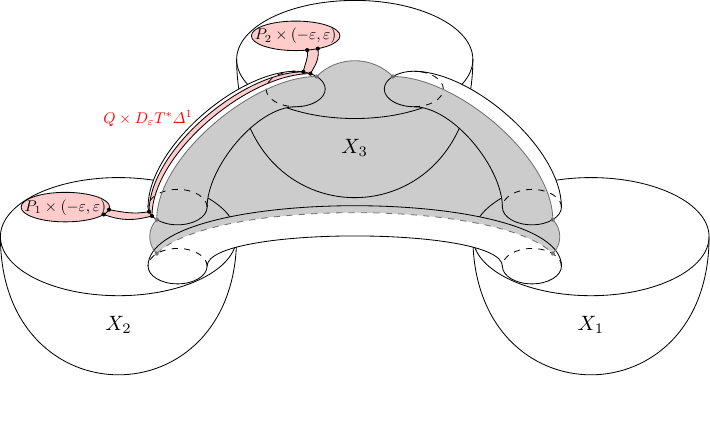}
						\hspace{2cm}
						\raisebox{10ex}{\includegraphics{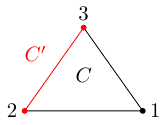}}
						\caption{Left: The surgery presentation of the Weinstein pair $(\# \boldsymbol V, P_1\#_Q P_2)$. Right: The simplicial subcomplex ${\color{red}{C'}} \hookrightarrow C$ included as the edge in $C = \varDelta^2$ connecting the vertices $2$ and $3$.}\label{fig:surgery_pres_pair_3}
					\end{figure}
					\item[$C' = \varDelta^2$] Pick a simplicial decomposition $(\varDelta^2, \boldsymbol V')$ of $P$ as follows. Let
					\[
						\boldsymbol V' := \left\{Q_1^{2n-4}, Q_2^{2n-4}, Q_3^{2n-4}, R^{2n-6}, P^{2n-2}_1,P^{2n-2}_2,P^{2n-2}_3, j_{12}, j_{23}, j_{13}, j_1, j_2, j_3\right\},
					\]
					where
					\begin{itemize}
						\item $P_1$, $P_2$ and $P_3$ are Weinstein $(2n-2)$-manifolds corresponding to the vertices of $\varDelta^2$.
						\item $Q_1$, $Q_2$ and $Q_3$ are Weinstein $(2n-4)$-manifolds corresponding to the edges of $\varDelta^2$.
						\item $R$ is a Weinstein $(2n-6)$-manifold corresponding to the $2$-face of $\varDelta^2$.
						\item $j_i \colon R \hookrightarrow \partial Q_i$ is a Weinstein hypersurface for $i\in \left\{1,2,3\right\}$ corresponding to the inclusions of the edges into the $2$-simplex.
						\item $j_{12} \colon Q_1 \#_R Q_2 \hookrightarrow \partial P_3$, $j_{23} \colon Q_2 \#_R Q_3 \hookrightarrow \partial P_1$ and $j_{13} \colon Q_3 \#_R Q_1 \hookrightarrow \partial P_2$ are Weinstein hypersurfaces corresponding to the inclusion of the vertices into the $2$-simplex.
					\end{itemize}
					We choose the hypersurface inclusion $(\varDelta^2, \boldsymbol V') \hookrightarrow (\varDelta^2, \boldsymbol V)$ that is given by the following.
					\begin{enumerate}
						\item The identity map $\varDelta^2 \hookrightarrow \varDelta^2$.
						\item Weinstein hypersurfaces $Q_i \hookrightarrow \partial V_i$ for $i\in \left\{1,2,3\right\}$ whose images are disjoint from the images of $\iota_i \in \boldsymbol V$.
						\item A Weinstein hypersurface $\iota_R \colon R \hookrightarrow \partial W$.
						\item Weinstein hypersurfaces $\iota_{i,R} \colon Q_i \#_R W \hookrightarrow \partial V_i$ for $i\in \left\{1,2,3\right\}$.
						\item Weinstein hypersurfaces $\# \widetilde{\boldsymbol V}_i \hookrightarrow \partial X_i$ for $i\in \left\{1,2,3\right\}$ where
						\begin{align*}
							\widetilde{\boldsymbol V}_i &:= \left\{P_i,R,W,\iota_R\right\} \\
							&\qquad \cup \left\{Q_k,V_k,j_k, \iota_{k,R}\right\}_{k \in \left\{1,2,3\right\} \setminus \left\{i\right\}} \cup \left\{j_{k \ell}\right\}_{\left\{k,\ell\right\} = \left\{1,2,3\right\} \setminus \left\{i\right\}}
						\end{align*}
						see \cref{fig:simplicial_decomp_pair_22_vertex}.
					\end{enumerate}
					\begin{figure}[!htb]
						\centering
						\includegraphics[scale=0.8]{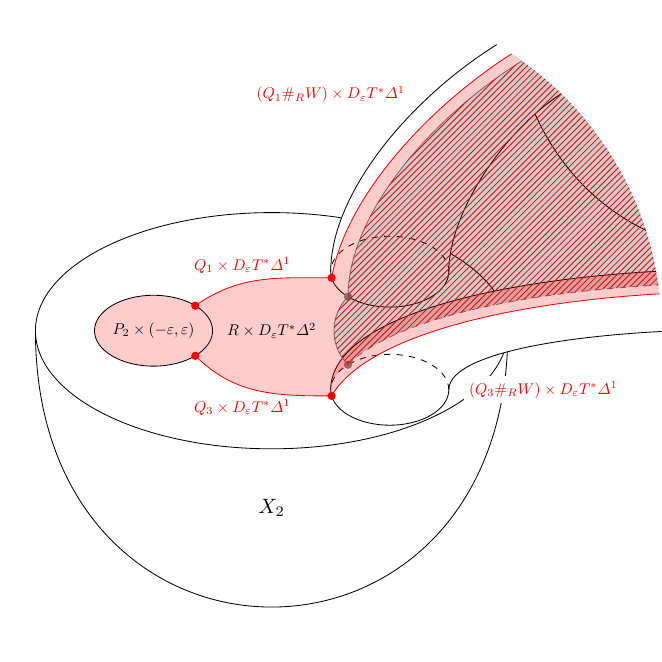}
						\caption{Part of the Weinstein hypersurface $P = \# \boldsymbol V'$ lying in $\partial X_2$.}\label{fig:simplicial_decomp_pair_22_vertex}
					\end{figure}
					Each Weinstein hypersurface $\iota_{k,R}$ extends over the $1$-simplex handles $V_k \times D_\varepsilon T^\ast \varDelta^1$ to a Weinstein hypersurface 
					\[
						(Q_k \#_R W) \times D_\varepsilon T^\ast \varDelta^1 \hookrightarrow \partial_\infty(V_k \times  D_\varepsilon T^\ast \varDelta^1),
					\]
					and the Weinstein hypersurface $\iota_R \colon R \hookrightarrow \partial W$ extends to a Weinstein hypersurface $R \times D_\varepsilon T^\ast \varDelta^2 \hookrightarrow \partial_\infty(W \times D_\varepsilon T^\ast \varDelta^2)$. These Weinstein hypersurfaces all glue together the same way as $2$-simplex handles are defined, see \cite[Section 2.2]{asplund2021simplicial}. The result is a Weinstein hypersurface $\# \boldsymbol V' \hookrightarrow \partial \left(\# \boldsymbol V\right)$, see \cref{fig:simplicial_decomp_pair_22}.
					\begin{figure}[!htb]
						\centering
						\includegraphics[scale=0.8]{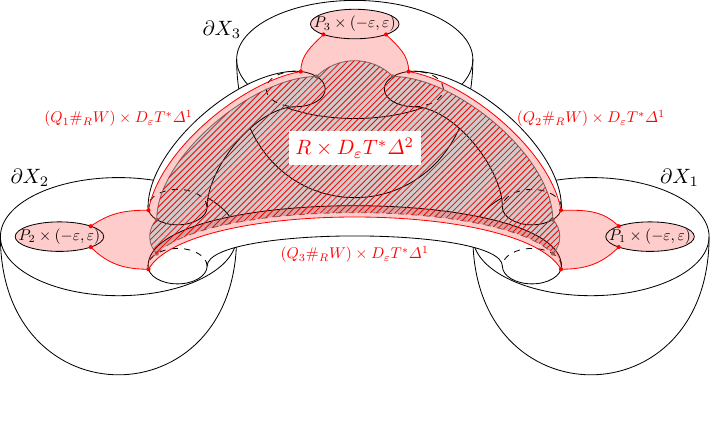}
						\caption{The Weinstein hypersurface $\# \boldsymbol V' \hookrightarrow \partial \# \boldsymbol V$ depicted in red.}\label{fig:simplicial_decomp_pair_22}
					\end{figure}
				\end{description}
		\end{description}
		\begin{rmk}
			From the point of view of Weinstein sectors, part of the the surgery description presented above corresponds to partial gluing of two Weinstein sectors along a shared subsector in the boundary, which was already described in \cite[Construction 12.18]{ganatra2022sectorial}.
		\end{rmk}
			\subsubsection{Good sectorial covers for sectors}
				The notion of a good sectorial cover of a Weinstein manifold introduced in \cite{asplund2021simplicial} now has a natural extension to Weinstein sectors. First recall the definition of a sectorial cover of a Liouville sector.
				\begin{dfn}[Sectorial cover {\cite[Definition 12.2 and Definition 12.19]{ganatra2022sectorial}}]
					Let $X$ be a Liouville sector. Suppose $X = X_1 \cup \cdots \cup X_n$, where each $X_i$ is a manifold-with-corners with precisely two faces $\partial^1 X_i := X_i \cap \partial X$ and the point set topological boundary $\partial^2 X_i$ of $X_i \subset X$, meeting along the corner locus $\partial X \cap \partial^2 X_i = \partial^1 X_i \cap \partial^2 X_i$. Such a covering $X = X_1 \cup \cdots \cup X_n$ is called \emph{sectorial} if and only if $\forall i \in \left\{1,\ldots, n\right\}$ there are functions $I_i \colon N^Z(\partial^2 X_i) \longrightarrow \R$ (where $N^Z$ denotes a neighborhood which is cylindrical with respect to the Liouville vector field $Z$) which is linear at infinity such that:
					\begin{itemize}
						\item $X_{I_i}$ is outward pointing along $\partial^2 X_i$.
						\item $X_{I_i}$ is tangent to $\partial^2 X_j$ along $\partial^2 X_i \cap \partial^2 X_j$ for $i\neq j$.
						\item $[X_{I_i},X_{I_j}] = 0$ along $N^Z(\partial^2 X_i) \cap N^Z(\partial^2 X_j)$.
					\end{itemize}
				\end{dfn}
				\begin{lma}[{\cite[Lemma 12.11]{ganatra2022sectorial}}]
					Let $X$ be a Liouville sector and suppose $X = X_1 \cup \cdots \cup X_n$ is a sectorial cover. For any $\varnothing \neq A \subset \left\{1,\ldots,n\right\}$ we have a Liouville isomorphism
					\begin{equation}\label{eq:coords_nghd_bdry}
						N^Z \left(\bigcap_{i\in A} \partial^2 X_i\right) \cong (X_A^{(k)} \times T^\ast \R^k, \lambda_{X_A^{(k)}} + \lambda_{T^\ast \R^k} + df),
					\end{equation}
					where $k := \abs A - 1$,  $X^{(k)}_{A}$ is a $(2n-2k)$-dimensional Weinstein manifold and $f$ is a real-valued function on $X^{(k)}_A \times T^\ast \R^k$ with support in $K \times T^\ast \R^k$ for some compact $K \subset X^{(k)}_{A}$.
				\end{lma}
				\begin{dfn}[Good sectorial cover]\label{dfn:good_sectorial_cover}
					Let $X$ be a Weinstein sector, and suppose $X=X_1 \cup \cdots \cup X_m$ is a sectorial cover. We say that the sectorial cover is \emph{good} if for every $\varnothing \neq	A \subset \{1,\ldots,m\}$ we have a Weinstein isomorphism
					\begin{equation}\label{eq:coordinates_intersections}
						N^Z\left(\bigcap_{i\in A} X_i\right) \cong (X^{(k)}_{A} \times T^\ast \R^k,\lambda_{X^{(k)}_{A}} + \lambda_{T^\ast \R^k} + df),
					\end{equation}
					extending the Weinstein isomorphism \eqref{eq:coords_nghd_bdry}, where $k := \abs A - 1$,  $X^{(k)}_{A}$ is a $(2n-2k)$-dimensional Weinstein sector and $f$ is a real-valued function on $X^{(k)}_A \times T^\ast \R^k$ with support in $K \times T^\ast \R^k$ for some compact $K \subset X^{(k)}_{A}$.
				\end{dfn}
				\begin{dfn}[Simplicial decomposition of a Weinstein sector]
					A simplicial decomposition of a  Weinstein sector $X$ is defined to be a simplicial decomposition of the Weinstein pair $(\widetilde X, P)$ which corresponds to $X$ via convex completion.
				\end{dfn}
				\begin{thm}
					Let $X$ be a Weinstein sector. There is a one-to-one correspondence (up to Weinstein homotopy) between good sectorial covers of $X$ and simplicial decompositions of $X$.
				\end{thm}
				\begin{proof}
					This is similar and a slight generalization of \cite[Theorem 3.11]{asplund2021simplicial}.

					Given a simplicial decomposition $((C,C'),(\boldsymbol V, \boldsymbol V'))$ of the Weinstein pair $(\widetilde X, P)$, it follows from \cite[Theorem 3.11]{asplund2021simplicial} that we obtain a good sectorial cover $\widetilde X = \widetilde X_1 \cup \cdots \cup \widetilde X_m$ of $\widetilde X$. By definition we have that for any $\varnothing \neq A \subset \left\{1,\ldots,m\right\}$
					\[
						N^Z \left(\bigcap_{i\in A} \widetilde X_i\right) \cong X_A^{(k)} \times T^\ast \R^k,
					\]
					where each $X_A^{(k)}$ is a Weinstein $(2n-2k)$-manifold. Letting $P_A := \eval[0]{P}_A$ with the restricted Weinstein structure gives a Weinstein manifold $P_A \hookrightarrow \partial X_A^{(k)} \times T^\ast \R^k$ which comes from a Weinstein hypersurface $P_A' \hookrightarrow \partial X_A^{(k)}$. Each pair $(X_A^{(k)},P_A')$ corresponds to a Weinstein $(2n-2k)$-sector $X_A^{(k)'}$ via convex completion \cite[Section 2.7]{ganatra2020covariantly}, and thus $\widetilde X_1 \cup \cdots \cup \widetilde X_m$ is a good sectorial cover of $(\widetilde X, P)$, which by definition is a good sectorial cover for $X$.

					To finish the proof we need to reconstruct the simplicial decomposition $((C,C'),(\boldsymbol V, \boldsymbol V'))$ of $(\widetilde X,P)$ using the good sectorial cover $X_1 \cup \cdots \cup X_m$ of $X$. For any $\varnothing \neq A \subset \left\{1,\ldots,m\right\}$ we have
					\[
						N^Z \left(\bigcap_{i\in A} X_i\right) \cong X_A^{(k)} \times T^\ast \R^k,
					\]
					where each $X_A^{(k)}$ is a Weinstein $(2n-2k)$-sector which via convex completion \cite[Section 2.7]{ganatra2020covariantly} corresponds to the Weinstein pair $(X_A^{(k)'},P_A')$. The simplicial complex $C$ is given by the \v{C}ech nerve of the cover $X_1 \cup \cdots \cup X_m$ of $X$. The simplicial subcomplex $i \colon C' \hookrightarrow C$ is given by the \v{C}ech nerve of the restriction of the cover $X_1 \cup \cdots \cup X_m$ to $\partial X$. The Weinstein hypersurface $P_A'$ is precisely of the form $\#(\boldsymbol V'_{\supset \sigma_k'} \sqcup_i \boldsymbol V_{\supsetneq \sigma_k}) \hookrightarrow \partial X_A^{(k)'}$ where $\sigma_k$ is the face corresponding to $A \subset \left\{1,\ldots,m\right\}$. Forgetting the Weinstein hypersurface $P$ makes $X_1 \cup \cdots \cup X_m$ a good sectorial cover of $\widetilde X$ which via \cite[Theorem 3.11]{asplund2021simplicial} corresponds to a simplicial decomposition $(C, \boldsymbol V)$. The set $\boldsymbol V'_{\supset \sigma_k'}$ consists of the Weinstein manifolds which are restrictions of $P_A'$ to further intersections, with the corresponding Weinstein hypersurfaces which also is equal to a restriction of $P \hookrightarrow \partial \widetilde X$. Taking the union over all $\varnothing \neq A \subset \left\{1,\ldots,m\right\}$ we obtain $(C', \boldsymbol V')$ which by construction is a simplicial decomposition of $P$, and by construction we have a hypersurface inclusion $(C', \boldsymbol V') \hookrightarrow (C,\boldsymbol V)$.
				\end{proof}
	\subsection{Weinstein hypersurfaces and simplicial decompositions}\label{sec:ce_dga_for_hypersurface_wrt_simp_decomp}
		We now review the definition of the Chekanov--Eliashberg dg-algebra of a Weinstein hypersurface with a chosen handle decomposition from \cite{asplund2020chekanov} and review its relationship to Chekanov--Eliashberg dg-algebras with loop space coefficients. We describe the Chekanov--Eliashberg dg-algebra of a Weinstein hypersurface with respect to a simplicial decomposition of the corresponding Weinstein pair.

		Let $X$ be a Weinstein $2n$-manifold and let $V^{2n-2} \hookrightarrow \partial X^{2n}$ be a Weinstein hypersurface. Let $X_0$ and $V_0$ denote the subcritical part of $X$ and $V$, respectively, see \cref{rmk:handle_decomp}. Fix a Weinstein handle decomposition $h$ of $V$. Now let $X_V$ be the Weinstein cobordism obtained by attaching the basic building block $V \times D_\varepsilon T^\ast \varDelta^1$ to $X \sqcup (\R \times (V \times \R))$, and let $\varSigma(h,\varepsilon)$ denote the union of attaching spheres of the top handles of $V \times D_\varepsilon T^\ast \varDelta^1$ in $X_V$. We will denote by $\varSigma(h) := \varSigma(h,\varepsilon_0)$ for some arbitrary but fixed $\varepsilon_0 > 0$. This is a link of Legendrian spheres in the positive end of $(X_{V_0})$.
		\begin{figure}[!htb]
			\centering
			\includegraphics{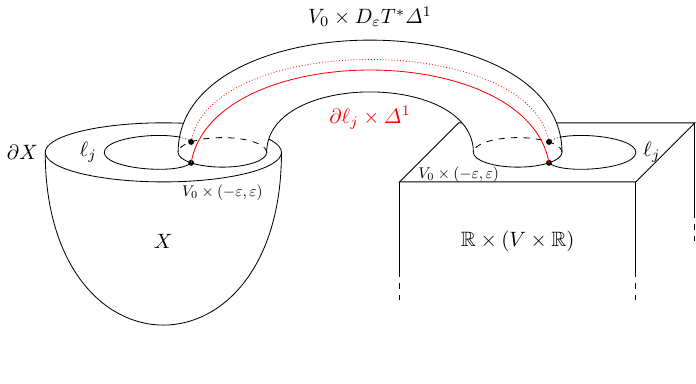}
			\caption{The Weinstein cobordism $X_V$ which we identify as the Weinstein manifold $X$ stopped at the Weinstein hypersurface $V \hookrightarrow \partial X$.}\label{fig:X_stopped_at_V}
		\end{figure}
		The union of Legendrian spheres $\varSigma(h)$ has the following geometric description.
		\begin{lma}\label{lma:top_attaaching_spheres_stopped_manifold}
			Let $\ell := \bigcup_{j} \ell_j$ denote the union of the core disks of the top handles in $h$. Then we have
			\[
				\varSigma(h,\varepsilon) = \ell \cup_{\partial \ell \times \left\{-1\right\}} (\partial \ell \times \varDelta^1) \cup_{\partial \ell \times \left\{1\right\}} \ell.
			\]
			Additionally let $\kappa := \bigcup_{i} \kappa_i$ denote the union of attaching spheres of the top handles of $X$. Then $\varSigma(h) \cup \kappa$ is the union of attaching spheres of the top Weinstein handles of $X_V$.
		\end{lma}
		\begin{proof}
			This is part of \cite[Lemma 2.21]{asplund2021simplicial}, see also \cite[Lemma 3.1]{asplund2020chekanov}. 
		\end{proof}
		\begin{lma}\label{lma:one-to-one-corr_reeb_chords}
			For all $\mathfrak a > 0$ there exists some $\delta > 0$ and an arbitrary small perturbation of the Weinstein hypersurface $V_0 \hookrightarrow \partial X$ such that for all $0 < \varepsilon < \delta$ the Reeb chords of $\varSigma(h,\varepsilon)$ of action $< \mathfrak a$ are in one-to-one grading preserving correspondence with Reeb chords of $\ell \subset \partial X_0$ of action $< \mathfrak a$, and Reeb chords of $\partial \ell \subset \partial V_0$ of action $< \mathfrak a$.
		\end{lma}
		\begin{proof}
			This is precisely \cite[Lemma 2.3]{asplund2020chekanov} and \cite[Lemma 2.24]{asplund2021simplicial}. 

			One of the key points is to prevent Reeb chords that start at a point in $\partial X$ to enter the basic building block $V \times D_\varepsilon T^\ast \varDelta^1$. By general position, a Reeb chord starting on $\ell \subset \partial X$ will not meet the singular Legendrian $\Skel V \subset \partial X$ and with a fixed action bound on the Reeb chords, we can shrink the attaching region a sufficient amount so that every such Reeb chord will stay away from the attaching locus of the basic building block $V \times D_\varepsilon T^\ast \varDelta^1$.
		\end{proof}
		\begin{dfn}[{\cite[Definition 3.2]{asplund2020chekanov}}]\label{dfn:ce_hypersurface}
			Let $X$ be a Weinstein manifold and $V \hookrightarrow \partial X$ a Weinstein hypersurface together with a chosen handle decomposition $h$ of $V$. We define the Chekanov--Eliashberg dg-algebra of the pair $(V,h)$ as
			\[
				CE^\ast((V,h);X) := CE^\ast(\varSigma(h);(X_{V_0})).
			\]
		\end{dfn}
		\begin{lma}\label{lma:dg-subalgebra}
			For any $\mathfrak a > 0$ there is some $\varepsilon > 0$ small enough so that the Chekanov--Eliashberg dg-algebra generated by the Reeb chords of $\varSigma(h,\varepsilon)$ that are contained in the basic building block $V \times D_{\varepsilon} T^\ast \varDelta^1$ of action $< \mathfrak a$ is canonically isomorphic to the Chekanov--Eliashberg dg-algebra generated by Reeb chords of $\partial \ell \subset \partial V_0$ of action $< \mathfrak a$.
		\end{lma}
		\begin{proof}
			This is proven in \cite[Corollary 2.6]{asplund2021simplicial}. The Reeb chords of $\varSigma(h)$ that are contained in the basic building block $V \times D_{\varepsilon_0} T^\ast \varDelta^1$ corresponds to Reeb chords of $\partial \ell \subset \partial V_0$ below the given action bound by \cref{lma:one-to-one-corr_reeb_chords}. There are also generic choices of almost complex structures $J$ on $\R \times \partial (X_{V_0})$ and $J_{V_0}$ on $\R \times \partial V_0$ such that we have a one-to-one correspondence between $J$-holomorphic curves in $\R \times \partial (X_{V_0})$ with boundary on $\R \times \varSigma(h)$, and $J_{V_0}$-holomorphic curves in $\R \times \partial V_0$ with boundary on $\R \times \partial \ell$ by \cite[Lemma 2.7]{asplund2020chekanov}.
		\end{proof}
		\begin{lma}[{\cite[Lemma 4.2]{asplund2020chekanov}}]\label{lma:dg-subalg_quasi-iso_loop_space}
			Let $X$ be a Weinstein manifold and $\varLambda \subset \partial X$ a smooth (possibly disconnected) Legendrian submanifold. Let $V$ be the union of a small cotangent neighborhood of each component of $\varLambda$ and let $h$ be a choice of handle decomposition of $V$ such that each component of $V$ has a single top handle. There is a quasi-isomorphism of dg-algebras
			\[
				CE^\ast(\partial \ell; V_0) \cong C_{-\ast} \left(\bigsqcup_{i \in \pi_0(\varLambda)} \varOmega \varLambda_i\right),
			\]
			where $\varLambda_i$ is the $i$-th component of $\varLambda$.
		\end{lma}
		\begin{proof}
			The idea of the proof is to construct a map $\varphi \colon CE^\ast(\partial \ell; V_0) \longrightarrow C_{-\ast} \left(\bigsqcup_{i \in \pi_0(\varLambda)} \varOmega \varLambda_i\right)$ by counting $J_{V_0}$-holomorphic curves in $\R \times \partial V_0$ such that the following diagram commutes up to dg-homotopy.
			\[
				\begin{tikzcd}[row sep=scriptsize, column sep=scriptsize]
					CW^\ast(F,F) \dar{\varPsi} \rar{\varPhi} & C_{-\ast} \left(\bigsqcup_{i \in \pi_0(\varLambda)} \varOmega \varLambda_i\right) \\
					CE^\ast(\partial \ell; V_0) \urar[swap]{\varphi} &
				\end{tikzcd}.
			\]
			Here $F$ denotes the disjoint union of one cotangent fiber of each component in the Weinstein neighborhood $N(\varLambda) \cong T^\ast \varLambda$ of $\varLambda \subset \partial X$. The map $\varPsi$ is the surgery $A_\infty$-quasi-isomorphism from \cite[Theorem 5.8]{bourgeois2012effect}, and $\varPhi$ is the $A_\infty$-quasi-isomorphism from \cite{abouzaid2012on,asplund2019fiber}, see \cite[Lemma 4.2]{asplund2020chekanov} for more details.
		\end{proof}
		The following is now almost an immediate consequence of \cref{lma:dg-subalg_quasi-iso_loop_space} and \cite[Theorem 1.1]{asplund2020chekanov}.
		\begin{thm}[{\cite[Theorem 1.2]{asplund2020chekanov}}]\label{thm:loop_space_dga}
			Let $X$ be a Weinstein manifold and $\varLambda \subset \partial X$ a smooth (possibly disconnected) Legendrian submanifold. Let $V$ be a small cotangent neighborhood of $\varLambda$ and let $h$ be a choice of handle decomposition such that $V$ has a single top handle for each component of $V$. There is a quasi-isomorphism of dg-algebras
			\[
				CE^\ast((V,h);X) \cong CE^\ast \left(\varLambda, C_{-\ast}\left(\bigsqcup_{i\in \pi_0(\varLambda)} \varOmega \varLambda_i\right)\right),
			\]
			where $\varLambda_i$ denotes the $i$-th component of $\varLambda$ and where $CE^\ast(\varLambda, C_{-\ast}(\varOmega \varLambda))$ is the Chekanov--Eliashberg dg-algebra with loop space coefficients as defined in \cite{ekholm2017duality}.
		\end{thm}
		\begin{proof}
			Note that the difference between this statement and \cite[Theorem 1.2]{asplund2021simplicial} is that we allow $\varLambda \subset \partial X$ to be disconnected. The proof is the same as the proof of \cite[Theorem 1.2]{asplund2020chekanov}.
		\end{proof}
		Let $(X,P)$ be a Weinstein pair and let $((C,C'),(\boldsymbol V, \boldsymbol V'))$ be a simplicial decomposition of $(X,P)$. Pick a handle decomposition for each $V_{\sigma_k} \in \boldsymbol V$ and denote the set of all such by $\boldsymbol h$. Similarly let $\boldsymbol h'$ denote the set of chosen handle decompositions of each $V_{\sigma_k'}' \in \boldsymbol V'$. The union of the top attaching spheres of $X \cong \# \boldsymbol V$ is denoted by $\varSigma(\boldsymbol h)$, see \cite[Definition 2.20 and Lemma 2.21]{asplund2021simplicial} for details on the constructions.

		\begin{figure}[!htb]
			\centering
			\includegraphics{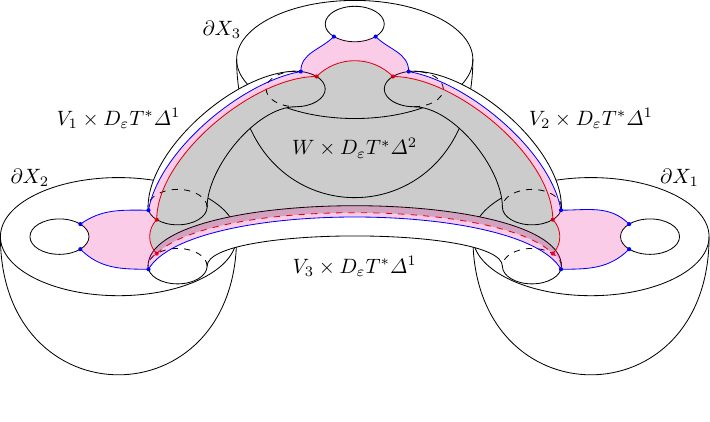}
			\caption{The Weinstein hypersurface $P \hookrightarrow \partial X$ presented as a join with respect to the simplicial decomposition $((C,C'), (\boldsymbol V, \boldsymbol V'), (\boldsymbol A, \boldsymbol A'))$ where $C = C' = \varDelta^2$ and the simplicial subcomplex $C' \hookrightarrow C$ being the identity map.}\label{fig:triple_join_rel}
		\end{figure}

		We now describe the union of the top attaching spheres of $X_P$ with respect to the simplicial decomposition of $(X,P)$.
		Let $\ell' = \bigcup_{\sigma_k' \in C'} \ell'_{\sigma_k'} = \bigcup_{\sigma_k' \in C'} \bigcup_{j} \ell'_{\sigma_k',j}$ denote the union of the core disks of the top handles of each $V'_{\sigma_k'} \in \boldsymbol V'$. Let $i \colon C' \hookrightarrow C$ be the simplicial subcomplex in the simplicial decomposition of $(X,P)$. For every $\sigma_0\in \im i$ we have the Weinstein hypersurface
		\begin{equation}\label{eq:hypersurf}
			\#(\boldsymbol V'_{\supset \sigma_0'} \sqcup_i \boldsymbol V_{\supsetneq \sigma_0}) \hookrightarrow \partial V_{\sigma_0},
		\end{equation}
		as in \cref{dfn:hypersurface_inclusion}, see \cref{fig:triple_join_rel_vertex}.

		Denote the union of the top attaching spheres of the $\#(\boldsymbol V'_{\supset \sigma_0} \sqcup_i \boldsymbol V_{\supsetneq \sigma_0})$ by $\varSigma(\boldsymbol h'_{\sigma_0'} \sqcup_i \boldsymbol h_{\sigma_0})$ where $\boldsymbol h'_{\sigma_0'} \sqcup_i \boldsymbol h_{\sigma_0}$ is the union of handle decompositions of the Weinstein manifolds in $\boldsymbol V'_{\supset \sigma_0} \sqcup_i \boldsymbol V_{\supsetneq \sigma_0}$. Let $\varDelta(\boldsymbol h'_{\sigma_0'} \sqcup_i \boldsymbol h_{\sigma_0})$ be the corresponding union of core disks of the top handles in $\boldsymbol h'_{\sigma_0'} \sqcup_i \boldsymbol h_{\sigma_0}$, and note that topologically $\varDelta(\boldsymbol h'_{\sigma_0'} \sqcup_i \boldsymbol h_{\sigma_0})$ is a union of $(n-1)$-disks in $\partial V_{\sigma_0}$ via the Weinstein embedding \eqref{eq:hypersurf}.
		\begin{figure}[!htb]
			\centering
			\includegraphics[scale=0.7]{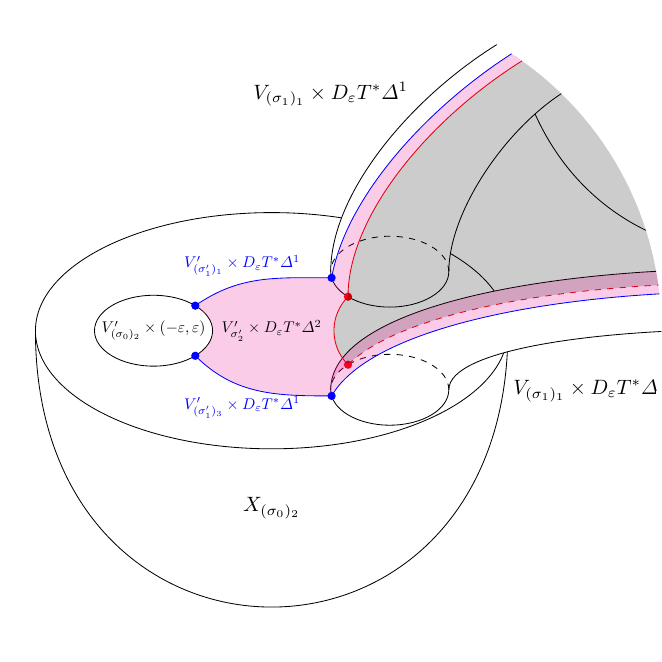}
			\caption{The Weinstein hypersurface $\#(\boldsymbol V'_{\supset (\sigma_0)_2} \sqcup_i \boldsymbol V_{\supsetneq (\sigma_0)_2}) \hookrightarrow \partial X_{(\sigma_0)_2}$ which is a part of the construction of $P \hookrightarrow \partial X$ with respect to the simplicial decomposition $((C,C'), (\boldsymbol V, \boldsymbol V'))$ as in \cref{fig:triple_join_rel}. The Weinstein hypersurfaces $V'_{(\sigma_1')_3} \#_{V'_{\sigma_2'}} W \hookrightarrow \partial V_3$ and $V'_{(\sigma_1')_1} \#_{V'_{\sigma_2'}} W \hookrightarrow \partial V_1$ are also depicted.}\label{fig:triple_join_rel_vertex}
		\end{figure}

		We then extend every $\varDelta(\boldsymbol h'_{\sigma_0'} \sqcup_i \boldsymbol h_{\sigma_0})$ trivially over $\# \boldsymbol V \cong X$ and denote the result by $\varSigma_P(\boldsymbol h')$. The result is the union of top attaching spheres of the Weinstein manifold $P \cong \# \boldsymbol V'$, see \cite[Definition 2.20]{asplund2021simplicial}. By construction we have a decomposition $\varSigma_P(\boldsymbol h') = \bigcup_{\sigma_k' \in C'} \varSigma_{P, \sigma_k'}(\boldsymbol h')$, where $\varSigma_{P, \sigma_k'}(\boldsymbol h')$ is the union of the top attaching spheres of $P$ living over the locus corresponding to the $k$-face $\sigma_k' \in C'$.
		\begin{figure}[!htb]
			\centering
			\includegraphics[scale=1]{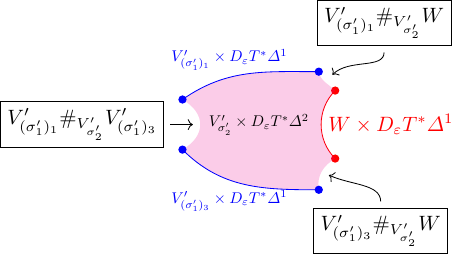}
			\caption{The first gluing step in the construction of the Weinstein hypersurface $\#(\boldsymbol V'_{\supset (\sigma_0)_2} \sqcup_i \boldsymbol V_{\supsetneq (\sigma_0)_2}) \hookrightarrow \partial X_{(\sigma_0)_2}$ in \cref{fig:triple_join_rel_vertex}.}\label{fig:triple_join_rel_vertex_handle}
		\end{figure}
		\begin{lma}
			The union of $(n-1)$-spheres $\varSigma_P(\boldsymbol h')$ is the union of the attaching spheres of the top handles of the Weinstein manifold $P$.
		\end{lma}
		\begin{proof}
			This is repetition of the proof of \cite[Lemma 2.21]{asplund2021simplicial}.
		\end{proof}
		Let $\varDelta_P(\boldsymbol h')$ denote the union of the core disks of the top handles in $\boldsymbol h'$. Since we have a Weinstein hypersurface $P \hookrightarrow \partial X$, it follows that $\varDelta_P(\boldsymbol h')$ is a Legendrian submanifold in $\partial X$ with boundary in $\partial P$. Like the attaching spheres, the core disks also decompose as $\varDelta_P(\boldsymbol h') = \bigcup_{\sigma_k' \in C'} \varDelta_{P, \sigma_k'}(\boldsymbol h')$ where $ \varDelta_{P, \sigma_k'}(\boldsymbol h')$ is the core disk with boundary $\partial \varDelta_{P, \sigma_k'}(\boldsymbol h') = \varSigma_{P, \sigma_k'}(\boldsymbol h')$.

		Our next task is to describe the Chekanov--Eliashberg dg-algebra of the pair $(P, \boldsymbol h')$ in $X$. To that end, we construct $(X_{P})_0$ by attaching the handle $P_0 \times D_\varepsilon T^\ast \varDelta^1$ to $X \sqcup (\R \times (P \times \R))$. By \cref{lma:top_attaaching_spheres_stopped_manifold} we have that the union of attaching spheres of $(X_{P})_0$ is
		\begin{equation}\label{eq:att_spheres_hypersurface}
			\varSigma_{X_P}(\boldsymbol h') := \varDelta_P(\boldsymbol h') \cup_{\partial \varDelta_P(\boldsymbol h') \times \left\{-1\right\}}(\partial \varDelta_P(\boldsymbol h') \times \varDelta^1) \cup_{\partial \varDelta_P(\boldsymbol h') \times \left\{1\right\}} \varDelta_P(\boldsymbol h').
		\end{equation}
		Below a given action bound on Reeb chords, and for arbitrarily thin simplicial handles it follows from \cref{lma:one-to-one-corr_reeb_chords} that the Reeb chords of $\varSigma_{X_P}(\boldsymbol h')$ corresponds to Reeb chords of $\varDelta_P(\boldsymbol h') \subset \partial X_0$ and Reeb chords of $\partial \varDelta_P(\boldsymbol h') = \varSigma_P(\boldsymbol h') \subset \partial P_0$.
		\begin{dfn}[{\cite[Definition 2.22]{asplund2021simplicial}}]
			Let $(C, \boldsymbol V)$ be a simplicial decomposition of $X$ and let $\boldsymbol h$ be a choice of handle decomposition of $\boldsymbol V$. Denote the union of core disks of the top handles in $\boldsymbol h$ by $\ell$.
			\begin{enumerate}
				\item Let $\mathcal R_X(\sigma_k)$ denote the set of Reeb chords of $\partial \ell_{\sigma_k} \cup \bigcup_{\substack{\sigma_i \supset \sigma_k \\ k+1 \leq i \leq m}}(\ell_{\sigma_i} \times \varDelta^{i-k-1}) \subset \partial V^{(k)}_{\sigma_k,0}$ for $k\in \left\{0,\ldots, m\right\}$, where $\ell_{\sigma_i} \times \varDelta^{i-k-1}$ are the core disks of the top handles of $V^{(i)}_{\sigma_i,0} \times D_\varepsilon T^\ast \varDelta^{i-k-1} \hookrightarrow \partial V^{(k)}_{\sigma_k,0}$.
				\item Let $\mathcal R_X(\sigma_k, \mathfrak a)$ denote the set of Reeb chords in $\mathcal R_X(\sigma_k)$ of action $< \mathfrak a$.
			\end{enumerate}
		\end{dfn}
		\begin{dfn}
			Let $((C,C'), (\boldsymbol V, \boldsymbol V'))$ be a simplicial decomposition of the Weinstein pair $(X,P)$. Let $\boldsymbol h$ and $\boldsymbol h'$ be choices of handle decompositions of $\boldsymbol V$ and $\boldsymbol V'$ respectively. Denote the union of core disks of the top handles in $\boldsymbol h$ and $\boldsymbol h'$ by $\ell$ and $\ell'$ respectively.
			\begin{enumerate}
				\item Let $\mathcal R_{(X,P)}(\sigma_k')$ denote the set of Reeb chords of $\ell'_{\sigma_k'} \cup \bigcup_{\substack{\sigma_i' \supset \sigma_k'\\ k+1 \leq i \leq m'}} (\ell'_{\sigma_i'}\times \varDelta^{i-k})$ for $k\in \left\{0,\ldots,m'\right\}$, where $\ell'_{\sigma_i'}\times \varDelta^{i-k}$ are the core disks of the top handles of $V'^{(i)}_{\sigma_i',0} \times D_\varepsilon T^\ast \varDelta^{i-k} \hookrightarrow \partial V_{\sigma_k,0}^{(k)}$.
				\item Let $\mathcal R_{(X,P)}(\sigma_k',\mathfrak a)$ denote the set of Reeb chords in $\mathcal R_{(X,P)}(\sigma_k')$ of action $< \mathfrak a$.
			\end{enumerate}
		\end{dfn}
		Recall from \cite[Section 2.2.4]{asplund2021simplicial} that the precise construction of $\# \boldsymbol V$ depends on a size parameter $\boldsymbol \varepsilon = \left\{0 < \varepsilon_m < \cdots < \varepsilon_1\right\}$, where each $\varepsilon_k > 0$ is the size of the simplicial handle $H^k_{\varepsilon_k}(V^{2n-2k}_{\sigma_k})$, see \cite[Definition 2.11]{asplund2021simplicial}. For notational simplicity we write $0 < \boldsymbol \varepsilon < \delta$ to mean $0 < \varepsilon_m$ and $\varepsilon_1 < \delta$.
		\begin{lma}\label{lma:one_to_one_corr_reeb_chords_core_disks}
			For all $\mathfrak a > 0$ there exists some $\delta > 0$ and an arbitrary small perturbation of $\boldsymbol V_0$ such that for all $0 < \boldsymbol \varepsilon < \delta$ the following holds.
			\begin{enumerate}
				\item There is a one-to-one grading preserving correspondence between Reeb chords of action $< \mathfrak a$ of $\varDelta_P(\boldsymbol h') \subset \partial X_0$ and Reeb chords in $\bigcup_{\substack{\sigma_k' \in C' \\ 0 \leq k \leq m'}} \mathcal R_{(X,P)}(\sigma_k', \mathfrak a)$.
				\item There is a one-to-one grading preserving correspondence between Reeb chords of $\partial \varDelta_P(\boldsymbol h') \subset \partial P_0$ of action $< \mathfrak a$ and Reeb chords in $\bigcup_{\substack{\sigma_k' \in C' \\ 0 \leq k \leq m'}} \mathcal R_P(\sigma_k', \mathfrak a)$.
			\end{enumerate}
		\end{lma}
		\begin{proof}
			\begin{enumerate}
				\item It follows from the construction of $\# \boldsymbol V$ that Reeb chords of $\varDelta_P(\boldsymbol h') \subset \partial X_0$ of action $< \mathfrak a$ (possibly after shrinking the size of the simplicial handles) are contained in $\pi^{-1}(0,0)$ where $\pi \colon \partial V^{(k)}_{\sigma_k} \times D_\varepsilon T^\ast \varDelta^k \longrightarrow D_\varepsilon T^\ast \varDelta^k$ is the projection to the second factor, corresponding to each $\sigma_k\in C$, see \cite[Lemma 2.24]{asplund2021simplicial}. It follows from the construction of $\varDelta_P(\boldsymbol h')$ that the part lying in the positive contact boundary of $V^{(k)}_{\sigma_k} \times D_\varepsilon T^\ast \varDelta^k$ has the form $(\ell'_{\sigma_k'} \times \varDelta^k) \cup \bigcup_{\substack{\sigma_i' \supset \sigma_k'\\ k+1 \leq i \leq m'}} (\ell'_{\sigma_i'}\times \varDelta^{i})$ from which the result follows.
				\item The proof is the same as in (1) above.
			\end{enumerate}
		\end{proof}
	\subsection{Simplicial descent for Weinstein pairs}\label{sec:descent_hypersurfaces}
		Let $((C,C'), (\boldsymbol V, \boldsymbol V'))$  be a simplicial decomposition of a Weinstein pair $(X,P)$. In this section we prove that we have a gluing formula that computes the Chekanov--Eliashberg dg-algebra $CE^\ast((P, \boldsymbol h'); X)$. In a special case this may be regarded as a gluing formula for the Chekanov--Eliashberg dg-algebra with loop space coefficients, see \cref{rmk:gluing_of_loop_spaces}.

		Let $\boldsymbol h$ and $\boldsymbol h'$ be handle decompositions of $\boldsymbol V$ and $\boldsymbol V'$ respectively, as in \cref{sec:ce_dga_for_hypersurface_wrt_simp_decomp}. Let $\varSigma_{X_P}(\boldsymbol h') \subset \partial (X_{P})_0$ be the union of the top attaching spheres of the simplicial handle $P_0 \times D_\varepsilon T^\ast \varDelta^1$. We use the notation
		\[
		 \varDelta_{P, \supset \sigma_k'}(\boldsymbol h') := \bigcup_{\substack{\sigma_i' \supset \sigma_k' \\ k\leq i\leq m'}} \varDelta_{P, \sigma_i'}(\boldsymbol h'), \quad \varSigma_{X_P, \supset \sigma_k'}(\boldsymbol h') := \bigcup_{\substack{\sigma_i' \supset \sigma_k' \\ k\leq i\leq m'}}\varSigma_{X_P, \sigma_i'}(\boldsymbol h').
		\]
		From \cite[Section 2.5.1]{asplund2021simplicial} we recall the following. For each $k$-face $\sigma_k \in C$ define
			\begin{equation}\label{eq:stopped_handle_data}
				\widetilde V^{(i)}_{\sigma_i}(\sigma_k) := \begin{cases}
					V^{(i)}_{\sigma_i}& \text{if } \sigma_i \supset \sigma_k \\
					(-\infty,0] \times \bigsqcup_{f\in F}(\# \widetilde{\boldsymbol V}_{\supsetneq \sigma_i,f}(\sigma_k)) \times \R & \text{otherwise}
				\end{cases}
			\end{equation}
			where $\widetilde{\boldsymbol V}_{\supsetneq \sigma_i,f}(\sigma_k)$ is the set of Weinstein manifolds $\widetilde V^{(\ell)}_{\sigma_{\ell}}(\sigma_k)$ for $f \supset \sigma_\ell \supsetneq \sigma_i$ with the same Weinstein hypersurfaces as in $\boldsymbol V$, and the Weinstein hypersurfaces induced by the inclusion $\{0\} \times (W \times \left\{0\right\}) \hookrightarrow (-\infty,0] \times (W \times \R)$ for those $\widetilde V^{(\ell)}_{\sigma_\ell}(\sigma_k)$ of the form as in the bottom row of \eqref{eq:stopped_handle_data}. Finally define $\widetilde{\boldsymbol V}(\sigma_k) := \bigcup_{\substack{f\in F \\ \sigma_i \subset f}} \widetilde{\boldsymbol V}_{\supsetneq \sigma_i,f}(\sigma_k)$ and
			\begin{equation}\label{eq:x_stopped_away_from_sigma_k}
				X(\sigma_k) := \# \widetilde{\boldsymbol V}(\sigma_k)\, .
			\end{equation}
			Similarly, for each $\sigma_k \in C$ we define the Weinstein hypersurface $P(\sigma_k) \hookrightarrow \partial X(\sigma_k)$ as follows. If $\sigma_k \not \in \im i$, then $P(\sigma_k) := P$. Else $P(\sigma_k')$ is defined in analogy with $X(\sigma_k)$ above.
		\begin{lma}\label{lma:one_to_one_corr_generators}
			For all $\mathfrak a > 0$ there exists some $\delta > 0$ and an arbitrary small perturbation of $\boldsymbol V_0$ such that for all $0 < \boldsymbol \varepsilon < \delta$ the following holds.
			\begin{enumerate}
				\item There is a one-to-one grading preserving correspondence between Reeb chords of action $< \mathfrak a$ of $\partial \varDelta_{P, \supset \sigma_k'}(\boldsymbol h') \subset \partial P(\sigma_k')_0$ and Reeb chords in $\bigcup_{\sigma_i' \supset \sigma_k'} \mathcal R_{P}(\sigma_i', \mathfrak a)$.
				\item There is a one-to-one grading preserving correspondence between Reeb chords of action $< \mathfrak a$ of $\varDelta_{P, \supset \sigma_k'}(\boldsymbol h') \subset \partial (X(\sigma_k')_{P(\sigma_k')_0})$ and Reeb chords in $\bigcup_{\sigma_i' \supset \sigma_k'} \mathcal R_{(X,P)}(\sigma_i', \mathfrak a)$.
				\item There is a one-to-one grading preserving correspondence between Reeb chords of action $< \mathfrak a$ of $\varSigma_{X_P, \supset \sigma_k'}(\boldsymbol h') \subset \partial (X(\sigma_k')_{P(\sigma_k')_0})$ and Reeb chords in $\bigcup_{\sigma_i' \supset \sigma_k'} (\mathcal R_{(X,P)}(\sigma_i', \mathfrak a) \cup \mathcal R_{P}(\sigma_i', \mathfrak a))$.
			\end{enumerate}
		\end{lma}
		\begin{proof}
			\begin{enumerate}
				\item This follows immediately from \cite[Lemma 2.36]{asplund2021simplicial}.
				\item This follows the same idea as \cite[Proof of Lemma 2.36]{asplund2021simplicial}. Namely, for sufficient thin handles, the Reeb chords in $\bigcup_{\sigma_i' \supset \sigma_k'} \mathcal R_{(X,P)}(\sigma_i',\mathfrak a)$ corresponds to Reeb chords of $\varDelta_{P, \supset \sigma_k'}(\boldsymbol h') \subset \partial (X_{P})_0$ appearing in the loci of the contact boundary of $(X_{P})_0$ corresponding to $\sigma_i' \supset \sigma_k'$ in the simplicial decomposition of $X$ in as in the proof of \cref{lma:one_to_one_corr_reeb_chords_core_disks}. These loci in $(X_{P})_0$ are precisely those that are unchanged in the definition of $\widetilde{\boldsymbol V}(\sigma_k')$, see \eqref{eq:stopped_handle_data}. The previously existing Reeb chords of $\varDelta_{P, \supset \sigma_k'}(\boldsymbol h') \subset \partial (X_{P})_0$ lying over critical points in loci corresponding to $i$-faces $\sigma_i' \not \supset \sigma_k'$ disappear when passing to $\partial (X(\sigma_k')_{P(\sigma_k')_0})$, as the contact manifold in these loci are of the form $W \times \R$, where the Legendrian is contained in the $\R$-factor, see \eqref{eq:stopped_handle_data}.
				\item By construction of $\varSigma_{X_P, \supset \sigma_k'}(\boldsymbol h')$ (see \eqref{eq:att_spheres_hypersurface}) this follows from items (1) and (2).
			\end{enumerate}
		\end{proof}

		\begin{thm}\label{thm:gluing_formula_hypersurface_dga}
			\begin{enumerate}
				\item There is a quasi-isomorphism of dg-algebras
				\[
					CE^\ast(\partial \varDelta_P(\boldsymbol h');P_0) \cong \colim_{\sigma_k' \in C'} CE^\ast(\partial \varDelta_{P, \supset \sigma_k'}(\boldsymbol h'); P(\sigma_k')_0).
				\]
				\item There is a quasi-isomorphism of dg-algebras
				\[
					CE^\ast(\varSigma_{X_P}(\boldsymbol h');(X_{P_0})) \cong \colim_{\sigma_k' \in C'} CE^\ast(\varSigma_{X_P, \supset \sigma_k'}(\boldsymbol h'); (X(\sigma_k')_{P(\sigma_k')_0})).
				\]
			\end{enumerate}
		\end{thm}
		\begin{proof}
			Let $\delta, \mathfrak a > 0$ and $0 < \boldsymbol \varepsilon < \delta$.
			\begin{enumerate}
				\item This is exactly the content of \cite[Theorem 1.1]{asplund2021simplicial} whose proof we now give a brief summary of. For any $\sigma_k' \in C'$, define $\mathcal A_{P,\sigma_k'}(\boldsymbol h'; \boldsymbol \varepsilon, \mathfrak a)$ to be the algebra generated by Reeb chords in $\bigcup_{\sigma_i' \supset \sigma_k'} \mathcal R_{P}(\sigma_i', \mathfrak a)$. We equip $\mathcal A_{P,\sigma_k'}(\boldsymbol h'; \boldsymbol \varepsilon, \mathfrak a)$ with the same differential as $CE^\ast(\partial \varDelta_{P, \supset \sigma_k'}(\boldsymbol h',\boldsymbol \varepsilon); P(\sigma_k')_0^{\boldsymbol \varepsilon})$ where $P(\sigma_k')_0^{\boldsymbol \varepsilon}$ is the Weinstein manifold $P(\sigma_k')_0$ constructed using basic building blocks of the specified thickness $\boldsymbol \varepsilon$. This differential is well-defined by the bijection between the generating sets in \cref{lma:one_to_one_corr_generators}(1). Thus by definition we have a canonical isomorphism of dg-algebras already on the chain level
				\[
					\mathcal A_{P,\sigma_k'}(\boldsymbol h'; \boldsymbol \varepsilon, \mathfrak a) \cong CE^\ast(\partial \varDelta_{P, \supset \sigma_k'}(\boldsymbol h',\boldsymbol \varepsilon); P(\sigma_k')_0^{\boldsymbol \varepsilon}).
				\]
				By taking the colimit over all $\sigma_k' \in C'$ we obtain a canonical isomorphism of dg-algebras
				\[
					\colim_{\sigma_k' \in C'} \mathcal A_{P,\sigma_k'}(\boldsymbol h';\boldsymbol \varepsilon, \mathfrak a) \cong CE^\ast(\partial \varDelta_{P}(\boldsymbol h',\boldsymbol \varepsilon);P_0^{\boldsymbol \varepsilon}),
				\]
				because there is a bijection between their generating sets, and \cite[Lemma 2.26]{asplund2021simplicial} shows that the differentials on either side coincide. The key point is that \cite[Lemma 2.26]{asplund2021simplicial} prevents the existence of a $J$-holomorphic curve in $\R \times \partial P_0^{\boldsymbol \varepsilon}$ with positive puncture a Reeb chord completely contained in the basic building block over $\sigma_k' \in C'$, and at least one negative puncture at a Reeb chord completely contained in the basic building block over some $\sigma_\ell'\in C'$ with $\sigma_\ell' \cap \sigma_k' = \varnothing$. Finally, we let $(\mathfrak a, \delta) \to (\infty,0)$ in a controlled way (see \cite[Lemma 2.31]{asplund2021simplicial}) and using the invariance of the Chekanov--Eliashberg dg-algebra (see \cite[Lemma 2.30]{asplund2021simplicial} and references therein) which gives the result.
				\item Similar to the above, we define for any $\sigma_k' \in C'$ the dg-algebra $\mathcal A_{X,P,\sigma_k'}(\boldsymbol h'; \boldsymbol \varepsilon, \mathfrak a)$ to be the algebra generated by Reeb chords in $\bigcup_{\sigma_i' \supset \sigma_k'} (\mathcal R_{(X,P)}(\sigma_i', \mathfrak a) \cup \mathcal R_{P}(\sigma_i', \mathfrak a))$. We equip it with the same differential as $CE^\ast(\varSigma_{X_P, \supset \sigma_k'}(\boldsymbol h',\boldsymbol \varepsilon); (X(\sigma_k')_{P(\sigma_k')_0})^{\boldsymbol \varepsilon})$, which by definition gives a canonical isomorphism of dg-algebras
				\[
					\mathcal A_{X,P,\sigma_k'}(\boldsymbol h'; \boldsymbol \varepsilon, \mathfrak a) \cong CE^\ast(\varSigma_{X_P, \supset \sigma_k'}(\boldsymbol h',\boldsymbol \varepsilon); (X(\sigma_k')_{P(\sigma_k')_0})^{\boldsymbol \varepsilon}).
				\]
				By taking the colimit over $\sigma_k' \in C'$ we get a canonical isomorphism of dg-algebras
				\[
					\colim_{\sigma_k'\in C'} \mathcal A_{X,P,\sigma_k'}(\boldsymbol h'; \boldsymbol \varepsilon, \mathfrak a) \cong CE^\ast(\varSigma_{X_P, \supset \sigma_k'}(\boldsymbol h',\boldsymbol \varepsilon); (X_{P})_0^{\boldsymbol \varepsilon}),
				\]
				because there is a bijection between their generating sets by \cref{lma:one_to_one_corr_generators}(2), and again by construction of simplicial decompositions, the relevant $J$-holomorphic curves with positive puncture at a Reeb chord of $\varSigma_{X_P, \supset \sigma_k'}(\boldsymbol h') \subset \partial (X(\sigma_k')_{P(\sigma_k')_0})$ must be simplicial (see \cite[Definition 2.27]{asplund2021simplicial}) because both $(C, \boldsymbol V)$ and $(C', \boldsymbol V')$ are simplicial decompositions of $X$ and $P$, respectively. Such $J$-holomorphic curves projects in each building block $V^{(k)}_{\sigma_k} \times T^\ast \varDelta^k$ to the origin in the zero section of the second factor, or to stable manifolds of the origin of other building blocks $V^{(\ell)}_{\sigma_\ell} \times T^\ast \varDelta^\ell$ for $\sigma_\ell \supset \sigma_k$, and like the proof of item (1) above, it gives that the differentials coincide.
			\end{enumerate}
		\end{proof}
		\begin{rmk}\label{rmk:category}
			Every colimit appearing in \cref{thm:gluing_formula_hypersurface_dga} is taken in the category $\mathbf{dga}$ consisting of associative, non-commutative, non-unital dg-algebras over varying non-unital rings, which is known to admit colimits, see \cite[Section 2.5]{asplund2021simplicial} for details. Every dg-algebra appearing in this paper is semi-free and hence the colimit is again a semi-free dg-algebra generated by the union of the generators.
		\end{rmk}
		\begin{rmk}\label{rmk:gluing_of_loop_spaces}
			Suppose $((C,C'), (\boldsymbol V,\boldsymbol V'))$ is a simplicial decomposition of the Weinstein pair $(X,T^\ast \varLambda)$, where $\varLambda \subset \partial X$ is a smooth connected Legendrian submanifold, such that the resulting handle decomposition of $P$ which is induced by $\boldsymbol V'$ has a single top handle. In analogy with \cref{thm:loop_space_dga} we have quasi-isomorphisms of dg-algebras $CE^\ast(\partial \varDelta_P,P_0) \cong C_{-\ast}(\varOmega \varLambda)$ and in particular a quasi-isomorphism of dg-algebras
			\[
				CE^\ast(\varSigma_{X_P}(\boldsymbol h');(X_{P_0})) \cong CE^\ast(\varLambda, C_{-\ast}(\varOmega \varLambda); X).
			\]
			The gluing formula in \cref{thm:gluing_formula_hypersurface_dga} may thus be interpreted in this special case as a gluing formula for the Chekanov--Eliashberg dg-algebra with loop space coefficients. However, we should be careful: Each of the dg-subalgebras $CE^\ast(\varSigma_{X_P, \supset \sigma_k'}(\boldsymbol h');(X(\sigma_k)_{P(\sigma_k')})_0)$ is not necessarily quasi-isomorphic to a Chekanov--Eliashberg dg-algebra with loop space coefficients.
		\end{rmk}
	\subsection{Relative Legendrian submanifolds}\label{sec:relative_legendrians}
		We first review a few definitions from \cite{asplund2020chekanov,asplund2021simplicial} before generalizing them to the setting of simplicial decompositions of Weinstein pairs.
		\begin{dfn}[{\cite[Section 6.1]{asplund2020chekanov}}]\label{dfn:relative_leg}
			Let $(X,V)$ be a Weinstein pair. A \emph{Legendrian submanifold relative to $V$} is a Legendrian submanifold-with-boundary $\varLambda \subset \partial X$ such that $\partial \varLambda = \varLambda \cap V \subset \partial V$ is a Legendrian submanifold.
		\end{dfn}
		\begin{figure}[!htb]
			\centering
			\includegraphics{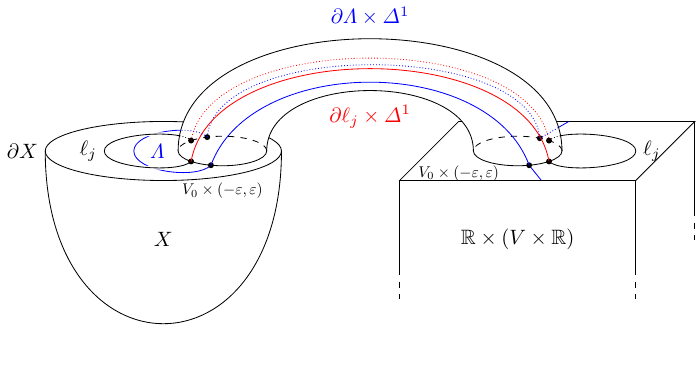}
			\caption{The definition of $CE^\ast(\varLambda, V; X)$ is given by the Chekanov--Eliashberg dg-algebra of the (non-compact) Legendrian submanifold $\varSigma_\varLambda$ which is shown in the figure.}
		\end{figure}
		\begin{dfn}[{\cite[Section 6.1]{asplund2020chekanov}}]\label{dfn:ce_of_rel_leg}
			Let $(X,V)$ be a Weinstein pair, and let $\varLambda$ be a Legendrian submanifold relative to $V$. Let $\varSigma_\varLambda := \varLambda \sqcup_{\partial \varLambda \times \left\{-1\right\}} (\partial \varLambda \times \varDelta^1) \sqcup_{\partial \varLambda \times \left\{1\right\}}( \left\{0\right\} \times \partial \varLambda \times \R) \subset \partial (X_{V_0})$ denote the union of top attaching spheres of $X_V$, and define
			\[
				CE^\ast(\varLambda, V; X) := CE^\ast(\varSigma_\varLambda; (X_{V_0})).
			\]
		\end{dfn}
		\begin{lma}[{\cite[Lemma 6.1]{asplund2020chekanov}}]\label{lma:subalgebra_of_relative_dga}
			The Reeb chords of $\varSigma(h)$ that correspond to Reeb chords of $\partial \ell \subset \partial V_0$ generates a dg-subalgebra of $CE^\ast(\varLambda,V;X)$ which is canonically quasi-isomorphic to $CE^\ast(\partial \varLambda; V_0)$.
		\end{lma}
		\begin{proof}
			This is repetition of the proof of \cref{lma:dg-subalgebra}.
		\end{proof}
		\begin{dfn}[{\cite[Definition 4.1]{asplund2021simplicial}}]
			A Legendrian submanifold \emph{relative to a simplicial decomposition} is a collection $\boldsymbol \varLambda$ consisting of one embedded Legendrian $(n-k-1)$-submanifold-with-boundary-and-corners $\varLambda_{\sigma_k} \subset \partial V^{(k)}_{\sigma_k}$ for each $\sigma_k\in C$, such that 
			\[
				\partial \varLambda_{\sigma_k} = \# \boldsymbol \varLambda_{\supsetneq \sigma_k} \subset \partial \# \boldsymbol V_{\supsetneq \sigma_k}\, ,
			\]
			where $\boldsymbol \varLambda_{\supsetneq \sigma_k} := \left\{\varLambda_{\sigma_i}\right\}_{\sigma_i \supsetneq \sigma_k}$, and $\# \boldsymbol \varLambda_{\supsetneq \sigma_k}$ is defined recursively in analogy to the Weinstein manifold $\# \boldsymbol V_{\supsetneq \sigma_k}$, see \cite[Section 2.2]{asplund2021simplicial}. 
		\end{dfn}
		We now extend these definitions to simplicial decompositions of Weinstein pairs.
		\begin{dfn}
			A Legendrian submanifold \emph{relative to a simplicial decomposition of a Weinstein pair} $(X,P)$ consists of a Legendrian submanifold $\boldsymbol \varLambda'$ relative to the given simplicial decomposition of $P$ and a collection $\boldsymbol \varLambda$ consisting of one embedded Legendrian $(n-k-1)$-submanifold-with-boundary-and-corners $\varLambda_{\sigma_k} \subset \partial V_{\sigma_k}^{(k)}$ for each $\sigma_k'\in C'$ such that
			\[
				\partial \varLambda_{\sigma_k} = \# (\boldsymbol \varLambda'_{\supset \sigma_k'} \sqcup_i \boldsymbol \varLambda_{\supsetneq \sigma_k}) \subset \partial \#(\boldsymbol V'_{\supset \sigma_k'} \sqcup_i \boldsymbol V_{\supsetneq \sigma_k})\, ,
			\]
			where
			$\boldsymbol \varLambda'_{\supset \sigma_k'} \sqcup_i \boldsymbol \varLambda_{\supsetneq \sigma_k} := \boldsymbol \varLambda'_{\supset \sigma_k'} \cup \boldsymbol \varLambda_{\supsetneq \sigma_{k}}$.
		\end{dfn}
		\begin{figure}[!htb]
			\centering
			\includegraphics[scale=0.9]{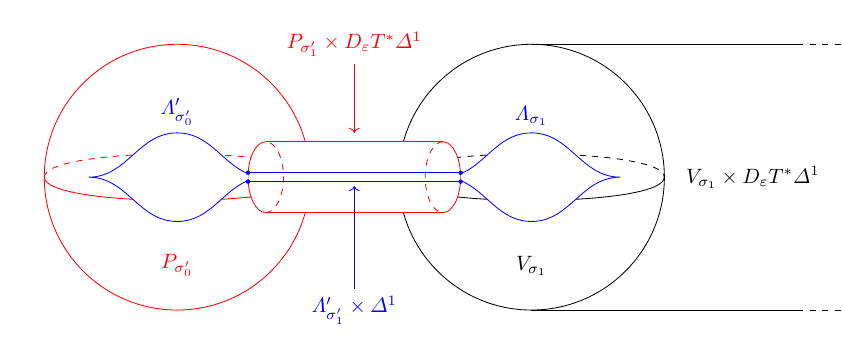}
			\caption{An example of a Legendrian submanifold relative to a simplicial decomposition with $C' = \varDelta^1$. The ambient manifold in view is $V_{\sigma_0}$ and the Legendrian submanifold $\partial \varLambda_{\sigma_0} = \varLambda'_{\sigma_0'} \#_{\varLambda'_{\sigma_1'}} \varLambda_{\sigma_1}$ which is shown in the figure.}\label{fig:rel_leg}
		\end{figure}
		\begin{rmk}
			We point out that Legendrian submanifolds-with-boundary was also studied from the point of view of sutured contact manifolds by Dattin \cite{dattin2022wrapped}, which is expected to generalize the results in this paper. From the point of view of sutured contact manifolds, Weinstein pairs $(X,V)$ give rise to a \emph{balanced} sutured contact manifold, and in that special case, we expect that our definition of $CE^\ast(\varLambda, V;X)$ coincides with the definition in \cite{dattin2022wrapped}.
		\end{rmk}
		\begin{dfn}\label{dfn:completion_rel_leg}
			Let $(\boldsymbol \varLambda, \boldsymbol \varLambda')$ be a Legendrian submanifold relative to a simplicial decomposition of a Weinstein pair $(X,P)$. Extending the various Legendrian submanifolds-with-boundary-and-corners in $\boldsymbol \varLambda$ over the simplicial handles used to construct $X$ (cf.\@ \cite[Definition 4.2]{asplund2021simplicial}) yields a Legendrian submanifold $\overline \varLambda \subset \partial X$ relative to $P$. We call $\overline \varLambda$ the \emph{$\boldsymbol V$-completion} of $\boldsymbol \varLambda$.
		\end{dfn}
		\begin{dfn}
			Let $((C,C'), (\boldsymbol V, \boldsymbol V'))$ be a simplicial decomposition of the Weinstein pair $(X,P)$ and let $\boldsymbol \varLambda$ be a Legendrian submanifold relative to $((C,C'), (\boldsymbol V, \boldsymbol V'))$. Define
			\[
				CE^\ast(\boldsymbol \varLambda; (X,P)) := CE^\ast(\overline \varLambda, P; X),
			\]
			where $\overline \varLambda$ is the $\boldsymbol V$-completion of $\boldsymbol \varLambda$.
		\end{dfn}
		We now turn to the analogous results for Legendrian submanifolds relative to simplicial decompositions of Weinstein pairs. Let $\varSigma(\boldsymbol h)$ denote the union of top attaching spheres of the simplicial handle $P \times D_\varepsilon T^\ast \varDelta^1$ used to construct $X_P$.

		We use the notation $X \to Y \to Z$ to mean a word of Reeb chords $vw$ where $w$ is a Reeb chord from $X$ to $Y$, and $w$ is a Reeb chord from $Y$ to $Z$. With the same notation as in \cite[Section 4.1]{asplund2021simplicial} we have the following.

		\begin{lma}\label{lma:surgery_desc_generators_relative}
			Assume $((C, C'),(\boldsymbol V, \boldsymbol V'))$ is a simplicial decomposition of $(X,P)$. For all $\mathfrak a > 0$ there exists some $\delta > 0$ and an arbitrarily small perturbation of the Weinstein hypersurface $P \hookrightarrow \partial X$ such that for all $0 < \boldsymbol \varepsilon < \delta$, the generators of $CE^\ast(\boldsymbol \varLambda;(X,P))$ are in one-to-one correspondence with composable words of Reeb chords of action $< \mathfrak a$ which are of the form $\boldsymbol \varLambda \to \boldsymbol \varLambda$ or
			\[
				\boldsymbol \varLambda \to \varSigma(\boldsymbol h) \to \cdots \to \varSigma(\boldsymbol h) \to \boldsymbol \varLambda.
			\]
			The differential on $CE^\ast(\boldsymbol \varLambda; (X, P))$ is induced by the differential in the Chekanov--Eliashberg dg-algebra of $\varSigma_{\varLambda} \cup \varSigma(\boldsymbol h) \subset \partial (X_{P_0})$.
		\end{lma}
		\begin{proof}
			The proof follows directly from the surgery description comparing generators before and after attachment of critical handles, see \cite[Lemma 4.5]{asplund2021simplicial} and \cite[Theorem 5.10]{bourgeois2012effect}.
		\end{proof}

		\begin{thm}\label{thm:gluing_relative_legendrians}
			Suppose $\boldsymbol \varLambda$ is a Legendrian submanifold relative to the simplicial decomposition $((C,C'), (\boldsymbol V, \boldsymbol V'))$ of the Weinstein pair $(X,P)$. Suppose that there are no Reeb chords from $\overline \varLambda$ to $\varSigma(\boldsymbol h)$ or that there are no Reeb chords from $\varSigma(\boldsymbol h)$ to $\overline \varLambda$. Then there is a quasi-isomorphism of dg-algebras
			\[
				CE^\ast(\boldsymbol \varLambda; (X,P)) \cong \colim_{\sigma_k\in C} CE^\ast(\boldsymbol \varLambda; (X(\sigma_k),P(\sigma_k'))).
			\]
		\end{thm}
		\begin{proof}
			By the surgery description of the generators in \cref{lma:surgery_desc_generators_relative}, the only Reeb chords of $\boldsymbol \varLambda$ of action $< \mathfrak a$ are $\boldsymbol \varLambda \to \boldsymbol \varLambda$ in $\partial (X_{P})_0$. The proof now follows by following the same strategy as \cref{thm:gluing_formula_hypersurface_dga}.
		\end{proof}

		For the rest of this section we restrict our attention to $C = C' = \varDelta^1$. Consider the case when $\varLambda$ is a smooth Legendrian submanifold in the boundary of $(X,P) := (X_1,P_1) \#_V (X_2,P_2)$ where $P_i = W_i \#_{Q_i} V$ for $i \in \left\{1,2\right\}$. We assume for simplicity that $W_i$ and $Q_i$ are subcritical. Consider a Legendrian submanifold $\boldsymbol \varLambda = \{\varLambda_1,\varLambda_2,\partial \varLambda_1, \partial \varLambda_2\}$ in $\partial X$ relative to $P$ and denote its completion by $\varLambda = \varLambda_1 \#_{\partial \varLambda_V} \varLambda_2$ where $\partial \varLambda_V$ is the common boundary of $\partial \varLambda_1$ and $\partial \varLambda_2$ along $V$ in $P_1$ and $P_2$, respectively. Here $\boldsymbol \varLambda_i = \{\varLambda_i, \partial \varLambda_i\}$ is a Legendrian submanifold in $\partial X_i$ relative to $P_i$. Let $\varSigma(h) = \ell_V \sqcup (\partial \ell_V \times \varDelta^1) \sqcup \ell_V$ be the top attaching sphere of the basic building block $V \times T^\ast \varDelta^1$ used to construct $X$, where $\ell_V$ is the union of top core disks of a handle decomposition of $V$ (and hence of each $P_i$ by the subcriticality assumption on $W_i$ and $Q_i$). 

		Let $\mathbb F$ be a field and let $\boldsymbol k := \bigoplus_{i \in \pi_0(\varLambda)} \mathbb F e_i$ where $\left\{e_i\right\}_{i\in \pi_0(\varLambda)}$ is a set of mutually orthogonal idempotents. We now give an alternative description of the dg-algebra $CE^\ast(\boldsymbol \varLambda;(X,P))$ as a $\boldsymbol k$-module. To that end let $\mathcal A_{i \to V}$ ($\mathcal A_{V\to i}$) denote the right (left) $CE^\ast((V,h);X_i)$-module freely generated by the Reeb chords from $\varLambda_i$ to $\ell_V$ and $\partial \varLambda_V$ to $\partial \ell_V$ ($\ell_V$ to $\varLambda_i$ and $\partial \ell_V$ to $\partial \varLambda_V$) in $\partial X_i$. Similarly we let $\mathcal A_{i \to i}$ denote the algebra generated by Reeb chords from $\varLambda_i$ to $\varLambda_i$ in $\partial X_i$ and $\partial \varLambda_i$ to $\partial \varLambda_i$ in $\partial P_i$. Let $(\mathcal A_V)_i$ denote the set of Reeb chords of $\ell_V$ in $\partial X_i$ and the set of Reeb chords of $\partial \ell_V \subset \partial V_0$. Also define $\mathcal A_V := (\mathcal A_V)_1 \cup (\mathcal A_V)_2$. Let $\mathcal A_{i\to V}$ ($\mathcal A_{V\to j}$) be the right (left) $\boldsymbol k$-module over the free algebra on $\mathcal A_V$ generated by the Reeb chords from $\varLambda_i$ to $\ell_V$ ($\ell_V$ to $\varLambda_j$). Then define
		\begin{align*}
			\mathcal A_{i\to V \to i} &:= \mathcal A_{i\to V} \otimes_{(\mathcal A_V)_i} \mathcal A_{V\to i}, \quad i \in \left\{1,2\right\} \\
			\mathcal A_{i\to V \to j} &:= \overline{\mathcal A}_{i\to V} \otimes_{\mathcal A_V} \overline{\mathcal A}_{V\to j}, \quad i \neq j \in \left\{1,2\right\}.
		\end{align*}
		In the following, we let $A \ast B$ and $A \ast_C B$ denote the free product and amalgamated free product of two algebras $A$ and $B$ (over $C$), respectively. For $\mathcal S$ a set, we let $\left\langle \mathcal S\right\rangle$ denote the free associative non-commutative algebra generated by $\mathcal S$. Let $A_1,\ldots,A_k$ be sets such that $A_i \cap A_j = \varnothing$ if $j \neq i+1$. Then we have a canonical isomorphism of free algebras
		\begin{equation}\label{eq:amalg_free_prod}
			\left\langle A_1 \cup \cdots \cup A_k\right\rangle \cong \left\langle A_1\right\rangle \ast_{A_1 \cap A_2} \left\langle A_2\right\rangle \ast_{A_2 \cap A_3} \cdots \ast_{A_{k-1} \cap A_k} \left\langle A_k\right\rangle.
		\end{equation}
		\begin{lma}\label{lma:quasi_iso_free_products}
			Fix a generic choice of metric and almost complex structure.
			\begin{enumerate}
				\item Let $i\in \left\{1,2\right\}$. The free algebra $\left\langle \mathcal A_{i\to i}\right\rangle \ast \left\langle \mathcal A_{i\to V \to i}\right\rangle$ becomes a dg-algebra when equipped with the differential of $CE^\ast(\varLambda_i,P_i;X_i)$ and there is a quasi-isomorphism of dg-algebras
				\[
					CE^\ast(\varLambda_i,P_i;X_i) \cong \left\langle \mathcal A_{i\to i}\right\rangle \ast \left\langle \mathcal A_{i\to V \to i}\right\rangle.
				\]
				\item The free algebra $\mathcal D := \left\langle (\mathcal A_{1\to 1}\cup \mathcal A_{1\to V \to 1}) \cap (\mathcal A_{2\to 2} \cup \mathcal A_{2\to V \to 2})\right\rangle$ becomes a dg-algebra when equipped with the differential of $CE^\ast(\partial \varLambda_V;V)$ and there is a quasi-isomorphism of dg-algebras
				\[
					\mathcal D \cong CE^\ast(\partial \varLambda_V;V).
				\]
				\item The free algebra
				\[
					\mathcal C := CE^\ast(\varLambda_1,P_1;X_1) \ast_{CE^\ast(\partial \varLambda_V;V)} CE^\ast(\varLambda_2,P_2;X_2) \ast \left\langle \mathcal A_{1\to V \to 2}\right\rangle \ast \left\langle \mathcal A_{2\to V \to 1}\right\rangle
				\]
				becomes a dg-algebra when equipped with the differential of $CE^\ast(\varLambda_i,P_i;X_i)$ such that $CE^\ast(\varLambda_1,P_1;X_1), CE^\ast(\varLambda_2,P_2;X_2) \subset \mathcal C$ are dg-subalgebras. Moreover there is a quasi-isomorphism of dg-algebras
					\[
						\mathcal C \cong CE^\ast(\varLambda,P;X).
					\]
			\end{enumerate}
		\end{lma}
		\begin{proof}
			\begin{enumerate}
				\item The claim that $\left\langle \mathcal A_{i\to i}\right\rangle \ast \left\langle \mathcal A_{i\to V \to i}\right\rangle$ becomes a dg-algebra when equipped with the differential of $CE^\ast(\varLambda_i,P_i;X_i)$ follows by considering the filtered dg-algebras by introducing an action bound like in the proof of \cref{thm:gluing_formula_hypersurface_dga}. With a given action bound on generators and for sufficiently thin handles as there, we get a bijection of the generator sets and hence we can equip both algebras with the differential of $CE^\ast(\varLambda_i,P_i;X_i)$. By taking the colimit as the action bound goes to $\infty$ and by invariance of Chekanov--Eliashberg dg-algebras like in the proof of \cref{thm:gluing_formula_hypersurface_dga}, the result then follows by \eqref{eq:amalg_free_prod}.
				\item This is similar to item (1) above. The set $(\mathcal A_{1\to 1}\cup \mathcal A_{1\to V \to 1}) \cap (\mathcal A_{2\to 2} \cup \mathcal A_{2\to V \to 2})$ precisely consists of words of Reeb chords $\partial \varLambda_V \to \partial \ell_V \to \cdots \to \partial \ell_V \to \partial \varLambda_V$ which corresponds to the generators of $CE^\ast(\partial \varLambda_V;V)$ below a certain action bound, see \cref{lma:surgery_desc_generators_relative} and \cite[Proposition 6.3]{asplund2020chekanov}.
				\item This is also similar to item (1) above. The extra thing we need to be mindful of is to ensure that $CE^\ast(\varLambda_1,P_1;X_1)$ and $CE^\ast(\varLambda_2,P_2;X_2)$ become dg-subalgebras of $\mathcal C$. This however, follows from \cite[Lemma 2.26]{asplund2021simplicial} (see also \cite[Lemma 3.5]{asplund2020chekanov}). Namely, the differential not have input at a generator of $CE^\ast(\varLambda_1,P_1;X_1)$ and a generator of $CE^\ast(\varLambda_2,P_2;X_2)$ as output (or vice versa), which means that both $CE^\ast(\varLambda_1,P_1;X_1)$ and $CE^\ast(\varLambda_2,P_2;X_2)$ become dg-subalgebras of $\mathcal C$. Then the result follows by \eqref{eq:amalg_free_prod}.
			\end{enumerate}
		\end{proof}
		\begin{rmk}\label{rmk:quasi_iso_free_prods}
			\begin{enumerate}
				\item Note that neither $\left\langle \mathcal A_{i\to i}\right\rangle$ nor $\left\langle \mathcal A_{i\to V \to i}\right\rangle$ becomes a dg-algebra when equipped with the differential of $CE^\ast(\varLambda_i,P_i;X_i)$, but their free algebra does. We do not know how to express $CE^\ast(\varLambda_i,P_i;X_i)$ as the colimit of \emph{dg-algebras} in general.

				Similarly note that neither $\left\langle \mathcal A_{1\to V \to 2}\right\rangle$ nor $\left\langle \mathcal A_{2\to V \to 1}\right\rangle$ is a dg-algebra when equipped with the differential from $CE^\ast(\varLambda,P;X)$. We also do not know how to express $CE^\ast(\varLambda,P;X)$ as the colimit of \emph{dg-algebras} in general.
				\item When $V$ is assumed to be subcritical we have $\mathcal A_{1\to V \to 2} = \mathcal A_{2\to V \to 1} = \varnothing$, and then \cref{lma:quasi_iso_free_products} gives a special case of \cref{thm:gluing_relative_legendrians} when $C = C' = \varDelta^1$. In this case it \emph{is} the case that $CE^\ast(\varLambda,P;X)$ is the colimit of dg-algebras.
				\item Additionally, when $P = \varnothing$, and $V$ is assumed to be subcritical, then \cref{lma:quasi_iso_free_products} recovers \cite[Theorem 1.3]{asplund2020chekanov} and a special case of \cite[Theorem 1.1]{asplund2021simplicial}.
			\end{enumerate}
		\end{rmk}
		\begin{rmk}\label{rmk:category_2}
			The free product and amalgamated free products used in \cref{lma:quasi_iso_free_products} are both special cases of colimits in the category $\mathbf{a}$ which consists of associative, non-commutative, non-unital algebras over varying non-unital rings. Namely, there is a forgetful functor $\mathbf{dga} \to \mathbf{a}$ by forgetting the differential and grading, which means that all colimits are preserved because forgetful functors are left adjoints. The free product is the coproduct in $\mathbf a$ and the amalgamated free product is the pushout in $\mathbf a$. Colimits of free algebras in this category are again free algebras, generated by the union of the generators.
		\end{rmk}
\section{Tangle contact homology and gluing formulas}\label{sec:tch_and_gluing_formulas}
	In this section we apply the results of \cref{sec:simplicial_decomps_for_pairs} to define tangle contact homology and obtain gluing results for knot contact homology.

	From now on and throughout the paper we let $\R^3_{x \geq 0} := \left\{(x,y,z) \in \R^3 \suchthat x\geq 0\right\}$.
	\subsection{Unframed knot contact homology}\label{sec:unframed_kch_for_tangles}
		In this section we first discuss the \emph{unframed} knot contact homology and how the results of \cite{asplund2020chekanov,asplund2021simplicial} yields gluing formulas for the unframed version of knot contact homology.
		\begin{dfn}[Tangle]
			An \emph{$r$-component tangle} in $\R^3_{x \geq 0}$ is a proper orientation preserving embedding 
			\[
				T \colon \left\{1,\ldots,r\right\} \times [-1,1] \longrightarrow \R^3_{x \geq 0}
			\]
			such that $T(k, \pm 1) \subset \partial \R^3_{x \geq 0}$ for all $k\in \left\{1,\ldots,r\right\}$. We call $\bigcup_{k=1}^r\{T(k,1),T(k,-1)\}$ the boundary of $T$ and denote it by $\partial T$.

			By abuse of notation we often also refer to the image of $T$ as an $r$-component tangle.
		\end{dfn}
		\begin{rmk}
			\begin{enumerate}
				\item In general we also allow components of a tangle to have empty boundary, but for notational simplicity we assume that each component of a tangle has boundary.
				\item We will always (even if it is not apparent in the notation) assume that tangles are oriented and framed submanifolds of $\R^3_{x\geq 0}$.
			\end{enumerate}
		\end{rmk}
		\begin{dfn}[Ambient isotopy]\label{dfn:equiv_of_tangles}
			We say that two tangles $T_1$ and $T_2$ are \emph{ambient isotopic}, denoted by $T_1 \simeq T_2$, if there exists an orientation preserving isotopy $h_t \colon \R^3_{x \geq 0} \longrightarrow \R^3_{x \geq 0}$ such that
			\begin{enumerate}
				\item $h_0 = \id$, and
				\item $\eval[0]{h_1}_{T_1} = T_2$.
			\end{enumerate}
		\end{dfn}
		\begin{dfn}[Trivial tangle]\label{dfn:trivial_tangle}
			We say that an $r$-component tangle $T$ in $\R^3_{\geq 0}$ is \emph{trivial} if it bounds $r$ disjoint embedded half-disks in $\R^3_{x\geq 0}$ whose boundary arcs belong to $\partial \R^3_{x \geq 0}$.
		\end{dfn}
		\begin{rmk}
			Be aware that our notion of a \emph{trivial tangle} may not be standard terminology. It is equivalent to there existing a projection in which the tangle diagram does not have any crossings.
		\end{rmk}
		\begin{asm}\label{asm:orthogonality_chords}
			\begin{enumerate}
				\item Tangles meet $\partial \R^3_{x \geq 0}$ orthogonally.
				\item The binormal geodesic chords of a tangle in $\R^3_{x\geq 0}$ are isolated.
				\item No binormal geodesic chord of a tangle in $\R^3_{x\geq 0}$ is contained in $\partial \R^3_{x\geq 0}$.
			\end{enumerate}
		\end{asm}
		\begin{rmk}
			Note that \cref{asm:orthogonality_chords} does not lose generality. It can always be achieved by small perturbations of the tangle and the given metric on $\R^3_{x\geq 0}$.
		\end{rmk}
		\begin{dfn}[Gluing of tangles]
			Let $T_1$ and $T_2$ be two oriented $r$-component tangles in $\R^3_{x \geq 0}$ such that $\partial T_1 = \partial T_2$. The gluing of $T_1$ and $T_2$ is defined to be the oriented link obtained by gluing the two copies of $\R^3_{x \geq 0}$ together along their common boundary by the identity map.
		\end{dfn}
		\begin{rmk}
			We will always assume that whenever we glue two tangles, their orientations are compatible so that the resulting link obtains an induced orientation.
		\end{rmk}

		\begin{dfn}[Knot contact homology following Ekholm--Etnyre--Ng--Sullivan]\label{dfn:kch_eens}
			The \emph{knot contact homology} of an $r$-component oriented link is defined as
			\[
				KCC^\ast_{\text{EENS}}(K) := CE^\ast(\varLambda_K, \boldsymbol k[\lambda_1^{\pm 1}, \mu_1^{\pm 1},\ldots, \lambda_r^{\pm 1}, \mu_r^{\pm 1}]; T^\ast \R^3),
			\]
			where $\lambda_1, \mu_1, \ldots, \lambda_r, \mu_r$ is a basis of $H_1(\varLambda_K)$ consisting of longitude classes and meridian classes, respectively.
		\end{dfn}
		\begin{rmk}
			We consider the \emph{fully non-commutative} version of knot contact homology in which the generators $\lambda_i^{\pm 1}$ and $\mu_i^{\pm 1}$ for $i\in \left\{1,\ldots,r\right\}$ in \cref{dfn:kch_eens} are declared to \emph{not} commute with Reeb chord generators.
		\end{rmk}

		\begin{dfn}[Unframed knot contact homology of a link]
			For an oriented link $K \subset \R^3$ we define 
			\[
				KCC_u^\ast(K) := CE^\ast(\varLambda_K;T^\ast \R^3) = \eval[0]{KCC^\ast_{\text{EENS}}(K)}_{\lambda_1 = \cdots \lambda_r = \mu_1 = \cdots \mu_r = 1}.
			\]
		\end{dfn}
		\begin{rmk}\label{rmk:unframed_framed}
			Note that the adjective ``unframed'' does not refer to the link $K$, but is referring to the coefficient ring. Even though we call $KCC^\ast_u$ the \emph{unframed} knot contact homology, $K$ is always assumed to be oriented, as we need the Legendrian submanifold $\varLambda_K$ to be oriented in order to consider signs in the differential of the Chekanov--Eliashberg dg-algebra.
		\end{rmk}
		Let $T$ be a tangle in $\R^3_{x \geq 0}$ and let $H := \partial \R^3_{x \geq 0}$. After convex completion of the (open) Weinstein sector $T^\ast \R^3_{x \geq 0}$ we obtain the Weinstein pair $(\widetilde{T^\ast \R^3_{x \geq 0}}, T^\ast H)$ \cite[Section 2.7]{ganatra2020covariantly}. Note that in our notation $\widetilde{T^\ast \R^3_{x \geq 0}}$ is still an open Weinstein sector. We convex complete along the horizontal boundary $\partial T^\ast \R^3_{x \geq 0} = T^\ast \eval[0]{\R^3}_{x=0}$ only.

		The unit conormal bundle of $T$ is a Legendrian submanifold $\varLambda_T \subset \partial_\infty \widetilde{T^\ast \R^3_{x \geq 0}}$ relative to $T^\ast H$ in the sense of \cref{dfn:relative_leg}.
		\begin{dfn}[Unframed tangle contact homology algebra]
			Let $T$ be an oriented tangle in $\R^3_{x \geq 0}$. We define the \emph{unframed tangle contact homology algebra} of $T$ as
			\[
				KCC_u^\ast(T) := CE^\ast(\varLambda_T, T^\ast H; \widetilde{T^\ast \R^3_{x \geq 0}}).
			\]
			The homology of the dg-algebra $KCC^\ast_{u}(T)$ is called \emph{unframed tangle contact homology}.
		\end{dfn}
		\begin{rmk}
			Note that the adjective \emph{unframed} does not refer to the framing or orientation of the tangle $T$ (cf.\@ \cref{rmk:unframed_framed}). Instead it refers to whether it is the Chekanov--Eliashberg dg-algebra of $\varLambda_T$ directly, or the Chekanov--Eliashberg dg-algebra of a cotangent neighborhood of $\varLambda_T$ with a choice of handle decomposition. The presence of such a handle decomposition can be interpreted as keeping track of framing data in the actual dg-algebra.
		\end{rmk}
		\begin{lma}\label{lma:subalgebra_kch_u}
			There is a dg-subalgebra of $KCC_u^\ast(T)$ which is canonically quasi-isomorphic to $KCC_u^\ast(\partial T)$.
		\end{lma}
		\begin{proof}
			By definition $KCC_u^\ast(T) = CE^\ast(\varLambda_T, T^\ast H; \widetilde{T^\ast \R^3_{x \geq 0}})$, and by \cref{lma:subalgebra_of_relative_dga} this dg-algebra contains a dg-subalgebra that is generated by Reeb chords of $\varSigma_{\varLambda_T}$ (see \cref{dfn:ce_of_rel_leg}) that corresponds to Reeb chords of $\partial \varLambda_T \subset ST^\ast H$. This dg-subalgebra is canonically quasi-isomorphic to $CE^\ast(\partial \varLambda_T, T^\ast H)$, and since $\partial \varLambda_T = \varLambda_{\partial T}$ in $ST^\ast H$ the result follows.
		\end{proof}
		\begin{lma}
			For generic choices of metrics and almost complex structures, the quasi-isomorphism class of the dg-algebra $KCC^\ast_u(T,h)$ is an ambient isotopy invariant of $T \subset \R^3_{x\geq 0}$.
		\end{lma}
		\begin{proof}
			Ambient isotopies of $T$ induces a Legendrian isotopy $\varLambda_T \subset \partial \widetilde{T^\ast \R^3_{x\geq 0}}$ such that its restriction to the boundary $\partial \varLambda_T \subset ST^\ast H$ also defines a Legendrian isotopy from which the result follows by the discussion in \cite[Remark 1.5]{asplund2020chekanov} (cf.\@ the proof of \cref{lma:kch_invariance}(3) below).
		\end{proof}
	\subsection{Knot contact homology for links}\label{sec:kch_links}
		In this section we extend the definition of the fully non-commutative knot contact homology to a family of such homologies, depending on a choice of a handle decomposition $h$ of a small cotangent neighborhood of the unit conormal torus in $ST^\ast \R^3$. For certain choices of $h$ this definition recovers the Ekholm--Etnyre--Ng--Sullivan fully non-commutative knot contact homology defined in \cite{ekholm2013knot,cieliebak2017knot,ekholm2018a}.

		\begin{dfn}[Knot contact homology for links]
			Let $K \subset \R^3$ be an oriented link. Define
			\[
				KCC^\ast(K,h) := CE^\ast((N(\varLambda_K),h);T^\ast \R^3),
			\]
			where $N(\varLambda_K) \cong D_\varepsilon T^\ast \varLambda_K$ is a small cotangent neighborhood of the unit conormal bundle of $K$, together with a chosen handle decomposition $h$.
		\end{dfn}
		In general we have that the quasi-isomorphism class of $KCC^\ast(K,h)$ depends on the handle decomposition $h$ as can be observed by \cite[Example 7.4 and Example 7.5]{asplund2020chekanov}.
		\begin{lma}\label{lma:kch_invariance}
			Let $h_1$ and $h_2$ be two handle decompositions of $N(\varLambda_K)$.
			\begin{enumerate}
				\item $KCC^\ast(K,h_1)$ and $KCC^\ast(K,h_2)$ are derived Morita equivalent.
				\item There is an isomorphism of algebras $HH_\ast(KCC^\ast(K,h_1)) \cong HH_\ast(KCC^\ast(K,h_2))$, where $HH_\ast$ denotes Hochschild homology.
				\item Denote the subcollection of top handles of $h_1$ and $h_2$ by $h^{\mathrm{top}}_1$ and $h^{\mathrm{top}}_2$ respectively. If the core disks in $h^{\mathrm{top}}_1$ and $h^{\mathrm{top}}_2$ are related by a Legendrian isotopy $h^{\mathrm{top}}_t$ in $ST^\ast \R^3$ with the property that $\partial h^{\mathrm{top}}_t$ is a Legendrian isotopy of $\partial h^{\mathrm{top}}_1$ and $\partial h^{\mathrm{top}}_2$ in $\partial N(\varLambda_K)_0$ then we have that $KCC^\ast(K,h_1)$ and $KCC^\ast(K,h_2)$ are dg-homotopy equivalent.
			\end{enumerate}
		\end{lma}
		\begin{proof}
			The proof of (1) and (2) were sketched in \cite[Remark 1.5]{asplund2020chekanov}. For the sake of completeness we summarize the proofs below.
			\begin{enumerate}
				\item By the surgery formula \cite[Theorem 1.1]{asplund2020chekanov} we have a quasi-isomorphism
				\[
					CE^\ast((N(\varLambda_K),h);T^\ast \R^3) \cong CW^\ast(C(h);(T^\ast \R^3)_{\varLambda_K}),
				\]
				where $(T^\ast \R^3)_{\varLambda_K}$ is $T^\ast \R^3$ stopped at $\varLambda_K$ and where $C(h)$ is the union of cocore disks of the top handle of $(T^\ast \R^3)_{\varLambda_K}$ depending on the handle decomposition $h$ of $N(\varLambda_K)$. The wrapped Fukaya category $\mathcal W((T^\ast \R^3)_{\varLambda_K})$ is generated by the summands of $C(h)$ by \cite[Theorem 1.1]{chantraine2017geometric} and \cite[Theorem 1.13]{ganatra2022sectorial}. As the derived module category is independent of choice of a generating set, the result follows.
				\item Attachment of the $1$-simplex handle $N(\varLambda_K)_0 \times D_\varepsilon T^\ast \varDelta^1$ gives rise to a natural geometric Weinstein cobordism with positive contact boundary $\partial (T^\ast \R^3)_{\varLambda_K}$ and negative contact boundary $\partial T^\ast \R^3$. By \cite[Theorem 5.6]{bourgeois2012effect} we have a quasi-isomorphism
				\[
					SH^\ast((T^\ast \R^3)_{\varLambda_K}) \cong HH_\ast(CE^\ast((N(\varLambda_K),h); T^\ast \R^3)) \oplus SH^\ast(T^\ast \R^3).
				\]
				Since $SH^\ast((T^\ast \R^3)_{\varLambda_K})$ and $SH^\ast(T^\ast \R^3)$ does not depend on the handle decomposition $h$, it follows that the Hochschild homology of the Chekanov--Eliashberg dg-algebra also does not.
				\item Such Legendrian isotopies of $h^{\text{top}}_1$ and $h^{\text{top}}_2$ induces a Legendrian isotopy $\varSigma(h_1) \simeq \varSigma(h_2)$ (cf.\@ \cite[Remark 1.5]{asplund2020chekanov}) and the result follows from the invariance of the Chekanov--Eliashberg dg-algebra up to dg-homotopy equivalent under Legendrian isotopy, see \cite[Appendix A]{ekholm2015legendrian} and \cite[Corollary 5.17]{ekholm2017symplectic}.
			\end{enumerate}
		\end{proof}
		\begin{lma}
			For generic choices of metrics and almost complex structures, the dg-homotopy equivalence class of the dg-algebra $KCC^\ast(K,h)$ is an ambient isotopy invariant of $K \subset \R^3$
		\end{lma}
		\begin{proof}
			Ambient isotopies of $K$ induces Legendrian isotopies of $\varLambda_K$ which induces isotopies of Weinstein hypersurfaces $N(\varLambda_K) \hookrightarrow \partial T^\ast \R^3$. Such isotopies gives Legendrian isotopies as in \cref{lma:kch_invariance}(3). Note in particular that the handle decomposition $h$ of $N(\varLambda_K)$ is unchanged under isotopies of $K$ as it is only the Weinstein embedding $N(\varLambda_K) \hookrightarrow \partial T^\ast \R^3$ that changes.
		\end{proof}
		\begin{thm}\label{thm:kch_and_kch_eens_single_top_handle}
			Let $K \subset \R^3$ be an oriented link. If $h$ is a handle decomposition of $N(\varLambda_K)$ such that each component of $N(\varLambda_K)$ has a single top handle, then for generic choices of metrics and almost complex structures there is a quasi-isomorphism
			\[
				KCC^\ast(K,h) \cong KCC_{\mathrm{EENS}}^\ast(K).
			\]
		\end{thm}
		\begin{proof}
			By \cref{thm:loop_space_dga} it follows that $KCC^\ast(K,h) \cong CE^\ast \left(\varLambda_K, \bigoplus_{i}C_{-\ast}(\varOmega \varLambda_{K_i}); T^\ast \R^3\right)$. Since $\varLambda_{K_i} \cong T^2$ for each component $K_i$ of $K$ it follows that we have a quasi-isomorphism $C_{-\ast}(\varOmega \varLambda_{K_i}) \cong \boldsymbol k[H_1(\varLambda_{K_i})]$ for each $i$. We obtain the fully non-commutative knot contact homology by definition of the Chekanov--Eliashberg dg-algebra with loop space coefficients as in \cite{ekholm2017duality} in which generators corresponding to chains of based loops do not commute with Reeb chord generators. The result follows. 
		\end{proof}
	\subsection{Tangle contact homology}\label{sec:kch_tangles}
		Given a tangle $T$ in $\R^3_{x\geq 0}$, its unit conormal bundle $\varLambda_T$ is a Legendrian relative to the Weinstein hypersurface $T^\ast H \hookrightarrow \partial \widetilde{T^\ast \R^3_{x\geq 0}}$ in the sense of \cref{dfn:relative_leg} where $H := \partial \R^3_{x\geq 0}$. We consider handle decompositions $h$ of $N(\varLambda_T)$ as induced by a Morse function on $\varLambda_T$ treated as a compact manifold-with-boundary, cf\@ \cref{ex:handle_decomp_sector}.
		
		Let $V_T$ denote the subcritical part of $N(\varLambda_T)$ with respect to the handle decomposition $h$ of $N(\varLambda_T)$, and let $\ell_T$ be the union of core disks of the top handles in $h$. We have that $V_T$ is a Weinstein subsector of $\R^3_{x\geq 0}$ which meets $T^\ast H$ along $Q_T := V_T \cap ST^\ast H$ since we assume that $T$ intersect $H$ orthogonally, see \cref{asm:orthogonality_chords}. After convex completion we have that $\ell_T$ is a Legendrian submanifold relative to the Weinstein pair $(\widetilde{T^\ast \R^3_{x\geq 0}}, P_T)$ in the sense of \cref{sec:relative_legendrians}, where $P_T$ is the Weinstein hypersurface $P_T := \widetilde{V_T } \#_{Q_T} T^\ast H \hookrightarrow \partial_\infty \widetilde{T^\ast \R^3_{x \geq 0}}$.
		\begin{dfn}[Tangle contact homology]\label{dfn:tch}
			Let $T$ be an oriented tangle in $\R^3_{x\geq 0}$ and let $h$ be a handle decomposition of $N(\varLambda_T)$. We define the \emph{tangle contact homology algebra of $T$} as
			\[
				KCC^\ast(T,h) := CE^\ast(\ell_T, P_T; \widetilde{T^\ast \R^3_{x\geq 0}}).
			\]
		\end{dfn}
		\begin{lma}
			Let $h$ be a handle decomposition of $N(\varLambda_T)$. For generic choices of metrics and almost complex structures, the quasi-isomorphism class of the dg-algebra $KCC^\ast(T,h)$ is an ambient isotopy invariant of $T \subset \R^3_{x\geq 0}$.
		\end{lma}
		\begin{proof}
			Ambient isotopies of $T$ induces Legendrian isotopies of the relative Legendrian submanifold $\ell_T$ in the sense of \cref{lma:kch_invariance}(3). Note in particular that the handle decomposition of $N(\varLambda_T)$ remains fixed and that it is only the Weinstein embedding $N(\varLambda_T) \hookrightarrow \partial \widetilde{T^\ast \R^3_{x\geq 0}}$ that changes. The result follows.
		\end{proof}
	
		\begin{lma}\label{lma:subalgebra_kch}
			Let $T$ be an oriented  tangle in $\R^3_{x\geq 0}$ and let $h$ be a handle decomposition of $N(\varLambda_T)$. There is a dg-subalgebra of $KCC^\ast(T,h)$ which is canonically quasi-isomorphic to $KCC^\ast(\partial T, h_\partial)$.
		\end{lma}
		\begin{proof}
			First we note that by definition of $KCC^\ast(T,h)$ and by \cref{lma:subalgebra_of_relative_dga} there is a dg-subalgebra of $KCC^\ast(T,h)$ canonically quasi-isomorphic to $CE^\ast(\partial \ell_T; P_T)$ where
			\[
				\partial \ell_T \simeq \eval[0]{\partial \ell_T}_{\partial \widetilde{V_T}} \sqcup \underbrace{(\partial \ell_T \cap \partial Q_T \times \varDelta^1)}_{=: \ell_{\partial T}} \sqcup \eval[0]{\partial \ell_T}_{ST^\ast H}.
			\]
			By construction we have $\partial \ell_T \subset \partial(\widetilde{V_T} \#_{Q_T} T^\ast H)$. By construction $Q_T \hookrightarrow \partial V_T, \partial T^\ast H$ is subcritical. Thus by the same proof as \cite[Theorem 4.6]{asplund2021simplicial} $CE^\ast(\partial \ell_T; P_T)$ is quasi-isomorphic to the colimit of the following diagram
			\[
				\begin{tikzcd}[row sep=scriptsize, column sep=scriptsize]
					CE^\ast(\ell_{\partial T}, Q_T) \rar \dar & CE^\ast(\eval[0]{\partial \ell_T}_{ST^\ast H}, Q_T; T^\ast H)\\
					CE^\ast(\eval[0]{\partial \ell_T}_{\partial \widetilde{V_T}}, Q_T; \widetilde{V_T})&
				\end{tikzcd}.
			\]
			In particular we note that by definition $CE^\ast(\eval[0]{\partial \ell_T}_{ST^\ast H}, Q_T; T^\ast H) = KCC^\ast(\partial T, h_\partial)$, and this is a dg-subalgebra of $CE^\ast(\partial \ell_T; P_T)$ and hence of $KCC^\ast(T, h)$.
		\end{proof}
	\subsection{Gluing formula recovering knot contact homology}\label{sec:gluing_formula_kch}
		Let $T_1$ and $T_2$ be two oriented tangles in $\R^3_{x\geq 0}$ with $\partial T_1 = \partial T_2$ and such that they glue up to an oriented link in $\R^3$. Let $h_1$ and $h_2$ be handle decompositions of $N(\varLambda_{T_1})$ and $N(\varLambda_{T_2})$, respectively, such that the handle decompositions $(h_1)_\partial$ and $(h_2)_\partial$ of the boundary $N(\partial \varLambda_T)$ agree (see \cref{dfn:handle_decomp_pair} and \cref{ex:handle_decomp_sector}). Suppose that the gluing of the two tangles $T_1$ and $T_2$ gives a link $K$. This induces a gluing of $N(\varLambda_{T_1})$ and $N(\varLambda_{T_2})$ such that it respects the handle decompositions $h_1$ and $h_2$ in the sense that the boundary critical points along $N(\partial \varLambda_T)$ match up after the gluing. This gives an induced handle decomposition on $N(\varLambda_K)$ which we denoted by $h_1 \# h_2$, and we call such $h_1$ and $h_2$ \emph{gluing compatible}.

		Let $h_H$ be a handle decomposition with a single top handle of $T^\ast H \cong T^\ast \R^2$ viewed as $B^4$ stopped at the unknot. Let $\ell_H$ denote the $2$-dimensional core disk in $h_H$ and let $\varSigma_H(h_H)$ denote the $2$-dimensional attaching sphere of the handle $T^\ast H \times T^\ast \varDelta^1$ in $\left(\widetilde{T^\ast \R^3_{x\geq 0}}\right)_{T^\ast H}$, used in the definition of $KCC^\ast(T_i,h_i)$. We use the notation $\partial \ell_H^1$ to denote the $1$-dimensional core disk of a handle decomposition of $T^\ast(\text{unknot})$ such that by attaching a stop on the unknot using this handle decomposition, the top attaching sphere of $T^\ast \R^2$ becomes exactly $\partial \ell$, see \cref{fig:core_disk_h}.

		\begin{figure}[!htb]
			\centering
			\includegraphics{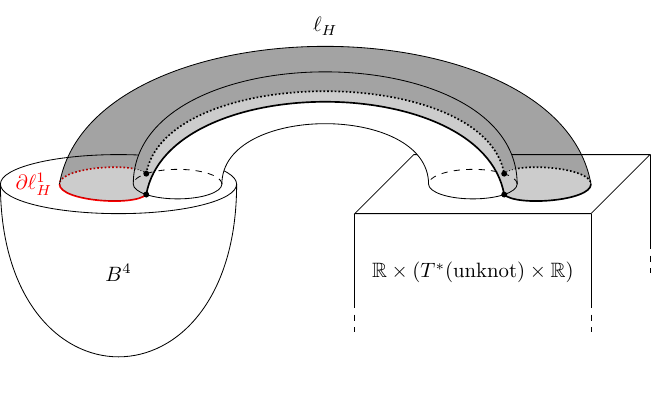}
			\caption{The $2$-dimensional core disk $\ell_H$ of the top handle of $T^\ast \R^2$ viewed as $B^4$ stopped at the unknot, with the $1$-dimensional core disk $\partial \ell_H^1$ in red.}\label{fig:core_disk_h}
		\end{figure}

		Applying \cref{lma:surgery_desc_generators_relative} we first describe the dg-algebras $KCC^\ast(T_i,h_i) := CE^\ast(\ell_{T_i}, P_{T_i}; \widetilde{T^\ast \R^3_{x\geq 0}})$. Let $\varSigma_{T_i}(h_i) = \ell_{T_i} \sqcup_{\partial \ell_{T_i} \times \left\{-1\right\}} (\partial \ell_{T_i} \times \varDelta^1) \sqcup_{\partial \ell_{T_i} \times \left\{1\right\}} (\left\{0\right\} \times \partial \ell_{T_i} \times \R) \subset \partial \left(\widetilde{T^\ast \R^3_{x\geq 0}}\right)_{P_{T_i}}$ be the non-compact Legendrian as in \cref{dfn:ce_of_rel_leg}. Let $\partial \ell_{T_i}^{H} := \partial \ell_{T_i} \cap B^4$, where $B^4$ is the subcritical part of $T^\ast H$, see \cref{fig:gen_with_hypersurface}.

		\begin{figure}[!htb]
			\centering
			\includegraphics{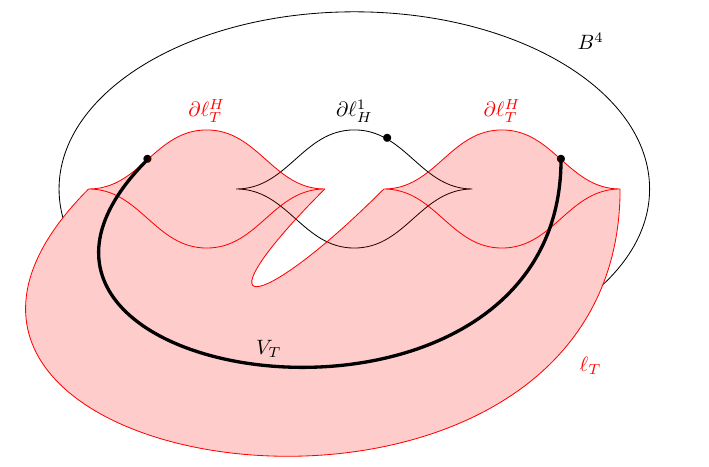}
			\caption{The figure shows the core disk $\ell_T$ of a handle decomposition $h$ of $N(\varLambda_T)$ with subcritical part $V_T \cong D^2 \times I$.}\label{fig:gen_with_hypersurface}
		\end{figure}
		\begin{lma}\label{lma:kcc_tangle_with_hypersurface}
			Fix a generic choice of metric and almost complex structure.
			\begin{enumerate}
				\item The dg-algebra $KCC^\ast(T,h)$ is quasi-isomorphic to the dg-algebra generated by composable words of Reeb chords of the form $\varSigma_{T}(h) \to \varSigma_{T}(h)$ and $\varSigma_{T}(h) \to \varSigma_H(h_H) \to \cdots \to \varSigma_H(h_H) \to \varSigma_{T}(h)$, and where the differential is the one induced from the differential of the Chekanov--Eliashberg dg-algebra of $\varSigma_{T}(h) \cup \varSigma_H(h_H) \subset \partial \left(\widetilde{T^\ast \R^3_{x\geq 0}}\right)_{P_{T,0}}$.
				\item The dg-subalgebra $KCC^\ast(\partial T, h_\partial) \subset KCC^\ast(T_i, h_i)$ is quasi-isomorphic to the dg-algebra generated by composable words of Reeb chords of the form $\partial \ell_{T_i}^H \to \partial \ell_{T_i}^H$ and $\partial \ell_{T_i}^H \to \partial \ell_H^1 \to \cdots \to \partial \ell_H^1 \to \partial \ell_{T_i}^H$ and where the differential is the one induced from the differential of the Chekanov--Eliashberg dg-algebra of $\partial \ell_{T_i}^H \cup \partial \ell_H^1 \subset \partial P_{T_i}$.
			\end{enumerate}
		\end{lma}
		\begin{proof}
			\begin{enumerate}
				\item This follows from \cref{lma:quasi_iso_free_products}(1).
				\item This follows from \cref{lma:quasi_iso_free_products}(2).
			\end{enumerate}
		\end{proof}

		After gluing two oriented tangles $T_1,T_2 \subset \R^3_{x\geq 0}$ to an oriented link $K \subset \R^3$ we obtain the following description of $KCC^\ast(K,h)$ in terms of $KCC^\ast(T_i,h_i)$ and $KCC^\ast(\partial T, h_\partial)$. After gluing the two pairs $(\widetilde{T^\ast \R^3_{x\geq 0}}, P_{T_1})$ and $(\widetilde{T^\ast \R^3_{x\geq 0}}, P_{T_2})$ along their common Weinstein subhypersurface $T^\ast H$ we obtain the Weinstein pair $(T^\ast \R^3, P)$ where $P := \widetilde V_{T_1} \#_{Q_T} \widetilde V_{T_2}$ with notation as in \cref{sec:kch_tangles}. Let $\varSigma_P(h)$ denote the union of the top attaching spheres of the handle $T^\ast H \times T^\ast \varDelta^1$ used in the construction of the pair $(T^\ast \R^3, P)$. Let $\ell_K$ denote the result of gluing $\ell_{T_1}$ and $\ell_{T_2}$ along their common boundary in $\partial B^4$ over the handle $B^4 \times T^\ast \varDelta^1$. Furthermore let $\varSigma_P(h)$ and $\varSigma_K(h_1 \# h_2)$ denote the corresponding Legendrians after stopping $T^\ast \R^3$ at $P$.
		\begin{lma}\label{lma:gluing_presence_of_hypersurface_chords}
			Let $\mathcal K$ denote the dg-algebra generated by composable words of Reeb chords of the form $\varSigma_K(h_1 \# h_2) \to \varSigma_K(h_1\# h_2)$ and $\varSigma_K(h_1\# h_2) \to \varSigma_P(h) \to \cdots \to \varSigma_P(h) \to \varSigma_K(h_1\# h_2)$. For generic choices of metrics and almost complex structures $\mathcal K$ equipped with the differential of the Chekanov--Eliashberg dg-algebra of $\varSigma_K(h_1\# h_2) \cup \varSigma_P(h) \subset \partial (T^\ast \R^3)_{P_0}$ is a dg-algebra such that $KCC^\ast(T_1,h_1), KCC^\ast(T_2,h_2) \subset \mathcal K$ are dg-subalgebras. Moreover, there is a quasi-isomorphism of dg-algebras
			\[
				\mathcal K \cong KCC^\ast(K, h_1\# h_2).
			\]
		\end{lma}
		\begin{proof}
			This follows from \cref{lma:surgery_desc_generators_relative}(3).
		\end{proof}

		We proceed to give a topological description of some of the generators appearing in \cref{lma:kcc_tangle_with_hypersurface}.
		\begin{lma}\label{lma:tcc_generators_hypersurface}
			Let $T \subset \R^3_{x\geq 0}$ be a tangle and let $h$ be a handle decomposition of $N(\varLambda_T)$.
			\begin{enumerate}
				\item Reeb chords of the form $\ell_T \to \ell_T$ are in one-to-one correspondence with oriented binormal geodesic chords of $T \subset \R^3_{x\geq 0}$,
				\item Composable words of Reeb chords of the form $\partial \ell_T^H \to \partial \ell_T^H$ or $\partial \ell_T^H \to \partial \ell_H^1 \to \cdots \to \partial \ell_H^1 \to \partial \ell_T^H$ are in one-to-one correspondence with oriented binormal geodesic chords of $\partial T \subset H$,
				\item Reeb chords of the form $\ell_H \to \ell_T$ or $\ell_T \to \ell_H$ are in one-to-one correspondence with oriented binormal geodesic chords from $H$ to $T$ and $T$ to $H$ in $\R^3_{x \geq 0}$, respectively.
				\item Reeb chords of the form $\ell_H \to \ell_H$ are in one-to-one correspondence with oriented binormal geodesic $H$ in $\R^3_{\geq 0}$.
			\end{enumerate}
		\end{lma}
		\begin{proof}
			The crucial point for the proof is to understand the Reeb dynamics near $T^\ast H \subset \partial \widetilde{T^\ast \R^3_{x\geq 0}}$ after taking Weinstein connected sum with $\R \times (T^\ast H \times \R)$ as described in \cref{sec:relative_legendrians}, in accordance with \cref{dfn:tch}. The hypersurface $H$ in $\partial \widetilde{T^\ast \R^3_{x\geq 0}}$ is the Legendrian lift of the boundary of $\R^3_{x\geq 0}$.

			Trajectories of the Reeb vector field that enters the surgery region near $T^\ast H \subset \partial \widetilde{T^\ast \R^3_{x\geq 0}}$ and leaves the surgery region, continuing with the flow in $\partial \widetilde{T^\ast \R^3_{x\geq 0}}$ are described as in \cite[Section B.3]{ekholm2017duality} and \cite[Appendix A]{asplund2019fiber}, after choosing a collar neighborhood compatible metric.

			More concretely, in terms of the geometry of $\R^3_{x\geq 0}$, we pick a warped product metric as in \cite[(A.1)]{asplund2019fiber}, on a collar neighborhood of $\partial \R^3_{x\geq 0}$. Any trajectory of the geodesic flow that enters the collar neighborhood will follow a path described as the product of the geodesic flow in $\partial \R^3_{x\geq 0}$, and a parabola with a maximum in the $\R$-coordinate of the collar neighborhood coordinates, which corresponds to the normal component of the trajectory of the geodesic flow in question. The maximum is finite if and only if the tangential component of the geodesic trajectory is non-zero.
			\begin{enumerate}
				\item The Reeb flow in $ST^\ast \R^3_{x \geq 0}$ away from $x = 0$ is the lift of the geodesic flow in $\R^3_{x\geq 0}$, so any oriented binormal geodesic chord of $T \subset \R^3_{x\geq 0}$ that avoids a neighborhood of $H$ corresponds to a Reeb chord of the form $\ell_T \to \ell_T$. We furthermore claim that no Reeb chord $\ell_T \to \ell_T$ enters the surgery region near $T^\ast H$. If any geodesic trajectory starting at $T$ enters the collar neighborhood of $H$ described above orthogonally to $H$, then the geodesic flow will continue up until the boundary and will not return to $T$. If the geodesic trajectory enters non-orthogonally, then it will leave the collar neighborhood non-orthogonally at a certain angle. Possibly after an arbitrarily small perturbation of the metric, the geodesic trajectory will not hit $T$ orthogonally when it returns, by general position.
				\item This follows from \cref{lma:surgery_desc_generators_relative}, and in particular \cite[Proposition 6.2]{asplund2020chekanov}.
				\item This follows from the fact that the Reeb flow in $ST^\ast \R^3_{x \geq 0}$ away from $x = 0$ corresponds to geodesic flow in $\R^3_{x\geq 0}$ and by the same argument as in item (1).
			\end{enumerate}
		\end{proof}
		Let $\mathcal G_{\partial H} = CE^\ast(\partial \ell_H, (B^4)_{D^2})$, see \cref{fig:core_disk_h}. We let $\mathcal G_{T\to H}$ ($\mathcal G_{H\to T}$) denote the right (left) module over the algebra $\mathcal G_{\partial H}$ generated by Reeb chords from $\ell_T$ to $\ell_H$ and $\partial \ell_T^H$ to $\partial \ell_H^1$ ($\ell_H$ to $\ell_T$ and $\partial \ell_H^1$ to $\partial \ell_T^H$). Similarly, let $\mathcal G_{T\to T}$ denote the set of Reeb chords of $\ell_T \subset \partial \widetilde T^\ast \R^3_{x\geq 0}$. We define
		\begin{equation}\label{eq:tensor_module}
			\mathcal G_{T \to H \to T} := \mathcal G_{T\to H} \otimes_{\mathcal G_{\partial H}} \mathcal G_{H\to T},
		\end{equation}
		as modules. In the following we will only consider the underlying sets of $\mathcal G_{T\to T}$, $\mathcal G_{T\to H}$, $\mathcal G_{H\to T}$ and $\mathcal G_{T\to H \to T}$.
		\begin{cor}\label{cor:top_desc_of_generators_tch}
			Let $T \subset \R^3_{x\geq 0}$ be a tangle and let $h$ be a handle decomposition of $N(\varLambda_T)$.
			\begin{enumerate}
				\item Elements of $\mathcal G_{T\to T}$ are in one-to-one correspondence with oriented binormal geodesic chords of $T \subset \R^3_{x\geq 0}$.
				\item Elements of $\mathcal G_{T\to H \to T}$ are in one-to-one correspondence with oriented binormal geodesic chords of $\partial T \subset H$ and words of oriented binormal geodesic chords in $\R^3_{x\geq 0}$ of the form $T \to H \to \cdots \to H \to T$.
			\end{enumerate}
		\end{cor}
		\begin{proof}
			This follows by combining \cref{lma:tcc_generators_hypersurface} and \cref{lma:quasi_iso_free_products}.
		\end{proof}
		\begin{thm}\label{thm:top_desc_tch}
			Let $T \subset \R^3_{x\geq 0}$ be an oriented tangle. Let $h$ be a choice of handle decomposition of $N(\varLambda_T)$. For generic choices of metrics and almost complex structures we have a quasi-isomorphism of dg-algebras
			\[
				KCC^\ast(T,h) \cong \left\langle \mathcal G_{T\to T}\right\rangle \ast \left\langle \mathcal G_{T\to H \to T}\right\rangle.
			\]
		\end{thm}
		\begin{proof}
			This is obtained by combining \cref{cor:top_desc_of_generators_tch} and \cref{lma:quasi_iso_free_products}(1).
		\end{proof}
		\begin{rmk}
			Neither $\mathcal G_{T\to T}$ nor $\mathcal G_{T\to H\to T}$ become dg-algebras when equipped with the differential used in $KCC^\ast(T,h)$, however their free product does, cf.\@ \cref{rmk:quasi_iso_free_prods}(1).
		\end{rmk}
		Now let $T_1$ and $T_2$ be two oriented tangles in $\R^3_{x\geq 0}$. As the discussion surrounding \eqref{eq:tensor_module} we define $\mathcal G_{T_i \to T_i}$, $\mathcal G_{T_i\to H \to T_i}$ in the same way for $i\in \left\{1,2\right\}$. Additionally we define $\mathcal G_H$ to denote the set of Reeb chords of $\ell_H$ in either one of the two copies of $\widetilde{T^\ast \R^3_{x\geq 0}}$. Let $\mathcal G_{T_i \to H}$ (and $\mathcal G_{H\to T_j}$) be the right (left) module over the free algebra on $\mathcal G_H$ generated by Reeb chords from $\ell_{T_i}$ to $\ell_H$ ($\ell_H$ to $\ell_{T_j}$). Then define
		\begin{equation}\label{eq:tensor_module_2}
			\mathcal G_{T_i \to H \to T_j} := \mathcal G_{T_i \to H} \otimes_{\mathcal G_H} \mathcal G_{H\to T_j},
		\end{equation}
		for $i\neq j\in \left\{1,2\right\}$.
		\begin{lma}\label{lma:top_desc_gen_gluing}
			Let $T_1\subset (\R^3_{x\geq 0})_1$ and $T_2 \subset (\R^3_{x\geq 0})_2$ be two oriented tangles. Elements of $\mathcal G_{T_i\to H \to T_j}$ for $i\neq j\in \left\{1,2\right\}$ are in one-to-one correspondence with words of oriented binormal geodesic chords in of the form $T_i \to H \to \cdots \to H \to T_j$, where ``$H \to H$'' is taken to mean an oriented binormal geodesic chord of $H$ in either of the two copies of $\R^3_{x\geq 0}$.
		\end{lma}
		\begin{proof}
			This follows from \cref{lma:gluing_presence_of_hypersurface_chords} in combination with \cref{lma:tcc_generators_hypersurface}(3).
		\end{proof}
		\begin{thm}[\cref{thm:intro_gluing_kch}]\label{thm:gluing_free_product}
			Let $T_1$ and $T_2$ be two tangles in $\R^3_{x\geq 0}$ whose gluing is the link $K \subset \R^3$, where $H := \partial \R^3_{x\geq 0}$. Let $h_1$ and $h_2$ be choices of handle decomposition of $N(\varLambda_{T_1})$ and $N(\varLambda_{T_2})$, respectively. Then for generic choices of metrics and almost complex structures we have a quasi-isomorphism of dg-algebras
			\[
				KCC^\ast(K, h_1 \# h_2) \cong KCC^\ast(T_1,h_1) \ast_{KCC^\ast(\partial T, h_{\partial})} KCC^\ast(T_2,h_2) \ast \left\langle \mathcal G_{T_1 \to H \to T_2}\right\rangle \ast \left\langle \mathcal G_{T_2 \to H \to T_1}\right\rangle,
			\]
			where the right hand side is equipped with the same differential as in $KCC^\ast(K,h_1\#h_2)$.
		\end{thm}
		\begin{proof}
			This follows by combining \cref{lma:top_desc_gen_gluing} and \cref{lma:quasi_iso_free_products}(3). Let us spell out the details. By \cref{lma:gluing_presence_of_hypersurface_chords} and the same discussion regarding Reeb dynamics as in the proof of \cref{thm:top_desc_tch} we see that 
			\begin{align*}
				KCC^\ast(K, h_1 \# h_2) &\cong \left\langle \mathcal G_{T_1 \to T_1} \cup \mathcal G_{T_2 \to T_2} \cup \mathcal G_{T_1 \to H \to T_1} \cup \mathcal G_{T_1 \to H \to T_2} \cup \mathcal G_{T_2 \to H \to T_1} \cup \mathcal G_{T_2 \to H \to T_2} \right\rangle \\
				&= \left\langle \mathcal G_{T_1\to T_1} \cup \mathcal G_{T_1\to H\to T_1} \cup \mathcal G_{T_2\to T_2} \cup \mathcal G_{T_2\to H \to T_2}\right\rangle \ast \left\langle \mathcal G_{T_1 \to H \to T_2}\right\rangle \ast \left\langle \mathcal G_{T_2 \to H \to T_1}\right\rangle.
			\end{align*}
			Now $\mathcal G_{T_1\to H \to T_1}$ and $\mathcal G_{T_1\to H \to T_1}$ has a common subset consisting of words of oriented binormal geodesic chords of $\partial T \subset H$. The dg-algebra generated by such is quasi-isomorphic with $KCC^\ast(\partial T, h_\partial)$ when equipped with the differential of $KCC^\ast(K,h_1\#h_2)$ by \cref{lma:kcc_tangle_with_hypersurface}. Hence
			\begin{align*}
				&\left\langle \mathcal G_{T_1\to T_1} \cup \mathcal G_{T_1\to H\to T_1} \cup \mathcal G_{T_2\to T_2} \cup \mathcal G_{T_2\to H \to T_2}\right\rangle \\
				&\qquad \cong \left\langle \mathcal G_{T_1\to T_1} \cup \mathcal G_{T_1\to H\to T_1}\right\rangle \ast_{KCC^\ast(\partial T, h_\partial)} \left\langle \mathcal G_{T_2\to T_2} \cup \mathcal G_{T_2\to H\to T_2}\right\rangle\\
				&\qquad \overset{\text{\cref{thm:top_desc_tch}}}{\cong}KCC^\ast(T_1,h_1) \ast_{KCC^\ast(\partial T, h_\partial)} KCC^\ast(T_2,h_2),
			\end{align*}
			and the result follows.
		\end{proof}
		\begin{rmk}
			We remark again that neither $\mathcal G_{T_1 \to H \to T_2}$ nor $\mathcal G_{T_2\to H\to T_1}$ become dg-algebras when equipped with the differential in $KCC^\ast(K, h_1\# h_2)$, cf.\@ item(1) of \cref{rmk:quasi_iso_free_prods}.
		\end{rmk}
	\subsubsection{Handle decompositions recovering the Ekholm--Etnyre--Ng--Sullivan knot contact homology}\label{sec:explicit_handle_decomp}
		Let $T_1$ and $T_2$ be two tangles in $\R^3_{x\geq 0}$ whose gluing is the knot $K \subset \R^3$. Using the gluing formula \cref{lma:gluing_presence_of_hypersurface_chords} we can recover $KCC^\ast(K,h_1 \# h_2)$. The purpose of this section is to show that we can choose $h_1$ and $h_2$ so that $h_1 \# h_2$ has a single top handle and hence so that the assumptions of \cref{thm:kch_and_kch_eens_single_top_handle} hold.
		\begin{figure}[!htb]
			\centering
			\raisebox{5ex}{\includegraphics{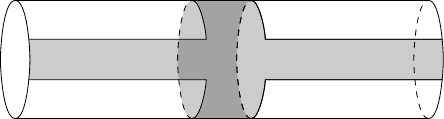}}
			\includegraphics{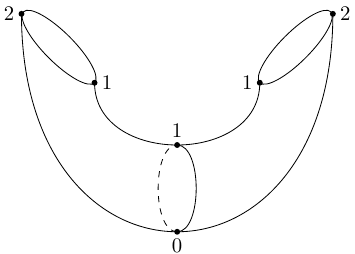}
			\caption{Left: The subcritical part of the handle decomposition of $N(\varLambda_{T_1})$. Right: Morse function with critical points and indices labeled inducing a handle decomposition of $N(\varLambda_{T_1})$.}\label{fig:handle_decomp1}
		\end{figure}
		\begin{figure}[!htb]
			\centering
			\raisebox{3ex}{\includegraphics{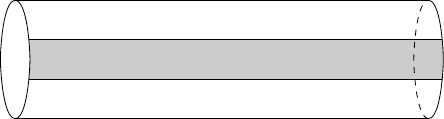}}
			\includegraphics{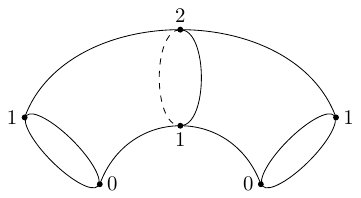}
			\caption{Left: The subcritical part of the handle decomposition of $N(\varLambda_{T_2})$. Right: Morse function with critical points and indices labeled inducing a handle decomposition of $N(\varLambda_{T_2})$.}\label{fig:handle_decomp2}
		\end{figure}

		Let $h_1$ be the handle decomposition induced by a Morse function on $\varLambda_{T_1} \cong S^1 \times I$ with two interior critical points of indices $0$ and $1$ respectively, and two boundary-unstable critical points of indices $1$ and $2$ respectively, on each of the two components of $\partial \varLambda_{T_1}$, see \cref{fig:handle_decomp1}.

		Let $h_2$ be the handle decomposition induced by a Morse function on $\varLambda_{T_2} \cong S^1 \times I$ with two interior critical points of indices $1$ and $2$ respectively, and two boundary-stable critical points of indices $0$ and $1$ respectively, on each of the two components of $\partial \varLambda_{T_2}$, see \cref{fig:handle_decomp2}.
	\subsection{Gluing formula recovering tangle contact homology}
		From now on we assume that we have an oriented tangle $T \subset \R^3_{x\geq 0}$ which can be further decomposed into an oriented tangle $T_2 \subset \R^3_{x\geq 1}$ and an oriented tangle $T_1 \subset \R^2 \times [0,1]$ with $\partial T_1 = T_2|_{\R^2 \times \left\{1\right\}}$ and such that they glue up to the oriented tangle $T$. We call such $T_1$ and $T_2$ \emph{gluing compatible}. In this section we prove a gluing formula for the tangle contact homology of $T$ in terms of the tangle contact homologies of $T_1$ and $T_2$ (and oriented binormal geodesic chords between tangles and hypersurfaces). As this section closely follows the same ideas as \cref{sec:gluing_formula_kch} we are more brief in our description of the setup.

		Consider a tangle $T \subset \R^2 \times [0,1]$. We will use the notation $H := \R^2 \times \left\{0\right\}$, $H' := \R^2 \times \left\{1\right\}$ and $\overline H := H \cup H'$. The tangle contact homology of $T$ together with a handle decomposition $h$ of the unit conormal bundle $\varLambda_T$ is defined in the same way as for a tangle in $\R^3_{x\geq 0}$, except that we use the total boundary $\overline H$ in place of $H$ in \cref{dfn:tch}.

		We let $h_1$ and $h_2$ be handle decompositions of $N(\varLambda_{T_1})$ and $N(\varLambda_{T_2})$, respectively, such that the handle decompositions along the common boundary $N(\partial \varLambda_{T_2}) \subset T^\ast H'$ agree. Like before we call such $h_1$ and $h_2$ \emph{gluing compatible}. The fact that $T = T_1 \#_{H'} T_2$ induces a gluing of $N(\varLambda_{T_1})$ and $N(\varLambda_{T_2})$ that respects the handle decompositions and gives $(N(\varLambda_T),h := h_1 \# h_2)$.

		Let $\mathcal G_{\overline H} := CE^\ast((T^\ast \overline H, h_{\overline H}); X)$ (see \cref{dfn:ce_hypersurface}). We let $\mathcal G_{T\to \overline H}$ ($\mathcal G_{T\to \overline H}$) denote the right (left) module over the free algebra on $\mathcal G_{\overline H}$. Then define 
		\[
			\mathcal G_{T \to \overline H \to T} := \mathcal G_{T \to \overline H} \otimes_{\mathcal G_{\overline H}} \mathcal G_{\overline H \to T},
		\]
		as modules.

		\begin{lma}
			Let $T \subset \R^2 \times [0,1]$ be a tangle and let $h$ be a handle decomposition of $N(\varLambda_T)$.
			\begin{enumerate}
				\item Elements of $\mathcal G_{T \to T}$ are in one-to-one correspondence with oriented binormal geodesic chords of $T \subset \R^2 \times [0,1]$.
				\item Elements of $\mathcal G_{T \to \overline H \to T}$ are in one-to-one correspondence with oriented binormal geodesic chords of $\partial T \subset \overline H$ and words of oriented binormal geodesic chords in $\R^2 \times [0,1]$ of the form $T \to \overline H \to \cdots \to \overline H \to T$.
			\end{enumerate}
		\end{lma}
		\begin{proof}
			This is the same proof as \cref{cor:top_desc_of_generators_tch}. The fact that $\overline H$ is disconnected in this setup does not change the proof.
		\end{proof}
		\begin{thm}
			Let $T \subset \R^2 \times [0,1]$ be an oriented tangle. Let $h$ be a handle decomposition of $N(\varLambda_T)$. For generic choices of metrics and almost complex structures we have a quasi-isomorphism of dg-algebras
			\[
				KCC^\ast(T,h) \cong \left\langle \mathcal G_{T \to T}\right\rangle \ast \left\langle \mathcal G_{T \to \overline H\to T}\right\rangle.
			\]
		\end{thm}
		\begin{proof}
			This is the same proof as \cref{thm:top_desc_tch}. The fact that $\overline H$ is disconnected in this setup does not change the proof.
		\end{proof}
		\begin{lma}\label{lma:one-to-one_corr_gen_gluing_tch}
			Let $T_1 \subset \R^2 \times [0,1]$ and $T_2 \subset \R^3_{x\geq 1}$ be two gluing compatible oriented tangles and let $T := T_1 \#_{H'} T_2$. Then there are bijections of sets
			\[
				\mathcal G_{T\to T} \cong \mathcal G_{T_1 \to T_1} \cup \mathcal G_{T_2 \to T_2}, \qquad \mathcal G_{T \to H \to T} \cong \bigcup_{i \neq j \in \left\{1,2\right\}} \mathcal G_{T_i \to \overline H \to T_j}.
			\]
		\end{lma}
		\begin{proof}
			This is similar to the proof of \cref{lma:kcc_tangle_with_hypersurface}. Namely, it traces back to \cref{lma:surgery_desc_generators_relative}. The convex completion of $\R^3_{x\geq 0}$ is the Weinstein pair $(X,P)$. Splitting $\R^3_{x\geq 0}$ by the hypersurface $H' = \{x=1\}$ and then taking the convex completion yields the Weinstein pair $(X,P\sqcup P)$. Now, we need to compare the generators of the Chekanov--Eliashberg dg-algebra (with action bound, as in \cref{lma:surgery_desc_generators_relative}) before and after critical handle attachment.
		\end{proof}
		\begin{thm}\label{thm:gluing_tch}
			Let $T_1 \subset \R^2 \times [0,1]$ and $T_2 \subset \R^3_{x\geq 1}$ be two gluing compatible oriented tangles together with gluing compatible handle decompositions $h_1$ and $h_2$ of $N(\varLambda_{T_1})$ and $N(\varLambda_{T_1})$, respectively. Then for generic choices of metrics and almost complex structures we have a quasi-isomorphism of dg-algebras
			\[
				KCC^\ast(T,h_1\#h_2) \cong KCC^\ast(T_1,h_1) \ast_{KCC^\ast(\partial T_2,h_{\partial T_2})} KCC^\ast(T_2,h_2) \ast \left\langle \mathcal G_{T_1 \to \overline H \to T_2}\right\rangle \ast \left\langle \mathcal G_{T_2 \to \overline H \to T_1}\right\rangle.
			\]
		\end{thm}
		\begin{proof}
			By \cref{lma:one-to-one_corr_gen_gluing_tch} it follows that there is a one-to-one correspondence between generators. The rest of the proof is exactly the same as the proof of \cref{thm:gluing_free_product}.
		\end{proof}
		We say that $T_0 \subset \R^3_{x\geq 0}$, $T_1 \subset \R^2 \subset [0,1]$ and $T_2 \subset \R^3_{x\geq 1}$ are gluing compatible if $T_1$ and $T_2$ are gluing compatible, and if $T_0$ and $T := T_1 \#_{H'} T_2$ are gluing compatible.
		\begin{cor}\label{cor:total_gluing_kch}
			Let $T_0 \subset \R^3_{x\geq 0}$, $T_1 \subset \R^2 \subset [0,1]$ and $T_2 \subset \R^3_{x\geq 1}$ be gluing compatible oriented tangle and let $K := T_0 \#_H (T_1 \#_{H'} T_2)$. Let $h_0$, $h_1$ and $h_2$ be gluing compatible handle decompositions of $N(\varLambda_{T_0})$, $N(\varLambda_{T_1})$ and $N(\varLambda_{T_2})$ and let $h := h_0 \# h_1 \# h_2$. Then for generic choices of metrics and almost complex structures we have a quasi-isomorphism of dg-algebras
			\begin{align*}
				KCC^\ast(K, h) \cong &KCC^\ast(T_0,h_0) \ast_{KCC^\ast(\partial T_0,h_{\partial T_0})} KCC^\ast(T_1,h_1) \ast_{KCC^\ast(\partial T_2, h_{\partial T_2})} KCC^\ast(T_2,h_2) \\
				& \qquad \ast \bigast_{i \neq j \in \left\{0,1,2\right\}} \left\langle \mathcal G_{T_i \to \overline H \to T_j}\right\rangle.
			\end{align*}
		\end{cor}
		\begin{proof}
			Combining \cref{thm:gluing_free_product} with \cref{thm:gluing_tch} yields a quasi-isomorphism
			\begin{align*}
				KCC^\ast(K,h) \cong &KCC^\ast(T_0,h_0) \ast_{KCC^\ast(\partial T_0, h_{\partial T_0})} KCC^\ast (T_1,h_1) \ast_{KCC^\ast(\partial T_2, h_{\partial T_2})} KCC^\ast(T_2,h_2)\\
				& \qquad \ast \left\langle \mathcal G_{T_1 \to \overline H \to T_2}\right\rangle \ast \left\langle \mathcal G_{T_2 \to \overline H \to T_1}\right\rangle \ast \left\langle \mathcal G_{T_0 \to H \to T}\right\rangle \ast \left\langle \mathcal G_{T \to H \to T_0}\right\rangle,
			\end{align*}
			where $T := T_1 \#_{H'} T_2$. Next, we see that there are bijections
			\[
				\mathcal G_{T_0 \to H \to T} \cong \mathcal G_{T_0 \to \overline H \to T_1} \cup \mathcal G_{T_0 \to \overline H \to T_2}, \qquad \mathcal G_{T \to H \to T_0} \cong \mathcal G_{T_1 \to \overline H \to T_0} \cup \mathcal G_{T_2 \to \overline H \to T_0},
			\]
			by the same argument as in the proof of \cref{lma:one-to-one_corr_gen_gluing_tch}. Namely, compare Reeb chords before and after splitting by $H'$. The result follows.
		\end{proof}
	\section{Untangle detection}\label{sec:untangle_detection}
		In this section we prove that $KCC^\ast(T,h)$ for $h = h_1$ or $h = h_2$ as defined in \cref{sec:explicit_handle_decomp} detects the triviality of $T$. It relies on the result that knot contact homology detects the $r$-component unlink.
		\begin{lma}\label{lma:kch_detects_unlink}
			$KCH^\ast_{\mathrm{EENS}}(K)$ detects the $r$-component unlink up to ambient isotopy and mirroring for any $r \geq 1$.
		\end{lma}
		\begin{proof}
			This is already known in the case $r = 1$ by \cite[Corollary 1.5]{cieliebak2017knot}. The same argument as in \cite[Proposition 2.21]{cieliebak2017knot} and \cite[Proof of Corollary 1.5]{cieliebak2017knot} (cf.~\cite[Remark 5]{gordon2017knot}) can be used for $r > 1$, which we employ below.

			Let $\mathfrak R$ denote the subring of $\boldsymbol k[\pi_1(\R^3 \setminus K)]$ generated by $\pi_1(\lambda_1^{\pm 1}, \mu_1^{\pm 1}, \ldots, \lambda_r^{\pm 1}, \mu_r^{\pm 1})$, and $\im(1- \mu_i)$ for $i\in \left\{1,\ldots,r\right\}$ where $1 - \mu_i$ denotes the map $\boldsymbol k[\pi_1(\R^3 \setminus K)] \longrightarrow \boldsymbol k[\pi_1(\R^3 \setminus K)]$ given by left multiplication by $1- \mu_i$. By \cite[Theorem 1.2 and Proposition 2.21]{cieliebak2017knot} we have that
			\begin{equation}\label{eq:main_result_celn_links}
				KCH^0_{\text{EENS}}(K) \cong \mathfrak R
			\end{equation}
			for framed oriented non-trivial knots in $\R^3$, i.e.\@ for $r = 1$. The outline of the proof is that $KCH^\ast_{\text{EENS}}(K)$ is first related to string homology in degree zero as defined in \cite[Section 2]{cieliebak2017knot}, and the string homology in degree zero is related to the cord algebra of $K$ which is known to be isomorphic to $\mathfrak R$, see \cite[Section 2.4]{cieliebak2017knot}. The key point is that both of these results remain true for $r > 1$ by the same proofs, which means that \eqref{eq:main_result_celn_links} still holds for $r > 1$.

			Next, by \cite[Lemma 2]{howie1985the} we have that $\pi_1(\R^3 \setminus K)$ is locally indicable and hence $\boldsymbol k[\pi_1(\R^3 \setminus K)]$ is an integral domain by \cite[Theorem 12]{higman1940the}. By computation we have that
			\[
				KCH^0_{\text{EENS}}(U) \cong \boldsymbol k[\pi_1(\lambda_1^{\pm 1}, \mu_1^{\pm 1}, \ldots, \lambda_r^{\pm 1}, \mu_r^{\pm 1})]/\left\langle((1- \lambda_i)(1- \mu_i) \suchthat i\in \left\{1,\ldots,r\right\}\right\rangle
			\]
			(see \cite{ekholm2013knot}) which is not an integral domain, and thus $KCH^\ast_{\text{EENS}}(K) \not\cong KCH^\ast_{\text{EENS}}(U)$.
		\end{proof}
		\begin{lma}\label{lma:unique_gluing_untangle}
				Suppose that $T_1$ and $T_2$ are two $r$-component tangles with $\partial T_1 = \partial T_2$ such that $T_2$ is trivial and $K := T_1 \# T_2$ is an $r$-component link. Then $K$ is the unlink if and only if $T_1$ is trivial.
			\end{lma}
			\begin{proof}
				Write $T_1 = \bigcup_{i=1}^r T_1^i$ and pick a Seifert surface $\varSigma_i$ of each $T_1^i$ with boundary $\partial \varSigma_i = \partial_T \varSigma_i \cup \partial_0 \varSigma_i$ where $\partial_T \varSigma_i = \varSigma_i \cap T_1^i$, $\partial_0 \varSigma_i = \varSigma_i \cap \partial \R^3_{x \geq 0}$ is an arc and $\partial_T \varSigma_i \cap \partial_0 \varSigma_i = \partial T_1^i$. Similarly, write $T_2 = \bigcup_{i=1}^r T_2^i$. By assumption $T_2$ is trivial so we may pick $r$ disjoint half disks $\varDelta_1,\ldots,\varDelta_r$ where each $\varDelta_i$ is a Seifert surface of $T_2^i$.

				Without loss of generality we may assume that $\partial \varDelta_i = \partial \varSigma_i$. Furthermore by assumption we have that $K$ is an $r$-component link, which means that we can construct an $r$-component Seifert surface of $K$ whose $i$-th component is the result of gluing together $\varSigma_i$ and $\varDelta_i$ along $\partial \R^3_{x\geq 0}$. We denote this component by $\varSigma_i \# \varDelta_i$. Thus, we have that $K$ is the $r$-component unlink if and only if $\bigcup_{i=1}^r (\varSigma_i \# \varDelta_i)$ is the union of $r$ disjoint embedded disks. It is clear that triviality of $T_1$ implies that $K$ is the unlink. For the converse, if $K$ is the unlink we can pick $r$ disjoint embedded circles, and removing $T_2$ from $K$ leaves $T_1$ together with $r$ disjoint embedded half-disks which means that $T_1$ is trivial.
			\end{proof}
			\begin{rmk}
				The assumption \cref{lma:unique_gluing_untangle} that $T_1 \# T_2$ is a link consisting of the same number of components as $T_1$ and $T_2$ is crucial, see \cref{fig:ambient_iso_after_gluing}.
				\begin{figure}[!htb]
					\centering
					\includegraphics{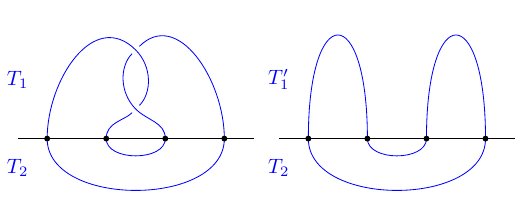}
					\caption{Two non-ambient isotopic tangles $T_1 \not\simeq T_1'$ such that $T_1 \# T_2 \simeq T_1' \# T_2$.}\label{fig:ambient_iso_after_gluing}
				\end{figure}
			\end{rmk}
			\begin{dfn}[Unknotted tangle]\label{dfn:unknotted_tangle}
				Let $T \subset \R^3_{x \geq 0}$ be an oriented tangle and let $\eta(x,y,z) := x$ denote the height function on $\R^3_{x\geq 0}$. We say that $T$ is \emph{unknotted} if there exists some tangle $T' \subset \R^3_{x\geq 0}$ such that $T \simeq T'$ and such that $\eta|_{T'}$ is a Morse function with no local minima away from $\partial \R^3_{x \geq 0}$.
			\end{dfn}
			\begin{rmk}
				Recall from \cref{dfn:trivial_tangle} that an unknotted tangle can still be non-trivial.
			\end{rmk}
			\begin{dfn}[Tangles of the same type]
				Let $T,T' \subset \R^3_{x\geq 0}$ be two oriented tangles. We say that $T$ and $T'$ are of the \emph{same type} if there exists a bijection $\varphi$ between the set of strands of $T$ and the set of strands of $T'$ such that each strand $T_i$ of $T$ has the same boundary as the strand $\varphi(T_i)$ of $T'$.
			\end{dfn}
			Note that given a tangle $T \subset \R^3_{x\geq 0}$ there is a unique trivial tangle of the same type as $T$. 

			We first show that tangle contact homology detects the triviality of $T$, assuming that $T$ is unknotted in \cref{sec:unknotted_tangles}. Afterwards in \cref{sec:general_case} we drop the unknottedness assumption.
		\subsection{Unknotted case}\label{sec:unknotted_tangles}
			\begin{lma}\label{lma:weak_gluing}
				Let $T, T' \subset \R^3_{x \geq 0}$ be two oriented tangles. Assume that $\mathcal G_{T \to H} \cong \mathcal G_{T' \to H}$ and $\mathcal G_{H\to T} \cong \mathcal G_{H\to T'}$ as right and left modules over $\mathcal G_H$, respectively. Then if there is a quasi-isomorphism of dg-algebras $KCC^\ast(T,h) \cong KCC^\ast(T',h')$ we have a quasi-isomorphism
				\[
					KCC^\ast(T \# T'', h \# h'') \cong KCC^\ast (T' \# T'', h' \# h'').
				\]
			\end{lma}
			\begin{proof}
				This follows from the definition of $\mathcal G_{T_i\to H\to T_j}$ in \eqref{eq:tensor_module_2}. Indeed, if $\mathcal G_{T \to H} \cong \mathcal G_{T' \to H}$ and $\mathcal G_{H\to T} \cong \mathcal G_{H\to T'}$ as right and left modules over $\mathcal G_H$, respectively, then it follows that 
				\[
					\mathcal G_{T\to H \to T''} \cong \mathcal G_{T'\to H \to T''}, \qquad \mathcal G_{T''\to H \to T} \cong \mathcal G_{T''\to H \to T'}
				\]
				from which the conclusion holds by \cref{thm:gluing_free_product}.
			\end{proof}
			\begin{lma}\label{lma:unknotted_same_geodesics}
				Let $T \subset \R^3_{x\geq 0}$ be an oriented and unknotted tangle and let $U \subset \R^3_{x\geq 0}$ be the trivial tangle of the same type as $T$. Then $\mathcal G_{T \to H} \cong \mathcal G_{U \to H}$ and $\mathcal G_{H\to T} \cong \mathcal G_{H\to U}$ as right and left modules over $\mathcal G_H$.
			\end{lma}
			\begin{proof}
				By an ambient isotopy, we can arrange so that each strand of $T$ lies very close to a strand of $U$, and in particular so that the set of oriented binormal geodesic chords between $T$ and $H$ are equal to the set of such between $U$ and $H$.
			\end{proof}
			\begin{lma}\label{thm:kcc_detects_untangle}
				Let $T \subset \R^3_{x\geq 0}$ be an oriented and unknotted tangle and let $U \subset \R^3_{x\geq 0}$ be the trivial tangle of the same type as $T$. Let $h$ be one of the handle decompositions $h_1$ and $h_2$ of $N(\varLambda_T)$ from \cref{sec:explicit_handle_decomp}. Then if there is a quasi-isomorphism $KCC^\ast(T,h) \cong KCC^\ast(U,h)$ then $T \simeq U$.
			\end{lma}
			\begin{proof}
				We will prove the contrapositive. Namely assume that $T \not\simeq U$. We can always find a splitting of the $r$-component unlink into two trivial $r$-component tangles $K = T_0 \# U$. Let $K' := T_0 \# T$. By \cref{lma:unique_gluing_untangle} it follows that $K \not \simeq K'$. Furthermore, because both $T$ and $U$ are unknotted tangles, we have $\mathcal G_{T \to H} \cong \mathcal G_{U \to H}$ and $\mathcal G_{H\to T} \cong \mathcal G_{H\to U}$ as right and left modules over $\mathcal G_H$ by \cref{lma:unknotted_same_geodesics}. Thus it follows by \cref{lma:kch_detects_unlink} that $KCC^\ast_{\text{EENS}}(K) \not \cong KCC^\ast_{\text{EENS}}(K')$ after which \cref{lma:weak_gluing} yields $KCC^\ast(T,h) \not\cong KCC^\ast(U,h)$.
			\end{proof}
			\subsection{General case}\label{sec:general_case}
				In this section we drop the assumption that $T \subset \R^3_{x \geq 0}$ is unknotted. Let $\mathring T := T \smallsetminus \partial T$ and $\mathring T = \bigcup_{i=1}^r \mathring T_i$. Let $\eta(x,y,z) := x$ be the height function defined on $\R^3_{x\geq 0}$. In general $\eta|_{\mathring T_i}$ has $b_i$ local minima and $b_i+1$ local maxima, and we may without loss of generality assume that the all local maximum values of $\eta|_{\mathring T}$ have value greater than $1$ and that all local minimum values of $h_{\mathring T}$ belong to $(0,1)$. 

				Then split $\R^3_{x\geq 0}$ along $H' = \{x=1\} \subset \R^3_{x\geq 0}$ to obtain a decomposition $T = T_1 \#_{H'} T_2$ where $T_2 \subset \R^3_{x\geq 1}$, $T_1 \subset \R^2 \times [0,1]$. Since $H' \subset \R^3_{x\geq 0}$ separates all local maximum values and local minimum values of $h$, it follows that $T_1$ and $T_2$ are unknotted in the sense of \cref{dfn:unknotted_tangle}.
				\begin{thm}\label{thm:untangle_detection}
					Let $T \subset \R^3_{x \geq 0}$ be an oriented tangle and let $U \subset \R^3_{x\geq 0}$ be the trivial tangle of the same type of $T$. Let $h$ be one of the handle decompositions $h_1$ and $h_2$ of $N(\varLambda_T)$ from \cref{sec:explicit_handle_decomp}. Then if there is a quasi-isomorphism $KCC^\ast(T,h) \cong KCC^\ast(U,h)$ then $T \simeq U$.
				\end{thm}
				\begin{proof}
					Write $T$ as the result of gluing the two unknotted tangles $T_1 \subset \R^2 \times [0,1]$ and $T_2 \subset \R^3_{x\geq 0}$, as described above. Letting $U_i$ be the trivial tangle in $\R^2 \times [0,1]$ and $\R^3_{x\geq 0}$, respectively, that is of the same type as $T_i$. Then $U = U_1 \#_{H'} U_2$. Now we proceed as in \cref{thm:kcc_detects_untangle}.

					Namely, pick a splitting of the $r$-component unlink into two trivial $r$-component tangles $K = T_0 \# U$ and let $K' := T_0 \# T = T_0 \# (T_1 \#_{H'} T_2)$. Assuming $T \not \simeq U$ it follows by \cref{thm:kcc_detects_untangle} that $K \not \simeq K'$ and hence $KCC^\ast_{\mathrm{EENS}}(K) \not \cong KCC^\ast_{\mathrm{EENS}}(K')$.

					In the decompositions $K = T_0 \# T_1 \#_{H'} T_2$ and $K' = T_0 \# U_1 \#_{H'} U_2$, each of $T_i$ and $U_i$ is an unknotted tangle, and hence by an ambient isotopy we can arrange so that each strand of $T_i$ lies very close to a strand of $U_i$ so that the set of oriented binormal geodesic chords between $T_i$ and $\overline H$ is equal to the set of such between $U_i$ and $\overline H$. The conclusion from this is that we have an isomorphism
					\begin{align*}
						\mathcal G_{T_i \to \overline H \to T_j} \cong \mathcal G_{U_i \to \overline H \to U_j},
					\end{align*}
					as modules over $\mathcal G_{\overline H}$ for $i\neq j \in \left\{0,1,2\right\}$ (where $U_0 := T_0$). Combining \cref{cor:total_gluing_kch} with \cref{thm:gluing_tch} we thus have
					\begin{align*}
						&KCC^\ast(T_0,h_0) \ast_{KCC^\ast(\partial T_0, h_{\partial T_0})} KCC^\ast(T,h) \ast \bigast_{\substack{i \neq j \in \left\{0,1,2\right\} \\ (i,j) \neq (1,2),(2,1)}} \left\langle \mathcal G_{T_i \to \overline H \to T_j}\right\rangle \\
						&\not\cong KCC^\ast(T_0,h_0) \ast_{KCC^\ast(\partial T_0, h_{\partial T_0})} KCC^\ast(U,h) \ast \bigast_{\substack{i \neq j \in \left\{0,1,2\right\} \\ (i,j) \neq (1,2),(2,1)}} \left\langle \mathcal G_{U_i \to \overline H \to U_j}\right\rangle,
					\end{align*}
					from which we obtain $KCC^\ast(T,h) \not\cong KCC^\ast(U,h)$.
				\end{proof}
			We round off the section by showing that tangle contact homology does not stay invariant if we allow the boundary of the tangles to move around in $\partial \R^3_{x\geq 0}$, as a consequence of the triviality detection.
			\begin{dfn}[Ambient isotopy with moving boundary]
				We say that two tangles $T_1$ and $T_2$ in $\R^3_{x\geq 0}$ are \emph{ambient isotopic with moving boundary} if there is an isotopy $h_t \colon \R^3_{x \geq 0} \longrightarrow \R^3_{x \geq 0}$ such that 
				\begin{enumerate}
					\item $h_0 = \id$, and
					\item $h_1|_{\mathring T_1} = \mathring T_2$ where $\mathring T_i := T_i \smallsetminus \partial T_i$.
				\end{enumerate}
			\end{dfn}
			\begin{cor}
				The quasi-isomorphism class of $KCC^\ast(T,h)$ is not invariant under ambient isotopies with moving boundary of $T$.
			\end{cor}
			\begin{proof}
				We can always find an ambient isotopy with moving boundary of $T$ with the trivial tangle of the same type as $T$, and \cref{thm:untangle_detection} shows that $KCC^\ast(T,h)$ does not stay invariant.
			\end{proof}

\section{Examples}\label{sec:ex_unknot}
	\begin{ex}[Unframed subalgebra]\label{ex:unframed_subalgebra}
		Let us first consider a single stranded tangle $T \subset \R^3_{x\geq 0}$. Its boundary is two points in $\partial \R^3_{x\geq 0} \cong \R^2$. The unframed knot contact homology $KCC^\ast_u(\partial T)$ is simply the knot contact homology (without homology coefficients) of two points $z_1,z_2 \in \R^2$. As $KCC^\ast_u(\partial T)$ is generated by oriented binormal geodesics, we know that there are two generators $c_1$ and $c_2$, both in degree $0$. Their differentials are also trivial. Hence $KCC^\ast_u(\partial T) \cong \boldsymbol k[c_1,c_2]$. Computing the degree and differential is done using the front projection of $\varLambda_{z_1}, \varLambda_{z_2} \subset ST^\ast \R^2 \cong J^1(S^1)$, which lie in $S^1 \times \R \cong \R^2 \setminus \left\{0\right\}$, see \cref{fig:unfarmed_subalgebra}.
		\begin{figure}[!htb]
			\centering
			\includegraphics{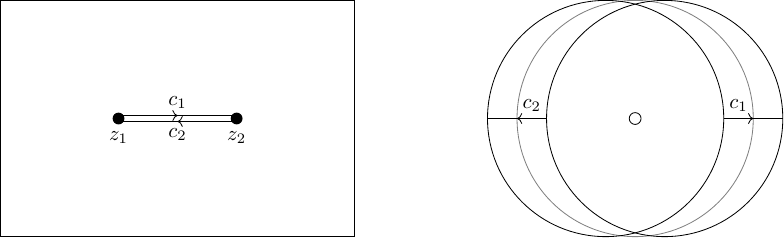}
			\caption{Left: The two oriented binormal geodesics of $\partial T \subset \R^2$. Right: The two front projections of $\varLambda_{z_1}$ and $\varLambda_{z_2}$ in $S^1 \times \R \cong \R^2 \setminus \left\{0\right\}$ with the two Reeb chords $c_1$ and $c_2$ drawn. The front projection of the zero section of $J^1(S^1)$ is the central gray circle.}\label{fig:unfarmed_subalgebra}
		\end{figure}

		We now give a surgery presentation of $KCC^\ast_u(\partial T)$. Namely, present $T^\ast \R^2$ as $B^4$ stopped at the standard unknot $\varLambda_{\text{unknot}} \subset S^3 = \partial B^4$, see also \cite[Example 7.4]{asplund2020chekanov}. This means that we present $T^\ast \R^2$ as the result of attaching a copy of $T^\ast (\varLambda_{\text{unknot}} \times \R_{\geq 0})$ to $B^4$ along a cotangent neighborhood of $\varLambda_{\text{unknot}}$, see \cref{fig:cotangent_plane_stopped}.
		\begin{figure}[!htb]
			\centering
			\includegraphics{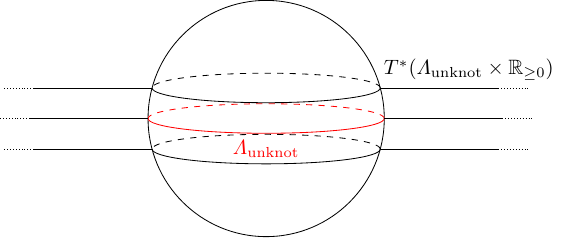}
			\caption{Surgery presentation of the open Weinstein sector $T^\ast \R^2$ as $B^4$ stopped at the standard unknot in $S^3$.}\label{fig:cotangent_plane_stopped}
		\end{figure}
		Now, we can use surgery techniques to give a quasi-isomorphic presentation of $KCC^\ast_u(\partial T)$ which is more amenable to proving the gluing formulas in this paper. Namely, removing the cotangent bundle $T^\ast(\varLambda_{\text{unknot}} \times \R_{\geq 0})$, we are left with the Legendrians $\varLambda_{z_1}, \varLambda_{z_2}, \varLambda_{\text{unknot}} \subset S^3 = \partial B^4$. By the results \cite[Theorem 5.10]{bourgeois2012effect} and \cite[Theorem 1.2]{ekholm2019holomorphic} it follows that $KCC^\ast_u(\partial T)$ is quasi-isomorphic to the dg-algebra generated by composable words of Reeb chords of the form
		\begin{equation}\label{eq:words_generators}
			\varLambda_{z_i} \to \overline{\varLambda}_{\text{unknot}} \to \cdots \to \overline{\varLambda}_{\text{unknot}} \to \varLambda_{z_j},
		\end{equation}
		for $i,j \in \left\{1,2\right\}$ where $\overline{\varLambda}_{\text{unknot}}$ is a Legendrian submanifold in the positive boundary of the Weinstein cobordism obtained by attaching a handle $(B^2 \times D_\varepsilon T^\ast [0,1], \frac 12(qdp-pdq) + 2xdy+ydx)$ to $B^4 \sqcup ((-\infty,0] \times B^2 \times \R)$, see \cite[Section 1.2]{asplund2020chekanov} for a general description. The Legendrian $\overline{\varLambda}_{\text{unknot}}$ is now defined in the same way as $\varSigma(h)$ in \cite[Equation (1.1)]{asplund2020chekanov}. The geometric meaning is that we consider loop space coefficients in the component $\varLambda_{\text{unknot}}$ only. The differential used in the dg-algebra generated by words of the form as in \eqref{eq:words_generators} is induced by the differential of the Chekanov--Eliashberg dg-algebra of the Legendrian link $\varLambda_{z_1} \cup \varLambda_{z_2} \cup \overline{\varLambda}_{\text{unknot}} \subset \partial B^4_{T^\ast D^2}$. After a Legendrian isotopy this link fits in a Darboux chart and the front projection in $\R^2$ is shown in \cref{fig:front_boundary_link}.
		\begin{figure}[!htb]
			\centering
			\includegraphics{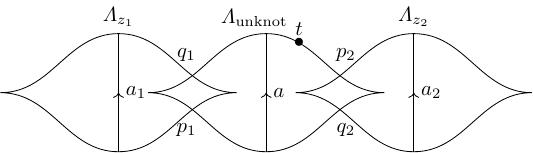}
			\caption{The front projection of the Legendrian link $\varLambda_{z_1} \cup \varLambda_{z_2} \cup \varLambda_{\text{unknot}}$ in a Darboux chart.}\label{fig:front_boundary_link}
		\end{figure}
		The Legendrian $\overline{\varLambda}_{\text{unknot}}$ consists of a $1$-dimensional disk $\ell_{\text{unknot}}$ with boundary $\partial \ell_{\text{unknot}} \subset B^2$ where $B^2$ is the subcritical part after a choice of handle decomposition of a Weinstein neighborhood of $\varLambda_{\text{unknot}}$. It follows from \cite[Lemma 3.4]{asplund2020chekanov} that Reeb chords of $\overline{\varLambda}_{\text{unknot}}$ corresponds to Reeb chords of $\ell_{\text{unknot}}$ that are contained in $S^3$ (this is the Reeb chord labeled by $a$ in \cref{fig:front_boundary_link}) and Reeb chords of $\partial \ell_{\text{unknot}} \subset B^2$. The Reeb chords of the latter kind corresponds to Reeb chords of two points in the boundary of $B^2$, and they form a dg-subalgebra described in detail in \cite[Section 2.3]{ekholm2015legendrian}, see also \cite[Example 7.3]{asplund2020chekanov} and references therein. With appropriate choices of Maslov potential, the two shortest Reeb chords corresponding to the point $t$ in \cref{fig:front_boundary_link} are denoted by $t^0_{12}$ and $t^1_{21}$ using the notation as in \cite[Example 7.3]{asplund2020chekanov}.

		Inside this Darboux chart there are two Reeb chords between $\varLambda_{z_1}$ and $\varLambda_{\text{unknot}}$ and also two Reeb chords between $\varLambda_{z_2}$ and $\varLambda_{\text{unknot}}$, labeled in \cref{fig:front_boundary_link} by $q_1,p_1$ and $q_2,p_2$, respectively. Note that since this Legendrian link is in $S^3$, there are in general more Reeb chords that necessarily leaves the Darboux chart. Using \cite[Lemma 4.6]{ekholm2015legendrian} and an action argument like the proof of \cite[Lemma 5.6]{ekholm2015legendrian} we obtain that $KCC^\ast_u(\partial T)$ is quasi-isomorphic to the dg-algebra generated by composable words formed by products of the following two kinds:
		\begin{enumerate}
			\item $a_1$ and $a_2$ both in degree $1$.
			\item The words $q_j c^k p_i$ for $k\geq 0$, $i,j\in \left\{1,2\right\}$, where $c$ is any Reeb chord $\overline{\varLambda}_{\text{unknot}}$, of degree $|c|$.
		\end{enumerate}
		The differential is induced by the differential of the Chekanov--Eliashberg dg-algebra of $\varLambda_{z_1} \cup \varLambda_{z_2} \cup \overline{\varLambda}_{\text{unknot}}$ which can be computed from \cref{fig:front_boundary_link} to be as follows with coefficients in $\Z_2$
		\[
			\begin{cases}
				\partial a_1 = p_1q_1 \\
				\partial a_2 = p_2q_2 \\
				\partial a = e_{\text{unknot}}+t^0_{12}+q_1p_1+t^0_{12}q_2p_2 \\
				\partial q_i = \partial p_i = 0
			\end{cases},
		\]
		where $e_{\text{unknot}}$ denotes the idempotent corresponding to the component $\overline{\varLambda}_{\text{unknot}}$. By the invariance of the Chekanov--Eliashberg dg-algebra and the fact that the dg-algebra can be computed locally in a Darboux chart (see \cite[Section 5.2.E]{ekholm2015legendrian}), the just described dg-algebra is quasi-isomorphic to $KCC^\ast_u(\partial T)$.
	\end{ex}
	
	\begin{ex}[Framed subalgebra]\label{ex:framed_subalgebra}
		We upgrade \cref{ex:unframed_subalgebra} and now compute $KCC^\ast(\partial T, h_{\partial})$. The difference is now that we pick a handle decomposition $h_{\partial}$ of a chosen Weinstein neighborhood $N(\varLambda_{\partial T}) = N(\varLambda_{z_1}) \cup N(\varLambda_{z_2})$. We construct $\overline{\varLambda}_{z_i}$ for $i\in \left\{1,2\right\}$ in the same way as $\overline{\varLambda}_{\text{unknot}}$ in \cref{ex:unframed_subalgebra}. Computationally, we have two more infinite families of Reeb chords to take into account, indicated by the two points $t(1)$ and $t(2)$, see \cref{fig:front_boundary_link_framed}.
		\begin{figure}[!htb]
			\centering
			\includegraphics{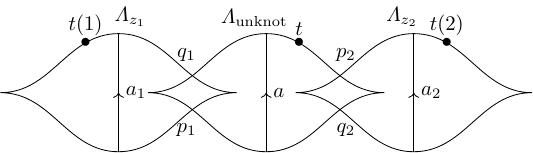}
			\caption{The front projection of the Legendrian link $\overline{\varLambda}_{z_1} \cup \overline{\varLambda}_{z_2} \cup \varLambda_{\text{unknot}}$ in a Darboux chart.}\label{fig:front_boundary_link_framed}
		\end{figure}
		By the general theory we now have that the dg-algebra $KCC^\ast(\partial T)$ is quasi-isomorphic to the dg-algebra generated by composable words formed by products of the following two kinds:
		\begin{enumerate}
			\item $a_1$ and $a_2$ both in degree $1$
			\item Reeb chords $t(1)^p_{ij}$ and $t(2)^p_{ij}$ with notation as in \cite[Example 7.3]{asplund2020chekanov}.
			\item The words $q_j c^k p_i$ for $k\geq 0$, $i,j\in \left\{1,2\right\}$, where $c$ is any Reeb chord $\overline{\varLambda}_{\text{unknot}}$, of degree $|c|$.
		\end{enumerate}
		The differential is induced by the differential of the Chekanov--Eliashberg dg-algebra of $\overline{\varLambda}_{z_1} \cup \overline{\varLambda}_{z_2} \cup \overline{\varLambda}_{\text{unknot}}$ which can be computed from \cref{fig:front_boundary_link_framed} to be as follows with coefficients in $\Z_2$
		\begin{equation}\label{eq:boundary_dga_framed_differential}
			\begin{cases}
				\partial a_1 = e_1+t(1)^0_{12}+p_1q_1 \\
				\partial a_2 = e_2+t(2)^0_{12}+p_2q_2 \\
				\partial a = e_{\text{unknot}}+t^0_{12}+q_1p_1+t^0_{12}q_2p_2 \\
				\partial q_i = \partial p_i = 0
			\end{cases},
		\end{equation}
		where $e_i$ and $e_{\text{unknot}}$ denotes the idempotent corresponding to the components $\overline{\varLambda}_{z_i}$ and $\overline{\varLambda}_{\text{unknot}}$, respectively for $i\in \left\{1,2\right\}$.

		Above and in \cref{ex:unframed_subalgebra} we have used a description of $KCC^\ast(\partial T, h_\partial)$ which only uses Reeb chords of the front projecting appearing in the Darboux chart. Instead, we consider $\overline{\varLambda}_{\text{unknot}}$ and $\overline{\varLambda_{z_i}}$ before applying the Legendrian isotopy to make them fit in a Darboux chart. Namely, the Reeb chords of the link $\overline{\varLambda}_{z_1} \cup \overline{\varLambda}_{z_2} \cup \overline{\varLambda}_{\text{unknot}}$ come in two families of $S^1$-families, one corresponding to Reeb chords of $\overline{\varLambda}_{z_1} \cup \overline{\varLambda}_{\text{unknot}}$ and one corresponding to Reeb chords of $\overline{\varLambda}_{z_2} \cup \overline{\varLambda}_{\text{unknot}}$. After Morsification, the shortest Reeb chords are as indicated in the following diagram.
		\begin{equation}\label{eq:boundary_dga_diagram}
			\begin{tikzcd}[row sep=scriptsize]
				\overline{\varLambda}_{z_1} \lar[loop left]{a_1} \rar[yshift=5]{q_1, \widehat{q}_1} & \overline{\varLambda}_{\text{unknot}} \lar[yshift=-5]{p_1,\widehat{p}_1} \rar[yshift=-5,swap]{p_2, \widehat p_2} \uar[loop above]{a} & \overline{\varLambda}_{z_2} \lar[yshift=5,swap]{q_2, \widehat q_2} \rar[loop right]{a_2}
			\end{tikzcd}
		\end{equation}
		The differential (with coefficients in $\Z_2$) is given by the following.
		\[
			\begin{cases}
				\partial q_1 = \partial q_2 = \partial p_1 = \partial p_2 = 0 \\
				\partial \widehat q_1 = q_1 + t^0_{12}q_1 \\
				\partial \widehat q_2 = q_2 + t^0_{12}q_2 \\
				\partial \widehat p_1 = p_1 + p_1t^0_{12} \\
				\partial \widehat p_2 = p_2 + p_2t^0_{12} \\
				\partial a = e_{\text{unknot}} + t^0_{12} + q_1p_1 + t^0_{12}q_2p_2 \\
				\partial a_1 = e_{1} + t(1)^{0}_{12} + p_1q_1 \\
				\partial a_2 = e_{2} + t(2)^0_{12} + p_2q_2
			\end{cases}.
		\]
	\end{ex}
	\begin{ex}[The dg-subalgebra when $H \cong S^2$]\label{ex:framed_subalgebra_sphere}
		Consider a single stranded tangle $T \subset B^3$ where $B^3$ denotes the closed $3$-ball. Its boundary is two points in $H := \partial B^3 = S^2$. The dg-subalgebras $KCC^\ast(\partial T)$ and $KCC^\ast_u(\partial T, h_\partial)$ in this case can be recovered from those in \cref{ex:unframed_subalgebra} and \cref{ex:framed_subalgebra}, respectively, by setting $t^p_{ij} = 1$. Geometrically, this is due to the fact that $T^\ast S^2$ is the result of attaching a standard Weinstein handle (as opposed to a stop) on the standard Legendrian unknot $\varLambda_{\text{unknot}} \subset S^3 = \partial B^4$. Since there are binormal geodesic chords of $S^2$ in the ball, there are additional Reeb chord generators of the Legendrian $\ell_H$ due to \cref{lma:tcc_generators_hypersurface}. We will not need to know the differential of these Reeb chords, but only those of $\partial \ell_H$. The differential in the case with the framed dg-subalgebra is induced by the following
		\[
			\begin{cases}
				\partial a_1 = e_1+t(1)^0_{12}+p_1q_1 \\
				\partial a_2 = e_2+t(2)^0_{12}+p_2q_2 \\
				\partial a = q_1p_1+q_2p_2 \\
				\partial q_i = \partial p_i = 0
			\end{cases}.
		\]
		This is the same as in \eqref{eq:boundary_dga_framed_differential} after setting $t^0_{12} = e$, where $e$ is the idempotent corresponding to the unknot component that is equal to $\partial \ell_H$.
	\end{ex}
	\begin{ex}[Unframed trivial tangle]\label{ex:unframed_trivial_tangle}
		Let us now consider $T \subset \R^3_{x \geq 0}$ the trivial tangle as depicted in \cref{fig:trivial_tangle_ex}. The hypersurface $H$ is the graph of the function $f(x,y) = x^2-y^2$ with a saddle point at the origin.
		\begin{figure}[!htb]
			\centering
			\includegraphics{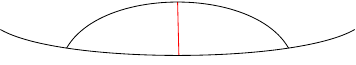}
			\hspace{2cm}
			\includegraphics[scale=0.75]{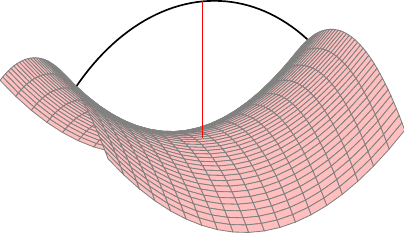}
			\caption{The trivial tangle in $\R^3_{x\geq 0}$. The vertical red line is the only unoriented binormal geodesic chord between $T$ and $H$.}\label{fig:trivial_tangle_ex}
		\end{figure}
		By \cref{lma:kcc_tangle_with_hypersurface} we know that $KCC^\ast_u(T)$ is quasi-isomorphic to the dg-algebra generated by oriented binormal geodesic chords between $T$ and $H$, since there are no binormal geodesic chords of $T$ and none of $H$. We also know that $KCC^\ast_u(\partial T) \subset KCC^\ast_u(T)$ is a dg-subalgebra that we have already described in \cref{ex:unframed_subalgebra}. The degree of both of the two Reeb chords $c_{T\to H}$ and $c_{H\to T}$ is $1$. The differential is given by
		\begin{equation}\label{eq:unframed_trivial_tangle}
			\begin{cases}
				\partial c_{T\to H} = q_1+q_2 \\
				\partial c_{H\to T} = p_1+p_2
			\end{cases}
		\end{equation}
		For each Reeb chord, the two $J$-holomorphic disks may be visualized in \cref{fig:trivial_tangle_ex} by imagining that the oriented binormal geodesic chord moving right or left and shrinks down to a point.

		We remark that replacing $\R^3_{x\geq 0}$ by $\overline{\R^3 \setminus B^3}$ gives the same result except that the dg-subalgebra is as described in \cref{ex:framed_subalgebra_sphere}. We imagine closing up the hypersurface $H$ to a large sphere inside $\R^3$.
	\end{ex}
	\begin{ex}[Framed trivial tangle]\label{ex:framed_trivial}
		Let us now consider $T \subset \R^3_{x \geq 0}$ the trivial tangle, see \cref{fig:trivial_tangle_ex}. We let $h_1$ and $h_2$ be the handle decomposition of $N(\varLambda_{T})$ as depicted in \cref{fig:handle_decomp1} and \cref{fig:handle_decomp2}, respectively. The subalgebra $KCC^\ast(\partial T, h_\partial)$ is as described in \cref{ex:framed_subalgebra}. 

		We first consider $KCC^\ast(T,h_2)$. There is one additional set of generators corresponding to $\eval[0]{\partial \ell_T}_{V_T}$, see \cref{fig:gen_with_hypersurface}. This set of generators is denoted by $\left\{\widehat t^p_{ij}\right\}$ and should be interpreted as follows. In $V_T$ there is a one-parameter family of dg-algebras all of which is quasi-isomorphic to the Chekanov--Eliashberg dg-algebra of two distinct points in $S^1$. After Morsification by a function with a maximum in the interior and two minima at the boundary of $\varLambda_T$. The generators $\left\{\widehat t^p_{ij}\right\}$ correspond to the maximum and $\left\{t(i)^p_{ij}\right\}$ correspond to the two minima; they are depicted in \cref{fig:generators_subalg_that}. Together $\left\{\widehat t^p_{ij}, t(1)^p_{ij}, t(2)^p_{ij}\right\}$ generate a dg-algebra that is quasi-isomorphic to $CE^\ast(\eval[0]{\partial \ell_T}_{V_T}, Q_T; V_T)$, see \cref{fig:front_boundary_link_framed} and \cref{fig:generators_subalg_that}. 
		\begin{figure}[!htb]
			\centering
			\includegraphics{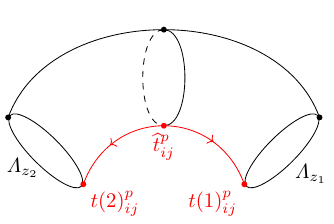}
			\caption{A sketch of the generators of $CE^\ast(\eval[0]{\partial \ell}_{V_T},Q_T;V_T)$ and their differential, using the handle decomposition $h_2$.}\label{fig:generators_subalg_that}
		\end{figure}
		The differential of $\widehat t^{p}_{ij}$ is given by
		\begin{equation}\label{eq:boundary_t_hat}
			\partial \widehat t^p_{ij} = t(1)^p_{ij}+t(2)^p_{ij} + G(\partial \widehat t^p_{ij}).
		\end{equation}
		where $G$ is defined as follows on monomials
		\begin{align*}
			G(\widehat t^{p_1}_{i_1j_1} \widehat t^{p_2}_{i_2j_2} \cdots \widehat t^{p_m}_{i_mj_m}) &= \widehat t^{p_1}_{i_1j_1} t(2)^{p_2}_{i_2j_2} \cdots t(2)^{p_m}_{i_mj_m} + t(1)^{p_1}_{i_1j_1} \widehat t^{p_2}_{i_2j_2} \cdots \widehat t(2)^{p_m}_{i_mj_m} \\
			&\quad + \cdots t(1)^{p_1}_{i_1j_1} t(1)^{p_2}_{i_2j_2} \cdots \widehat t^{p_m}_{i_mj_m},
		\end{align*}
		and extended to the whole algebra by linearity \cite[Corollary 5.6]{ekholm2008isotopies}.

		As in \cref{ex:unframed_trivial_tangle} we have two unoriented binormal geodesic chords $c_{T\to H}$ and $c_{H\to T}$ corresponding to generators of $KCC^\ast(T,h_2)$ of degree $1$. Their differential is given by the same formula as in the unframed case, namely as in \eqref{eq:unframed_trivial_tangle}.

		Similar to the above we now describe $KCC^\ast(T,h_1)$. The dg-algebra $CE^\ast(\eval[0]{\partial \ell}_{V_T},Q_T;V_T)$ is now more complicated and we do not fully describe it. We remark however that if setting $t(1)^p_{ij} = t(2)^p_{ij}$, it is quasi-isomorphic to the one exhibited in \cite[Section 2.5]{ekholm2015legendrian} which is the Chekanov--Eliashberg dg-algebra of the top attaching sphere of $T^\ast T^2$ equipped with its standard handle decomposition consisting of one $0$-handle, two $1$-handles and one $2$-handle.

		Furthermore, there is a certain dg-subalgebra of $CE^\ast(\eval[0]{\partial \ell}_{V_T},Q_T;V_T)$ that is generated by $\left\{c^p_{ij}\right\}$ which again is quasi-isomorphic to the Chekanov--Eliashberg dg-algebra of two points in $S^1$, see \cref{fig:generators_subalg_hat}
		\begin{figure}[!htb]
			\centering
			\includegraphics{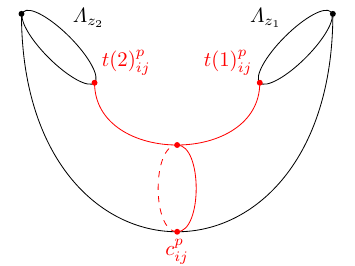}
			\caption{A sketch of some of the generators of $CE^\ast(\eval[0]{\partial \ell}_{V_T},Q_T;V_T)$ using the handle decomposition $h_1$.}\label{fig:generators_subalg_hat}
		\end{figure}
		As above we have two unoriented binormal geodesic chords $c_{T\to H}$ and $c_{H\to T}$ corresponding to generators of $KCC^\ast(T,h_1)$ of degree $1$. Their differential is given by the following.
		\[
			\begin{cases}
				\partial c_{T\to H} = q_1+q_2c^0_{12} \\
				\partial c_{H\to T} = p_1+c^1_{21}p_2
			\end{cases}
		\]
		We remark that replacing $\R^3_{x\geq 0}$ by $\overline{\R^3 \setminus B^3}$ gives the same result except that the dg-subalgebra is as described in \cref{ex:framed_subalgebra_sphere}. We imagine closing up the hypersurface $H$ to a large sphere inside $\R^3$.
	\end{ex}
	\begin{ex}[Unknot]
		Let $K \subset \R^3$ be the unknot. First split $K$ into the gluing with $H$ that is (an affine transformation of) the graph of $f(x,y) = x^2-y^2$ of two tangles as shown in \cref{fig:unknot_example} and \cref{fig:3d_tangle}. Finally, we will close up $H$ to a large sphere so that $T_2$ is the tangle lying in the region bounded by the sphere in $\R^3$. 
		\begin{figure}[!htb]
			\centering
			\includegraphics{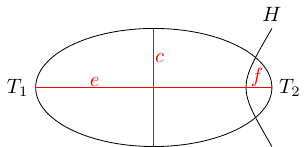}
			\caption{The unknot obtained by gluing of the two tangles $T_1$ and $T_2$.}\label{fig:unknot_example}
		\end{figure}
		\begin{figure}[!htb]
			\centering
			\includegraphics[scale=0.75]{tangle_ex_3d_1.pdf}
			\hspace{2cm}
			\includegraphics[scale=0.75]{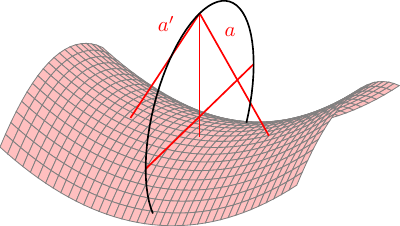}
			\caption{Left: A side view of the tangle $T_2$. Right: A side view of the tangle $T_1$ revealing two extra binormal geodesic chords $a$ and $a'$.}\label{fig:3d_tangle}
		\end{figure}

		Let $h_1$ be the handle decomposition of $N(\varLambda_{T_1})$ depicted in \cref{fig:handle_decomp2} and let $h_2$ be the handle decomposition of $N(\varLambda_{T_2})$ depicted in \cref{fig:handle_decomp1}. We already computed $KCC^\ast(T_2,h_2)$ in \cref{ex:framed_trivial} so we will focus on $KCC^\ast(T_1,h_1)$. From \cref{fig:unknot_example} we see that we have two unoriented binormal geodesic chords $c$ and $e$ corresponding to Reeb chords $c_1$, $c_2$, $e_{T_1\to H}$, $e_{H\to T_1}$, where $|c| = 1$ and $|e| = 2$. Furthermore we have that \cref{fig:3d_tangle} reveals two more unoriented binormal geodesic chords $a$ and $a'$ with $|a| = |a'| = 1$. The differentials of the external Reeb chords are given by (with coefficients in $\Z_2$) the following.
		\[
			\begin{cases}
				\partial a_{T_1 \to H} = q_1t(1)^0_{12}+q_2t(2)^0_{12} \\
				\partial a_{H\to T_1} = t(1)^0_{12}p_1+t(2)^0_{12}p_2 \\
				\partial a'_{T_1 \to H} = q_1t(1)^1_{21}+q_2t(2)^1_{21} \\
				\partial a'_{H\to T_1} = t(1)^1_{21}p_1+t(2)^1_{21}p_2 \\
				\partial c_1 = e_{1} + t(1)^1_{21} + p_1q_2 \\
				\partial c_2 = e_{1} + t(1)^1_{21} + p_2q_1 \\
				\partial e_{T_1 \to H} = q_1(c_1+a_1) + q_2(c_2+a_2+\widehat t^{1}_{21}) + a(q_1+q_2) + a_{T_1 \to H} + a'_{T_1 \to H} \\
				\partial e_{H\to T_1} = (c_2+a_1)p_1 + (c_1+a_2+\widehat t^{1}_{21})p_2 + (p_1+p_2)a + a_{H\to T_1} + a'_{H\to T_1}
			\end{cases}.
		\]
		The differentials of $c_1$ and $c_2$ closely resemble those in the knot contact homology of the unknot as described in \cite[Example 3.13]{ng2014a}. The differentials of $e_{T_1\to H}$ and $e_{H\to T_1}$ has many non-trivial contributions from holomorphic disks appearing as boundary degenerations which are introduced by the hypersurface $H$. There are also extra contributions involving the two unoriented binormal geodesic chords $a$ and $a'$ which are introduced by the hypersurface $H$.
	\end{ex}
\bibliographystyle{alpha}
\bibliography{tch}
\end{document}